\documentclass{article}
\usepackage[utf8]{inputenc}

\pdfoutput=1
\usepackage{amsmath}
\usepackage{amsfonts}
\usepackage{amssymb}
\usepackage{centernot}
\usepackage{graphicx}
\usepackage{amsthm}
\usepackage{enumerate}

\DeclareMathOperator{\coker}{coker }
\usepackage{tikz}
\usetikzlibrary{decorations.markings}
\usetikzlibrary{arrows}
\usetikzlibrary{calc}
\usetikzlibrary{matrix}
\providecommand{\U}[1]{\protect\rule{.1in}{.1in}}
\newtheorem{theorem}{Theorem}

\newtheorem{corollary}[theorem]{Corollary}

\newtheorem{definition}[theorem]{Definition}

\newtheorem{lemma}[theorem]{Lemma}

\newtheorem{proposition}[theorem]{Proposition}

\newcommand{\IMG}{\textup{IMG}}
\newcommand{\QIMG}{\textup{QIMG}}
\newcommand{\IMQ}{\textup{IMQ}}
\newcommand{\Core}{\textup{Core}}
\newcommand{\Aut}{\textup{Aut}}
\newcommand{\Dis}{\textup{Dis}}

\theoremstyle{remark}

\begin{document}

\title{Multivariate Alexander quandles, II. The involutory medial quandle of a link (corrected)}
\author{Lorenzo Traldi\\Lafayette College\\Easton, PA 18042, USA\\traldil@lafayette.edu
}
\date{ }
\maketitle

\begin{abstract}
Joyce showed that for a classical knot $K$, the involutory medial quandle $\textup{IMQ}(K)$ is isomorphic to the core quandle of the homology group $H_1(X_2)$, where $X_2$ is the cyclic double cover of $\mathbb S ^3$, branched over $K$. It follows that $|\textup{IMQ}(K)| = | \det K |$. In the present paper, the extension of Joyce's result to classical links is discussed. Among other things, we show that for a classical link $L$ of $\mu \geq 2$ components, the order of the involutory medial quandle is bounded as follows:
\[
\frac{\mu | \det L |}{2} \geq |\textup{IMQ}(L)| \geq \frac{ \mu | \det L |} {2^{\mu -1}}.
\]
In particular, $\textup{IMQ}(L)$ is infinite if and only if $\det L =0$. We also show that in general, $\textup{IMQ}(L)$ is a strictly stronger invariant than $H_1(X_2)$. That is, if $L$ and $L'$ are links with $\textup{IMQ}(L) \cong \textup{IMQ}(L')$, then $H_1(X_2) \cong H_1(X'_2)$; but it is possible to have $H_1(X_2) \cong H_1(X'_2)$ and $\textup{IMQ}(L) \not \cong \textup{IMQ}(L')$. In fact, it is possible to have $X_2 \cong X'_2$ and $\textup{IMQ}(L) \not \cong \textup{IMQ}(L')$.


\end{abstract}

\section{Introduction}

Let $\mu$ be a positive integer, and let $L=K_1 \cup \dots \cup K_{\mu}$ be a classical link of $\mu$ components. That is, $K_1, \dots, K_{\mu}$ are pairwise disjoint, piecewise smooth knots (closed curves) in $\mathbb{S}^3$. Almost forty years ago, Joyce \cite{J} and Matveev \cite{M} introduced a powerful invariant of oriented links, the fundamental quandle $Q(L)$. Both the theory of link quandles and the general theory of quandles have seen considerable development since then. 

In this paper we focus on involutory medial quandles. (Joyce \cite{J} used the term ``abelian'' rather than ``medial.'') Involutory medial quandles are much simpler than arbitrary quandles, and they provide invariants of unoriented links. They are defined as follows.

\begin{definition}
\label{qdef}
An \emph{involutory medial quandle} is a set $Q$ equipped with a binary operation $\triangleright$, which satisfies the following properties.
\begin{enumerate}
\item (idempotence property) $x\triangleright x=x$ $\forall x \in Q$.
\item (involutory property) $(x\triangleright y) \triangleright y=x$ $\forall x,y \in Q$.
\item (right distributive property) $(x\triangleright y) \triangleright z=(x\triangleright z) \triangleright (y\triangleright z)$ $\forall x,y,z \in Q$.
\item (medial property) $(w\triangleright x) \triangleright (y\triangleright z)=(w\triangleright y) \triangleright (x\triangleright z)$ $\forall w,x,y,z \in Q$.
\end{enumerate}
\end{definition}

All the quandles we consider in this paper satisfy Definition \ref{qdef}, but we should mention that for general quandles the medial property is removed, and the involutory property is replaced by the weaker requirement that for each $y \in Q$, the formula $\beta_y(x)=x \triangleright y$ defines a permutation $\beta_y$ of $Q$. 

A particular type of involutory medial quandle is the core quandle of an abelian group.

\begin{definition}
\label{Core}
If $A$ is an abelian group, then the \emph{core quandle} $\textup{Core}(A)$ is an involutory medial quandle on the set $A$, defined by $a \triangleright b = 2 b - a$.
\end{definition}

In this paper we use the word ``link'' to mean an unoriented classical link; when we occasionally refer to an oriented link, we say so. A link is denoted $L=K_1 \cup \dots \cup K_{\mu}$, and $D$ denotes a diagram of $L$. As usual, $D$ is obtained from a generic projection of $L$ in the plane, i.e.\ a projection whose only singularities are crossings (transverse double points). At each crossing, two short segments are removed to distinguish the underpassing curve from the overpassing curve, as indicated in Fig.\ \ref{crossfig}. The effect of removing these segments is to cut the images of $K_1, \dots, K_ \mu$ into arcs. We use $A(D)$ to denote the set of arcs of $D$, and $C(D)$ to denote the set of crossings of $D$. 

\begin{figure} [bth]
\centering
\begin{tikzpicture} [>=angle 90]
\draw [thick] (1,0.5) -- (-0.6,-0.3);
\draw [thick] (-0.6,-0.3) -- (-1,-0.5);
\draw [thick] (-1,0.5) -- (-.2,0.1);
\draw [thick] (0.2,-0.1) -- (1,-0.5);
\node at (1.4,0.5) {$a$};
\node at (-1.4,0.5) {$b$};
\node at (1.4,-0.5) {$b'$};
\end{tikzpicture}
\caption{The arc labels $b$ and $b'$ may be interchanged.}
\label{crossfig}
\end{figure}
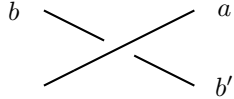

When we use notation like $S(L)$ to denote a structure $S$ obtained from a diagram $D$, we implicitly intend that the structure is a link type invariant, up to the appropriate kind of isomorphism. For the structures discussed in this paper, invariance is well known and may be verified using the Reidemeister moves.

\begin{definition}
\label{imq}
Let $D$ be a diagram of a link $L$. Then $\textup{IMQ}(L)$ is the involutory medial quandle generated by elements $q_a$, $a \in A(D)$, subject to the requirement that if Fig.\ \ref{crossfig} represents a crossing of $D$, then $q_b \triangleright q_a= q_{b'}$ and $q_{b'} \triangleright q_a= q_b$.
\end{definition}
Joyce \cite{J} used the notation $\mathrm{AbQ}_2(L)$ rather than $\textup{IMQ}(L)$. He proved the following. 

\begin{theorem} (Joyce \cite[Sec.\ 18]{J})
\label{joyceknot}
For a classical knot $K$, $\textup{IMQ}(K)$ is isomorphic to the core quandle of the homology group $H_1(X_2)$, where $X_2$ is the cyclic double cover of $\mathbb S ^3$, branched over $K$. It follows that $|\IMQ(K)|=|\det K|$.
\end{theorem}

The purpose of the present paper is to discuss the generalization of Theorem \ref{joyceknot} to classical links. The general situation is much more complicated: $\textup{IMQ}(L)$ and $\textup{Core}(H_1(X_2))$ are not isomorphic, in general, and there are other involutory medial quandles in the picture. 

The first of these other quandles is associated with a group defined as follows.

\begin{definition}
\label{img}
Let $D$ be a diagram of a link $L$. Then $\textup{IMG}(L)$, the involutory medial group of $L$, is the group generated by elements $g_a$, $a \in A(D)$, with three kinds of relations.
\begin{enumerate}
    \item If $a \in A(D)$, then $g_a ^2=1$.
    \item If $a_1,a_2,a_3 \in A(D)$ and $c_1,c_2,c_3$ are conjugates of $g_{a_1},g_{a_2},g_{a_3}$, then $c_1c_2c_3=c_3c_2c_1.$
    \item If Fig.\ \ref{crossfig} represents a crossing of $D$, then $g_a g_b g_a= g_{b'}$ and $g_a g_{b'} g_a= g_b$.
\end{enumerate}
\end{definition}

As with the quandle $\textup{IMQ(L)}$, the group $\textup{IMG(L)}$ appears in \cite{J}, but we use different notation and terminology from Joyce's. In particular, we use the term ``medial'' rather than ``abelian.'' It would be confusing to refer to $\textup{IMG}(L)$ as the ``involutory abelian group of $L$'' because $\textup{IMG}(L)$ is not commutative, in general.

\begin{definition}
\label{qimg}
Let $D$ be a diagram of a link $L$. Then $\textup{QIMG}(L)$ denotes the subset of $\textup{IMG}(L)$ that includes all conjugates of the $g_a$ elements, $a \in A(D)$. It is a quandle under the operation given by conjugation: $x \triangleright y = yxy^{-1}=yxy$. 
\end{definition}

Let $\mathbb{Z}^{A(D)}$ and $\mathbb{Z}^{C(D)}$ be the free abelian groups on the sets $A(D)$ and $C(D)$. Then there is a homomorphism $r_D:\mathbb{Z}^{C(D)} \to \mathbb{Z}^{A(D)}$ given by
\[
r_D(c)=2a - b - b'
\]
whenever $c \in C(D)$ is a crossing as represented in Fig.\ \ref{crossfig}. If the arcs $a,b,b'$ are not distinct then the corresponding terms in $r_D(c)$ are added together; for instance, if $a=b \neq b'$ then $r_D(c)=a - b'$.

\begin{definition}
\label{almod}
The cokernel of $r_D$ is denoted $M_A(L)_{\nu}$. The canonical epimorphism $\mathbb{Z}^{A(D)} \to M_A(L)_{\nu}$ is denoted $s_D$.
\end{definition}

The notation $M_A(L)_{\nu}$ reflects the fact that this group is the tensor product of two modules over the Laurent polynomial ring $\Lambda_ \mu = \mathbb Z [t_1^{\pm 1}, \dots, t_ \mu ^{\pm 1 }]$. One $\Lambda_ \mu$-module is the (multivariate) Alexander module $M_A(L)$ of an oriented version of $L$. The other $\Lambda_ \mu$-module is $\mathbb Z _ \nu$, the group of integers considered as a $\Lambda_ \mu$-module via the map $\nu:\Lambda_ \mu \to \mathbb Z$ with $\nu(t_i ^{\pm 1}) = -1$ $\forall i$. (That is, the scalar multiplication in $\mathbb Z _ \nu$ is given by $t_i^{\pm 1} \cdot n = -n$ $\forall i \in \{1, \dots, \mu\}$ $\forall n \in \mathbb Z$.) The group $M_A(L)_{\nu}$ may also be described in two other ways: it is the direct sum $\mathbb Z \oplus H_1(X_2)$, and it is the tensor product of the Alexander module of $\textup{IMG}(L)$ with $\mathbb Z _\nu$. See Sec.\ \ref{amod} for details.

Just as the group $\textup{IMG}(L)$ contains the quandle $\textup{QIMG}(L)$, the abelian group $M_A(L)_{\nu}$ contains a quandle $Q_A(L)_{\nu}$. In order to define $Q_A(L)_{\nu}$ we need one more ingredient, a homomorphism derived from the link module sequence of Crowell \cite{C1, C3}. 
Let $\kappa_D:A(D) \to \{1,\dots,\mu\}$ be the function with $\kappa_D(a)=i$ whenever $a$ is an arc of $D$ that belongs to the image of $K_i$, and let $A_{\mu}$ be the direct sum 
\[
A_ \mu = \mathbb{Z} \oplus \underbrace{ \mathbb{Z}_2 \oplus \cdots \oplus \mathbb{Z}_2}_{\mu-1}.
\] 
Let $\Phi_{\nu}:\mathbb{Z}^{A(D)} \to A_{\mu}$ be the homomorphism with $\Phi_{\nu}(a)=(1,0,\dots,0)$ for every $a \in A(D)$ with $\kappa_D(a)=1$, and $\Phi_{\nu}(a)=(1,0,\dots,0,1,0,\dots,0)$, with the second $1$ in the $i$th coordinate, for every $a \in A(D)$ with $\kappa_D(a)=i>1$. Then it is easy to see that $\Phi_{\nu} (r_D(c))=0$ $\forall c \in C(D)$, so $\Phi_{\nu}$ defines a homomorphism $\phi_{\nu}:M_A(L)_{\nu} \to A_{\mu}$ with $\phi_{\nu}(s_D(a))=\Phi_{\nu}(a)$ $\forall a \in A(D)$. 

\begin{definition}
\label{coreprime} 
The subset $\phi_{\nu}^{-1}(\phi_{\nu}(s_D(A(D)))) \subset M_A(L)_ \nu$ is denoted $Q_A(L)_ \nu$.
\end{definition}

Notice that every $x \in Q_A(L)_ \nu$ has $2 \phi_ \nu(x)=(2,0, \dots, 0) \in A_ \mu$. If $x,y \in Q_A(L)_ \nu$, then $\phi_ \nu (x \triangleright y) = 2 \phi_ \nu(y)-\phi_ \nu(x) = 2 \phi_ \nu(x)-\phi_ \nu(x) =\phi_ \nu(x)$, so $x \triangleright y \in $ $Q_A(L)_ \nu$. We deduce that $Q_A(L)_ \nu$ is a subquandle of $\textup{Core}(M_A(L)_ \nu)$.

We are now ready to discuss extending Theorem \ref{joyceknot} to links. The extension is stated in three separate theorems. The first two theorems concern the relationships among the quandles, and the third concerns the quandles' cardinalities.

\begin{theorem}
\label{main1}
\begin{enumerate}
    \item There is a surjective quandle map $\textup{IMQ}(L) \to \textup{QIMG}(L)$ defined by $q_a \mapsto g_a$ $\forall a \in A(D)$. This map is an isomorphism if $\mu=1$, or if $\mu=2$ and $\det L \neq 0$. In general, $\IMQ(L)$ and $\QIMG(L)$ are not isomorphic if $\mu >2$.
    \item There is a quandle isomorphism  $\textup{QIMG}(L) \to Q_A(L)_ \nu$ defined by $g_a \mapsto s_D(a)$ $\forall a \in A(D)$.
    \item There is an injective quandle map $Q_A(L)_ \nu \to \Core(H_1(X_2))$. If $\mu \leq 2$, then $Q_A(L)_ \nu$ and  $\textup{Core}(H_1(X_2))$ are isomorphic quandles. If $\mu>2$, there is no surjective quandle map $Q_A(L)_ \nu \to \Core(H_1(X_2))$.
\end{enumerate}
\end{theorem}

\begin{theorem}
\label{main2}
Consider the following statements about involutory medial quandles associated to two links $L$ and $L'$.
\begin{enumerate}
    \item $\textup{IMQ}(L) \cong \textup{IMQ}(L')$.
    \item $\textup{QIMG}(L) \cong \textup{QIMG}(L')$.
    \item $Q_A(L)_ \nu \cong Q_A(L')_ \nu$.
    \item $\textup{Core}(H_1(X_2)) \cong \textup{Core}(H_1(X'_2))$, where $X'_2$ is the cyclic double cover of $\mathbb S ^3$, branched over $L'$.
\end{enumerate} 
The implications $1 \implies 2 \iff 3 \implies 4$ hold in general. The converse of $1 \implies 2$ holds when $\mu=1$, and it also holds when $\mu=2$ and $\det L \neq 0$; it fails when $\mu >2$. The converse of $3 \implies 4$ holds when $\mu \leq 3$, and fails when $\mu>3$.
\end{theorem}

\begin{theorem}
\label{main3}
If $\mu=1$, then $\textup{IMQ}(L)$, $\textup{QIMG}(L)$, $Q_A(L)_ \nu$ and $H_1(X_2)$ are all of cardinality $|\det L|$.  If $\mu>1$ and $\det L \neq 0$, then $|H_1(X_2)|=| \det L|$ and
\[
\frac {\mu |\det L|}{2} \geq | \textup{IMQ}(L) | \geq | \textup{QIMG}(L) | = | Q_A(L)_ \nu|  =\frac {\mu |\det L|}{2^{\mu -1}}.
\]
If $\det L = 0$, then  $\textup{IMQ}(L)$, $\textup{QIMG}(L)$, $Q_A(L)_ \nu$ and $H_1(X_2)$ are all infinite.
\end{theorem}

When $\mu=1$, or $\mu=2$ and $\det L \neq 0$, Theorems \ref{main1} and \ref{main3} together imply $\textup{IMQ}(L) \cong \textup{QIMG}(L) \cong Q_A(L)_ \nu \cong \textup{Core}(H_1(X_2))$. Thus Theorems \ref{main1} -- \ref{main3} do extend Theorem \ref{joyceknot}.

Theorems \ref{main2} and \ref{main3} both leave room for improvement. In Theorem \ref{main2}, we do not know whether the converse of $1 \implies 2$ holds or fails when $\mu=2$ and $\det L =0$. In Theorem \ref{main3}, we hope that the inequalities can be sharpened.

Here is an outline of our discussion. In Sec.\ \ref{sec:mistakes}, we summarize mistakes that appeared in earlier versions of the paper. In Sec.\ \ref{struc}, we discuss the elementary theory of involutory medial quandles, which is a small part of the work of Jedli\v{c}ka, Pilitowska, Stanovsk\'{y} and Zamojska-Dzienio on general medial quandles \cite{JPSZ1, JPSZ2}. In Sec.\ \ref{sec:linkq}, we discuss some properties of $\IMG(L), \IMQ(L)$ and $\QIMG(L)$. In Sec.\ \ref{amod}, we connect the abelian group $M_A(L)_{\nu}$ with classical machinery involving Alexander matrices, Alexander modules, and branched double covers. 

In Sec.\ \ref{sec:longitudes}, we discuss the elements of $M_A(L)_\nu$ that represent longitudes of the components of $L$. In Sec.\ \ref{sec:twoexamples}, we illustrate the definitions of the previous sections with two 3-component links. These two examples serve to verify the failure of the converse of $1 \implies 2$ in Theorem \ref{main2}. In Sec.\ \ref{2red} we discuss $Q_A(L)_\nu$, and in Sec.\ \ref{proofmain}, we complete the proof of Theorem \ref{main1}. In Sec.\ \ref{sec:proof3}, we see that Theorem \ref{main3} follows almost immediately once we have Theorem \ref{main1}. 

It takes a little more time to verify Theorem \ref{main2}. In Sec.\ \ref{proof1}, we show that $M_A(L)_\nu$ and $\phi_\nu$ can be defined by modifying Definitions \ref{almod} and \ref{coreprime} to use elements of $Q_A(L)_\nu$, rather than elements of of $A(D)$. In Sec. \ref{proof2}, we show that $M_A(L)_\nu$ and $\phi_\nu$ can also be defined in a similar way, using elements of $\IMQ(L)$. In Sec.\ \ref{coresec}, we introduce the characteristic subquandle of the core quandle of a finitely generated abelian group, and prove that the characteristic subquandle is a classifying invariant. In Sec.\ \ref{proof3}, we use these results to verify the positive assertions  of Theorem \ref{main2} (i.e., the assertions that certain implications hold). 

The fact that $4 \centernot \implies 3$ in Theorem \ref{main3} is verified with two pairs of examples in Sec.\ \ref{sec:fiveexamples}. The links in the second pair are denoted $L'$ and $L''$. They have the interesting property that $Q_A(L')_\nu \not \cong Q_A(L'')_\nu$, even though the cyclic double covers of $\mathbb S^3$ branched over $L'$ and $L''$ are homeomorphic to each other.

\section{Mistakes}
\label{sec:mistakes}

The work presented in this paper was developed over a period of approximately two years. For much of that time, we mistakenly believed that the converses of $1 \implies 2$ and $3 \implies 4$ in Theorem \ref{main2} are generally valid. We persisted in the former mistake long enough that it was included in the account published in this journal \cite{mvaq2}. We are grateful to Kyle Miller \cite{Mi} for helping us understand the mistake. We are also grateful to the editors for the opportunity to publish a replacement for the entire paper \cite{mvaq2}, rather than a mere erratum. We hope that readers will simply ignore the incorrect account published in \cite{mvaq2}.

The first paper in the series \cite{mvaq1} was not affected by the errors in \cite{mvaq2}. (As far as we know, the only mistake in \cite{mvaq1} is a typographical error in a subscript on p.\ 20.) The results stated in the third paper \cite{mvaq3} are also unaffected by the errors in \cite{mvaq2}. However, there is a regrettable error in an offhand comment in the introduction of \cite{mvaq3}, where the fundamental quandle is mistakenly described as the union of the conjugacy classes of meridians in the link group. (This is the same kind of mistake as believing the converse of $1 \implies 2$ in Theorem \ref{main2}.) A correct version of this comment would describe the union of the conjugacy classes of meridians as an image of the fundamental quandle. This offhand comment was intended only for motivation, and the mistake does not affect any of the results stated in \cite{mvaq3}.

\section{Involutory medial quandles}
\label{struc}

In this section we give a brief account of some theory regarding involutory medial quandles. The results are extracted from the more general discussion of medial quandles given by Jedli\v{c}ka, Pilitowska, Stanovsk\'{y} and Zamojska-Dzienio \cite{JPSZ1, JPSZ2}. The notation and terminology in these papers are different from those of many knot-theoretic references, like \cite{EN} or \cite{J}; for instance the roles of the first and second variables in the quandle operation are reversed. So although the mathematical content of this section is all taken from \cite{JPSZ1} and \cite{JPSZ2}, notation and terminology have been modified for the convenience of readers familiar with the conventions of the  knot-theoretic literature.

Let $Q$ be an involutory medial quandle. An \emph{automorphism} of $Q$ is a bijection $f:Q \to Q$ with $f(x \triangleright y)=f(x) \triangleright f(y)$ $\forall x,y \in Q$. A group structure on the set $\Aut(Q)$ of automorphisms of $Q$ is defined by function composition. If $y \in Q$ then the \emph{translation} of $Q$ corresponding to $y$ is the function $\beta_y:Q \to Q$ given by $\beta_y(x)=x \triangleright y$; property 3 of Definition \ref{qdef} implies that $\beta_y$ is an automorphism of $Q$. (Translations are called \emph{inner automorphisms} in some references.) Notice that property 2 of Definition \ref{qdef} implies that $\beta_y^{-1}=\beta_y$ $\forall y \in Q$. If $y,z \in Q$ then the composition $\beta_y\beta_z^{-1}=\beta_y \beta_z$ is an \emph{elementary displacement} of $Q$; the subgroup of $\Aut(Q)$ generated by the elementary displacements is denoted $\Dis(Q)$, and its elements are \emph{displacements}. (Displacements are called \emph{transvections} in some references.)

\begin{proposition}
\label{imqprops}
If $Q$ is an involutory medial quandle, then the following properties hold.
\begin{enumerate}
    \item $\beta_{y \triangleright z}=\beta_z\beta_y\beta_z$ $\forall y,z \in Q$.
    \item $\beta_y\beta_z\beta_x=\beta_x\beta_z\beta_y$ $\forall x,y,z \in Q$.
    \item $\Dis(Q)$ is an abelian group.
\end{enumerate}
\end{proposition}
\begin{proof}
Suppose $x,y,z \in Q$. For item 1, notice that
\[
\beta_{y \triangleright z}(x)=x \triangleright (y \triangleright z)=((x \triangleright z) \triangleright z) \triangleright (y \triangleright z)
\]
\[
=((x \triangleright z) \triangleright y) \triangleright (z \triangleright z)=((x \triangleright z) \triangleright y) \triangleright z =\beta_z\beta_y\beta_z(x) \text{.}
\]
For item 2, notice that property 4 of Definition \ref{qdef} tells us $\beta_{y \triangleright z}\beta_x =\beta_{x \triangleright z}\beta_y$. It follows from this and item 1 that $\beta_z\beta_y\beta_z\beta_x=\beta_z\beta_x\beta_z\beta_y$. As $\beta_z^2$ is the identity map, we deduce that
\[
\beta_y\beta_z\beta_x=\beta_z\beta_z\beta_y\beta_z\beta_x=\beta_z\beta_z\beta_x\beta_z\beta_y=\beta_x\beta_z\beta_y.
\]

Now, suppose $a,b,c,d \in Q$. Using the formula of item 2 twice, we have \[(\beta_a\beta_b)(\beta_c\beta_d)=\beta_a(\beta_b\beta_c\beta_d)=\beta_a(\beta_d\beta_c\beta_b)
\]
\[=(\beta_a\beta_d\beta_c)\beta_b=(\beta_c\beta_d\beta_a)\beta_b=(\beta_c\beta_d)(\beta_a\beta_b).
\]
That is, the elementary displacements $\beta_a\beta_b$ and $\beta_c\beta_d$ commute.
\end{proof}
\begin{definition}
\label{orbit}
Let $Q$ be an involutory medial quandle. An \emph{orbit} in $Q$ is an equivalence 
 class under the equivalence relation generated by $x \sim x \triangleright y$ $\forall x,y \in Q$.
\end{definition}
\begin{proposition}
\label{orb}
If $x \in Q$ then the orbit of $x$ in $Q$ is $\{d(x) \mid d \in \Dis(Q)\}$.
\end{proposition}
\begin{proof}
A displacement is a composition of translations, so the orbit of $x$ includes $d(x)$ for every displacement $d$. 

Now, suppose $y$ is an element of the orbit of $x$. Then there are elements $y_1,\dots,y_n \in Q$ such that $y=\beta_{y_n}\cdots \beta_{y_1}(x)$. If $n$ is even, then $\beta_{y_n}\cdots \beta_{y_1}=(\beta_{y_n}\beta_{y_{n-1}})\cdots (\beta_{y_2}\beta_{y_1})$ is a displacement. If $n$ is odd, then $y=\beta_{y_n}\cdots \beta_{y_1}\beta_x(x)$ and $\beta_{y_n}\cdots \beta_{y_1}\beta_x=(\beta_{y_n}\beta_{y_{n-1}})\cdots (\beta_{y_3}\beta_{y_{2}})(\beta_{y_1}\beta_x)$ is a displacement.
\end{proof}
\begin{definition}
\label{semireg}
An involutory medial quandle is \emph{semiregular} if the identity map is the only displacement with a fixed point.
\end{definition}

If $A$ is an abelian group, the subgroup $\{a \in A \mid 2a=0\}$ is denoted $A(2)$.
\begin{proposition}
\label{Coreprop}
Let $A$ be an abelian group. Then $\textup{Core}(A)$ is involutory, medial and semiregular. Moreover, $\Dis(\Core(A)) \cong A/A(2)$.
\end{proposition}
\begin{proof}
 It is easy to see that core quandles satisfy Definition \ref{qdef}. 
 
To verify semiregularity, suppose $d \in \Dis(\Core(A))$. Then $d=\beta_{a_1}\cdots\beta_{a_{2n}}$ for some elements $a_1,\dots,a_{2n} \in A$, so $d(a)=2a_1-2a_2+- \dots -2a_{2n}+a$ $\forall a \in A$. If $d(a)=a$ for one $a \in A$, it must be that $2a_1-2a_2+- \dots -2a_{2n}=0$, and hence $d(a)=a$ for every $a \in A$.

Let $f:A \to \Dis(\Core(A))$ be the function with $f(a)=\beta_a \beta_0$ $\forall a \in A$. Then $f(a)(x)=2a-(2 \cdot 0 - x)=2a+x$ $\forall a,x \in A$. As
\[
f(a_1+a_2)(x)=2(a_1+a_2)+x=2a_1 + (2a_2+x)=f(a_1)(f(a_2)(x)) \text{,}
\]
$f$ is a homomorphism. It is obvious that $\ker f=A(2)$. If $a_1,a_2 \in A$ then the elementary displacement $\beta_{a_1}\beta_{a_2}$ is given by $\beta_{a_1}\beta_{a_2}(x)=2a_1-(2a_2-x)=2(a_1-a_2)+x$, so $\beta_{a_1}\beta_{a_2}=f(a_1-a_2)$. The elementary displacements $\beta_{a_1}\beta_{a_2}$ generate $\Dis(\Core(A))$, so it follows that $f$ is surjective. 
\end{proof}

\begin{proposition}
\label{invorb}
Let $Q$ be an involutory medial quandle. Then for each orbit in $Q$, there is a subgroup $S \subseteq \Dis(Q)$ such that the orbit is isomorphic, as a quandle, to $\textup{Core}(\Dis(Q)/S)$. If $Q$ is semiregular, then each orbit in $Q$ is isomorphic to $\textup{Core}(\Dis(Q))$. 
\end{proposition}
\begin{proof}
Let $x \in Q$. Observe that if $d=\beta_y \beta_z$ is an elementary displacement, then as $\Dis(Q)$ is commutative,  \[
\beta_x d \beta_x = (\beta_x \beta_y) (\beta_z \beta_x) = (\beta_z \beta_x) (\beta_x \beta_y) = \beta_z \beta_x^2 \beta_y = \beta_z  \beta_y= \beta_z^{-1}  \beta_y^{-1} = d^{-1}.
\]
The elementary displacements generate $\Dis(Q)$, so it follows that $\beta_x d \beta_x = d^{-1}$ $\forall d \in \Dis(Q)$.

Let $S_x=\{d \in \Dis(Q) \mid d(x)=x\}$ be the stabilizer of $x$ in $\Dis(Q)$. Then $d_1,d_2 \in \Dis(Q)$ determine the same coset in the quotient group $\Dis(Q)/S_x$ if and only if $d_1(x)=d_2(x)$, so there is an injective map $f_x$ from $\Dis(Q)/S_x$ to the orbit of $x$ in $Q$, defined by $f_x(dS_x)= d(x)$. According to Proposition \ref{orb}, $f_x$ is not only injective; it is also surjective.
 
Suppose $c,d \in \Dis(Q)$. As $d$ is a quandle automorphism of $Q$,
\[
d(d^{-1}c(x) \triangleright x)=dd^{-1}c(x) \triangleright d(x)=c(x) \triangleright d(x)
\]
and hence
\[
c(x) \triangleright d(x) = d(d^{-1}c(x) \triangleright x)= d \beta_xd^{-1}c(x)= (d) (\beta_x d^{-1} \beta_x) (\beta_x c \beta_x)(x).
\]
According to the observation of the first paragraph, it follows that 
\[
c(x) \triangleright d(x) = (d) (d) (c^{-1}) (x) = d^2c^{-1}(x) \text{.}
 \]
The quandle operation in the core quandle of an abelian group is given by $a \triangleright b = 2b-a$ in additive notation, or $a \triangleright b = b^2a^{-1}$ in multiplicative notation. It follows that
\[
f_x(cS_x \triangleright dS_x)  = f_x(d^2c^{-1}S_x) = d^2c^{-1}(x) = c(x) \triangleright d(x) = f_x(cS_x) \triangleright f_x(dS_x) \text{,}
\]
so the bijection $f _x$ is a quandle isomorphism between $\textup{Core}(\Dis(Q)/S_x)$ and the orbit of $x$ in $Q$.

If $Q$ is semiregular then for every $x$ in $Q$, $S_x$ contains only the identity map.
\end{proof}

\begin{corollary}
\label{surjiso}
Suppose $Q_1$ and $Q_2$ are semiregular, involutory medial quandles, and $f:Q_1 \to Q_2$ is a surjective quandle map. Then $f$ induces an epimorphism $\Dis(f):\Dis(Q_1) \to \Dis(Q_2)$ of abelian groups, and $f$ is an isomorphism if and only if both of these statements hold: (a) $\Dis(f)$ is an isomorphism. (b) If $x$ and $y$ belong to different orbits in $Q_1$, then $f(x) \neq f(y)$.
\end{corollary}
\begin{proof}
The epimorphism $\Dis(f)$ is given by 
\[
\Dis(f)(\beta_{q_1} \cdots \beta_{q_{2n}}) = \beta_{f(q_1)} \cdots \beta_{f(q_{2n})} \quad \forall q_1, \dots , q_{2n} \in Q_1.
\]

If $f$ is an isomorphism, then it is clear that (a) and (b) hold. For the converse, suppose (a) and (b) hold, $x \neq y \in Q_1$ and $f(x)=f(y)$. Then (b) tells us that $y$ belongs to the orbit of $x$ in $Q_1$. According to Proposition \ref{orb}, it follows that there is a displacement $d \in \Dis(Q_1)$ with $d(x)=y$. Then $\Dis(f)(d)(f(x))=f(y)$, so $f(x)=f(y)$ is a fixed point of $\Dis(f)(d)$. As $Q_2$ is semiregular, it follows that $\Dis(f)(d)$ is the identity map of $Q_2$. Hence $d \in \ker \Dis(f)$, violating (a).
\end{proof}

Before proceeding, we should mention that the theory developed by Jedli\v{c}ka, Pilitowska, Stanovsk\'{y} and Zamojska-Dzienio \cite{JPSZ1,JPSZ2} is more general and more powerful than we have indicated; they provide a complete structure theory of medial quandles. 

\section{The quandles $\IMQ(L)$ and $\textup{QIMG}(L)$}
\label{sec:linkq}

Let $D$ be a diagram of a link $L$. The group $\textup{IMG}(L)$ and the quandles $\textup{IMQ}(L)$ and $\QIMG(L)$ are defined in the introduction. In this section we mention some properties of the two quandles, and we discuss the relationships between their automorphism groups and $\IMG(L)$. These relationships fall under Joyce's concept of ``augmented quandles'' \cite{J}.

It is not difficult to count the orbits in $\textup{IMQ}(L)$.
\begin{proposition}
\label{imqorb}
$\textup{IMQ}(L)$ has $\mu$ orbits, one for each component of $L$. The orbit corresponding to $K_i$ includes every element $q_a$ such that $a \in A(D)$ and $\kappa_D(a)=i$.
\end{proposition}
\begin{proof}
By definition, $\textup{IMQ}(L)$ is generated by the elements $q_a$ with $a \in A(D)$, so every $x \in \IMQ(L)$ is obtained from some $q_a$ through some sequence of $\triangleright$ operations. Thus every orbit in $\textup{IMQ}(L)$ contains an element associated with a particular component $K_i$ of $L$.

Suppose $i \in \{1,\dots,\mu\}$, and $a$ is an arc of $A(D)$ that belongs to $K_i$. As we walk along $K_i$ starting at $a$, each time we pass from one arc of $K_i$ to another we obtain another element of the same orbit of $\textup{IMQ}(L)$, because we pass through a crossing in which the two arcs of $K_i$ are the two underpassing arcs. Therefore every arc $b$ belonging to $K_i$ has $q_b$ in the same orbit of $\textup{IMQ}(L)$ as $q_a$. 

To verify that no orbit contains $q_a$ elements corresponding to arcs belonging to distinct components, let $Q$ be the quandle obtained from $\textup{IMQ}(L)$ by adding relations that require $x \triangleright y = x$ $\forall x,y$. It is easy to see that $Q$ has $\mu$ elements, one for each component of $L$; and there is a well-defined quandle homomorphism mapping $\textup{IMQ}(L)$ onto $Q$.
\end{proof}

\begin{proposition}
\label{imgbetahom}
There is a homomorphism $\beta:\textup{IMG}(L) \to \Aut(\textup{IMQ}(L))$, with $\beta(g_a)=\beta_{q_a}$ $\forall a \in A(D)$.
\end{proposition}
\begin{proof}
Recall that $\IMG(L)$ is generated by the elements $g_a$ with $a \in A(D)$, subject to the three kinds of relations mentioned in Definition \ref{img}. To prove the proposition, it suffices to show that these three kinds of relations are satisfied in $\Aut(\textup{IMQ}(L))$.

The first kind of relation, $\beta_{q_a}^2=1$, follows immediately from the fact that $\textup{IMQ}(L)$ is involutory. The second and third kinds of relations are verified in items 2 and 1 of Proposition \ref{imqprops}, respectively. \end{proof}

\begin{corollary}
\label{imh}
Let $\IMG ^2(L)$ be the subgroup of $\IMG(L)$ generated by the products $g_a g_b$ with $a,b \in A(D)$, and let $a^* \in A(D)$ be a fixed element. Then:
\begin{enumerate}
    \item $\IMG ^2(L)$ is an abelian group.
    \item $\IMG ^2(L)$ is generated by the elements $h_a = g_a g_{a^*}$ with $a \in A(D)$. Also, $h_{a^*}=1$.
    \item For any crossing of $D$ as pictured in Fig.\ \ref{crossfig}, $h_{b'}=h_a^2h_b^{-1}$. 
    \item The homomorphism $\beta$ of Proposition \ref{imgbetahom} has $\beta(\IMG ^2(L))=\Dis(\IMQ(L))$.
\end{enumerate}
\end{corollary}

\begin{proof}
For item 1, notice that if $a,b,c,d \in A(D)$ then 
\[
(g_ag_b)(g_cg_d) = g_a(g_bg_cg_d) = g_a(g_dg_cg_b) 
\]
\[
= (g_ag_dg_c)g_b = 
(g_cg_dg_a)(g_b)=
(g_cg_d)(g_ag_b).
\]
That is, the generators of $\IMG ^2(L)$ all commute with each other.

Item 2 follows from the equalities $g_ag_b = h_ah_b^{-1}$ and $g_{a^*}^2=1$.

For item 3, notice that 
\[
h_{b'} = g_{b'}g_{a^*} = g_ag_bg_ag_{a^*} = g_ag_{a^*}g_{a^*}g_bg_ag_{a^*}=h_ah_b^{-1}h_a
\]
and according to item 1, $h_ah_b^{-1}h_a=h_a^2 h_b^{-1}$.

Item 4 follows from the fact that $\Dis(\IMQ(L))$ is generated by the elementary displacements.
\end{proof}

The following analogous results hold for $\QIMG(L)$. The proofs are the same, \emph{mutatis mutandi}.

\begin{proposition}
\label{qimgorb}
$\QIMG(L)$ has $\mu$ orbits, one for each component of $L$. The orbit corresponding to $K_i$ includes every element $g_a$ such that $a \in A(D)$ and $\kappa_D(a)=i$.
\end{proposition}

\begin{proposition}
\label{qimgbetahom}
There is a homomorphism $\beta:\textup{IMG}(L) \to \Aut(\QIMG(L))$, with $\beta(g_a)=\beta_{g_a}$ $\forall a \in A(D)$.
\end{proposition}

\begin{corollary}
\label{qimgh}
The homomorphism $\beta$ of Proposition \ref{qimgbetahom} maps $\IMG ^2(L)$ onto $\Dis(\IMQ(L))$.
\end{corollary}

The following results provide some more properties of $\QIMG(L)$.

\begin{lemma}
\label{reglem}
Suppose $n$ is an odd, positive integer, and $c_1,\dots,c_n$ are conjugates in $\textup{IMG}(L)$ of $g_{a_1},\dots, g_{a_n}$, where $a_1,\dots,a_n \in A(D)$. Then $c_1 \cdots c_n=(c_1 \cdots c_n)^{-1}=c_n \cdots c_1$ in $\textup{IMG}(L)$.
\end{lemma}
\begin{proof}
If $n=1$ we have $c_1=gg_{a_1}g^{-1}$ for some $g \in \IMG(L)$, and according to part 1 of Definition \ref{img}, $c_1^2=gg_{a_1}^2g^{-1}=gg^{-1}=1$. If $n=3$ then according to part 2 of Definition \ref{img} and the $n=1$ case of the lemma, we have $(c_1c_2 c_3)^2=(c_1c_2c_3)(c_3c_2c_1)=(c_1c_2)(c_3^2)(c_2c_1)=c_1c_2^2c_1=c_1^2=1$.

The proof proceeds using induction on $n \geq 5$. The inductive hypothesis implies that the lemma holds when $n$ is replaced by $1,3,n-4$ or $n-2$, so
\[
(c_1 \cdots c_n)^2= (c_1 \cdots c_{n-2})(c_{n-1}c_nc_1)(c_2 \cdots c_n)
\]
\[
= (c_{n-2} \cdots c_1) (c_1c_nc_{n-1})(c_2 \cdots c_n) = (c_{n-2} \cdots c_2) c_1^2 (c_nc_{n-1}c_2)(c_3 \cdots c_n)
\]
\[
= (c_{n-2} \cdots c_2) (c_2c_{n-1}c_n)(c_3 \cdots c_n)= (c_{n-2} \cdots c_3) c_2^2 c_{n-1}c_n (c_3 \cdots c_n)
\]
\[
= (c_3 \cdots c_{n-2}) c_{n-1}c_n (c_3 \cdots c_n) = (c_3 \cdots c_n)^2=1.
\]
\end{proof}

\begin{corollary}
\label{imgcor}
If $x \in \IMG^2(L)$ and $x^2=1$, then $\beta(x)=1 \in \Dis(\QIMG(L))$.
\end{corollary}
\begin{proof}
There are $a_1,\dots,a_{2n} \in A(D)$ such that $x=g_{a_1} \cdots g_{a_{2n}}$. As $\beta(x)$ is defined using conjugation by $x$, it follows that for any $z\in \IMQ(L)$, $\beta(x)(z)=xzx^{-1}$. Then according to Lemma \ref{reglem}, 
\[
\beta(x)(z)=g_{a_1} \cdots g_{a_{2n}} \cdot(z g_{a_{2n}} \cdots g_{a_{1}})=g_{a_1} \cdots g_{a_{2n}} \cdot(g_{a_1} \cdots g_{a_{2n}} z) = x^2 z = z.
\]
\end{proof}

\begin{proposition}
\label{imqsemi}
$\textup{QIMG}(L)$ is semiregular.
\end{proposition}
\begin{proof}
Suppose $d \in \Dis(\QIMG(L))$ and there is a $z \in \IMQ(L)$ with $z=d(z)$. According to Corollary \ref{qimgh}, there is an $x=g_{a_1} \cdots g_{a_{2n}} \in \IMG^2(L)$ such that $d=\beta(x)$; then $z=d(z)=\beta(x)(z)=xzx^{-1}$. Then according to Lemma \ref{reglem},
\[
z= x z x^{-1} = g_{a_1} \cdots g_{a_{2n}} \cdot(z g_{a_{2n}} \cdots g_{a_{1}})=g_{a_1} \cdots g_{a_{2n}} \cdot(g_{a_1} \cdots g_{a_{2n}} z) = x^2 z.
\]
Cancelling $z$, we conclude that $x^2=1$. Corollary \ref{imgcor} tells us that $d=\beta(x)$ is the identity map.
\end{proof}

\section{The group $M_A(L)_{\nu}$}
\label{amod}
In this section we connect $M_A(L)_{\nu}$ with three other abelian groups. Each of the three connections provides its own insight into the properties of $M_A(L)_{\nu}$.

\subsection{$M_A(L)_{\nu}$ and the Alexander module of $L$}

The (multivariate) Alexander module $M_A(L)$ is a famous invariant of oriented links. It is a module over the ring $\Lambda_ \mu=\mathbb{Z}[t_1^{\pm 1},\dots,t_{\mu}^{\pm 1}]$ of Laurent polynomials, and the effect of reversing the orientation of a link component is to interchange the roles of $t_i$ and $t_i^{-1}$ for the variable $t_i$ corresponding to that component. (N.b. Many references use the term ``Alexander module'' to refer to the reduced version of the module, obtained by setting $t_i=t_j$ $\forall i,j$.) The theory of Alexander modules is very rich, and includes many connections with other invariants. We do not attempt to survey this rich theory here; the reader who would like an overview is referred to Fox's famous survey \cite{F}, and to Hillman's excellent book \cite{H}.

It will be useful to work with the definition of the Alexander module derived from the Wirtinger presentation using Fox's free differential calculus. 

\begin{definition}(Fox \cite[(2.8)]{F1})
\label{fdc}
Let $F$ be the free group on a set $S$, and let 
\[
w = \prod\limits _{i=1}^n s_i^{\epsilon_i} \in F \text{,}
\]
where $s_1, \dots, s_n \in S$ and $\epsilon_1, \dots , \epsilon_n \in \{ \pm 1\}$. For $1 \leq i \leq n$, define $w_i$ as follows.
\[
w_i = \begin{cases}
\prod\limits_{j=1}^{i-1} s_j^{\epsilon_j} , & \text{if } \epsilon_i=1 \\
\prod\limits _{j=1}^{i} s_j^{\epsilon_j} , & \text{if } \epsilon_i=-1
\end{cases}
\]
Then for each $s \in S$, the free derivative of $w$ with respect to $s$ is the following element of the integral group ring $\mathbb Z F$:
\[
\frac{\partial w}{\partial s} = \sum\limits_{s_i=s} \epsilon_i w_i
\]
\end{definition}

Let $L$ be an oriented link with a diagram $D$. The Wirtinger presentation of the link group $G=\pi_1(\mathbb S ^3 - L)$ has generators corresponding to the arcs of $D$, and relators corresponding to the crossings of $D$. The relator corresponding to a crossing of $D$ as indicated in Fig.\ \ref{orcrossfig} is $ab'a^{-1}b^{-1}$. (N.b.\ The labels $b,b'$ in Fig.\ \ref{orcrossfig} are not interchangeable: with respect to the orientation of $a$, $b$ is on the left and $b'$ is on the right.) It is easy to see that the abelianization $G/G'$ is free abelian, with one generator $t_i$ for each component $K_i$ of $L$; $t_i$ is the image in $G/G'$ of every generator $a \in A(D)$ with $\kappa_D(a)=i$. The abelianization map $\alpha:G \to G/G'$ is given by $\alpha(a) = t_{\kappa_D(a)}$ $\forall a \in A(D)$. The integral group ring of the abelianization, $\mathbb Z (G/G')$, is naturally isomorphic to the Laurent polynomial ring $\Lambda_\mu$, and it is conventional to identify the two rings with each other.

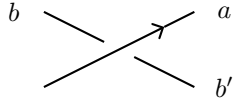
\begin{figure} [bth]
\centering
\begin{tikzpicture} [>=angle 90]
\draw [thick] (1,0.5) -- (0.6,0.3);
\draw [thick] [<-] (0.6,0.3) -- (-1,-0.5);
\draw [thick] (-1,0.5) -- (-.2,0.1);
\draw [thick] (0.2,-0.1) -- (1,-0.5);
\node at (1.4,0.5) {$a$};
\node at (-1.4,0.5) {$b$};
\node at (1.4,-0.5) {$b'$};
\end{tikzpicture}
\caption{A crossing in an oriented diagram.}
\label{orcrossfig}
\end{figure}

Now, let $\Lambda_{\mu}^{A(D)}$ and $\Lambda_{\mu}^{C(D)}$ be the free $\Lambda_{\mu}$-modules on the sets $A(D)$ and $C(D)$, and let $\rho_D:\Lambda_{\mu}^{C(D)} \to \Lambda_{\mu}^{A(D)}$ be the $\Lambda_{\mu}$-linear map given by
\[
\rho_D(c)=(1-t_{\kappa_D(b)})a+t_{\kappa_D(a)}b'-b
\]
whenever $c \in C(D)$ is a crossing of $D$ as indicated in Fig.\ \ref{orcrossfig}. That is, $\rho_D$ is defined by applying Definition \ref{fdc} to the Wirtinger relator $ab'a^{-1}b^{-1}$, and then applying $\alpha$ to the free derivatives.

\begin{definition}
\label{multialmod}
The \emph{Alexander module} $M_A(L)$ is the cokernel of $\rho_D$. The canonical surjection $\Lambda_{\mu}^{A(D)} \to M_A(L)$ is denoted $\gamma_D$.
\end{definition}

If $\nu:\Lambda_\mu \to \mathbb Z$ is the homomorphism with $\nu(t_i)=-1$ $\forall i \in \{1, \dots, \mu\}$, then the description of the group $M_A(L)_ \nu$ in Definition \ref{almod} is obtained from Definition \ref{multialmod} simply by applying $\nu$ to all coefficients. As mentioned in the introduction, it follows that $M_A(L)_\nu$ is the tensor product of $M_A(L)$ with the $\Lambda_\mu$-module $\mathbb Z _\nu$ obtained from the abelian group $\mathbb Z$ by setting $t_i^{\pm 1} \cdot n=-n $ $ \forall i \in \{1, \dots, \mu\}$ $\forall n \in \mathbb Z$. To say the same thing in a different way: the matrix $R_D$ representing the map $r_D$ of Definition \ref{almod} is the image under $\nu$ of an Alexander matrix representing the map $\rho_D$.

\begin{proposition}
\label{redundant}
Let $r_D:\mathbb Z ^{C(D)} \to \mathbb Z ^ {A(D)}$ be the homomorphism that appears in Definition \ref{almod}. Then for each crossing $c_0 \in C(D)$, the image of $r_D$ is generated by the elements $r_D(c)$ with $c \neq c_0$.
\end{proposition}
\begin{proof}
It is well known that any one relator in a Wirtinger presentation of a link group is redundant; see \cite{F} for instance. It follows that any one row of an Alexander matrix derived from a Wirtinger presentation using Fox's free differential calculus is redundant. The matrix $R_D$ representing $r_D$ is obtained from such an Alexander matrix by setting $t_1, \dots, t_ \mu$ equal to $-1$, so any one row of $R_D$ is redundant.
\end{proof}

\subsection{$M_A(L)_{\nu}$ and the Alexander module of $\IMG(L)$}

In this subsection we discuss another way to describe $M_A(L)_\nu$ as a tensor product of an Alexander module with $\mathbb Z _\nu$. 

Definition \ref{img} implies that if $\alpha:\IMG(L) \to \IMG(L)/ \IMG(L)'$ is the abelianization map, then $\IMG(L)/\IMG(L)'$ is the abelian group generated by the elements $\alpha(g_a)$, $a \in A(D)$, subject to two kinds of relations: $\alpha(g_a)^2=1$ $\forall a \in A(D)$, and $\alpha(g_b)=\alpha(g_{b'})$ for each crossing of $D$ as indicated in Fig.\ \ref{crossfig}. The latter relations imply that $\alpha(g_{b_1})=\alpha(g_{b_2})$ if and only if $\kappa_D(b_1)=\kappa_D(b_2)$, so the images under $\alpha$ of the $g_a$ elements of $\IMG(L)$ may be denoted $t_1, \dots, t_\mu$ without ambiguity. The former relations imply directly that $t_i^2=1$ $\forall i$. These facts allow us to identify the integral group ring $\mathbb Z(\IMG(L)/ \IMG(L)')$ with the quotient ring $\Lambda'_\mu = \Lambda_ \mu / (t_1^2-1, \dots, t_\mu ^2 -1)$ in a natural way. 

Abusing notation, we use $\nu$ to denote both the map $\nu:\mathbb Z (\IMG(L)) \to \mathbb Z$ with $\nu(g_a)=-1$ $\forall a \in A(D)$, and the map $\nu:\Lambda'_\mu \to \mathbb Z$ with $\nu(t_i)=-1 $ $\forall i \in \{1, \dots, \mu \}$. We use $\mathbb Z _\nu$ to denote both the $\mathbb Z (\IMG(L))$-module and the $\Lambda'_\mu$-module on $\mathbb Z$ defined using these $\nu$ maps.

\begin{proposition}
\label{modiso}
If $M$ is the Alexander module of the group $\IMG(L)$, then $M_A(L)_\nu \cong M \otimes _{\Lambda'_\mu} \mathbb Z _\nu$.
\end{proposition}
\begin{proof}
The Alexander module $M$ of $\IMG(L)$ is a $\Lambda'_\mu$-module with a presentation matrix $J$ that has a row for each relator in Definition \ref{img}, and a column for each generator. The entries of this ``Jacobian'' matrix $J$ are obtained by applying the free differential calculus to each relator from Definition \ref{almod}, applying $\alpha$ to the free derivatives, and identifying $\mathbb Z (\IMG(L)/\IMG(L)')$ with $\Lambda'_\mu$ as mentioned above. For a thorough discussion of this approach to the Alexander modules of finitely presented groups, we refer to Crowell \cite{C3}.

We proceed to describe the entries of the Jacobian matrix $J$. Remember that each $a \in A(D)$ has $\alpha(g_a)=t_{\kappa_D(a)}$ and $t_{\kappa_D(a)}^2=1$.

1. The only nonzero free derivative of a relator $g_a^2$ is $1+g_a$, with respect to $g_a$. Its image under $\alpha$ is $1+t_{\kappa_D(a)}$. 

2. (i) If $a_1,a_2,a_3 \in A(D)$ then the relation $g_{a_1}g_{a_2}g_{a_3}=g_{a_3}g_{a_2}g_{a_1}$ gives rise to the relator $g_{a_1}g_{a_2}g_{a_3}g_{a_1}g_{a_2}g_{a_3}$, or $(g_{a_1}g_{a_2}g_{a_3})^2$. The images under $\alpha$ of the nonzero free derivatives of this relator are $1+t_{\kappa_D(a_1)}t_{\kappa_D(a_2)}t_{\kappa_D(a_3)}$ with respect to $g_{a_1}$, $t_{\kappa_D(a_1)}+t_{\kappa_D(a_2)}t_{\kappa_D(a_3)}$ with respect to $g_{a_2}$ and $t_{\kappa_D(a_1)}t_{\kappa_D(a_2)}+t_{\kappa_D(a_3)}$ with respect to $g_{a_3}$. 

(ii) If we replace $g_{a_1},g_{a_2},g_{a_3}$ with conjugates $c_1=wg_{a_1}w^{-1},c_2=xg_{a_2}x^{-1}$ and $c_3=yg_{a_3}y^{-1}$, we obtain the relator $r=(c_1c_2c_3)^2$. The free derivative of $r$ with respect to a generator $g_a$ is a sum of terms, one for each appearance of $g_a$ in $r$. For instance, if $x=x_1g_ax_2$ then this appearance of $g_a$ in $x$ provides four appearances of $g_a$ in the relator $r$. These four appearances contribute 
\[
t_{\kappa_D(a_1)}\alpha(x_1)+t_{\kappa_D(a_1)}t_{\kappa_D(a_2)}\alpha(x_1)+t_{\kappa_D(a_2)}t_{\kappa_D(a_3)}\alpha(x_1)+t_{\kappa_D(a_3)}\alpha(x_1)
\]
to the value of $\alpha(\partial r / \partial g_a)$. 

In addition to these contributions from generators appearing in $w,x$ and $y$, there are contributions from the appearances of $g_{a_1}, g_{a_2}$ and $g_{a_3}$ in the middles of $c_1, c_2$ and $c_3$. These contributions are $(1+t_{\kappa_D(a_1)}t_{\kappa_D(a_2)}t_{\kappa_D(a_3)})\alpha(w)$ with respect to $g_{a_1}$, $(t_{\kappa_D(a_1)}+t_{\kappa_D(a_2)}t_{\kappa_D(a_3)})\alpha(x)$ with respect to $g_{a_2}$ and $(t_{\kappa_D(a_1)}t_{\kappa_D(a_2)}+t_{\kappa_D(a_3)})\alpha(y)$ with respect to $g_{a_3}$.

3. If $c \in C(D)$ is a crossing as illustrated in Fig.\ \ref{crossfig}, then the relation $g_a g_b g_a= g_{b'}$ gives rise to the relator $g_a g_b g_ag_{b'}$. The images under $\alpha$ of the nonzero free derivatives of this relator are $1+t_{\kappa_D(a)}t_{\kappa_D(b)}$ with respect to $g_a$, $t_{\kappa_D(a)}$ with respect to $g_b$, and $t_{\kappa_D(b)}$ with respect to $g_{b'}$.

When we form the tensor product of the Alexander module of $\IMG(L)$ with $\mathbb Z _\nu$, the right exactness of tensor products implies that the resulting abelian group has a presentation matrix $\nu(J)$, where $J$ is the matrix whose entries are described in 1, 2, 3 above. When we apply $\nu$ to the image under $\alpha$ of a free derivative of a relator of either of the first two types, we always get $0$. For the third type of relator, we get $2$ for $g_a$, $-1$ for $g_b$ and $-1$ for $g_{b'}$. That is, the tensor product $M \otimes _{\Lambda'_\mu} \mathbb Z _\nu$ has a presentation matrix $\nu(J)$ that is the same as the matrix $R_D$ representing the map $r_D$ of Definition \ref{almod}, with extra rows of zeroes. The proposition follows. 
\end{proof}

The isomorphism $M_A(L)_\nu \cong M \otimes _{\Lambda'_\mu} \mathbb Z _\nu$ is useful because the Alexander module $M$ is itself isomorphic to a tensor product:
\begin{equation}
\label{isomodiso}
I \otimes _{\mathbb Z (\IMG(L))} \Lambda'_{\mu} \cong M \text{,}
\end{equation}
where $I$ is the augmentation ideal of the integral group ring $\mathbb Z (\IMG(L))$. (That is, $I$ is the ideal of $\mathbb Z (\IMG(L))$ generated by the elements $g-1$, where $g \in \IMG(L)$.) To be specific, if $g_a$ is one of the generators of $\IMG(L)$ then the isomorphism (\ref{isomodiso}) maps $(g_a-1) \otimes 1$ to the element of $M$ corresponding to the column of the Jacobian matrix $J$ obtained from free derivatives with respect to $g_a$. Again, we refer to Crowell \cite{C3} for details. 

Tensoring the isomorphism (\ref{isomodiso}) with the identity map of $\mathbb Z _\nu$, we obtain an isomorphism
\[
I \otimes _{\mathbb Z (\IMG(L))} \mathbb Z_{\nu} \cong I \otimes _{\mathbb Z (\IMG(L))} (\Lambda'_{\mu} \otimes_{\Lambda'_{\mu}} \mathbb Z _\nu) 
\]
\[
\cong (I \otimes _{\mathbb Z (\IMG(L))} \Lambda'_{\mu}) \otimes_{\Lambda'_{\mu}} \mathbb Z _\nu \cong M \otimes_{\Lambda'_{\mu}} \mathbb Z _\nu \text{.}
\]
Composing this isomorphism with the one from Proposition \ref{modiso}, we deduce the following.

\begin{corollary}
\label{mapping1}
There is an isomorphism 
\[
f:I \otimes _{\mathbb Z (\IMG(L))} \mathbb Z_{\nu} \to M_A(L)_\nu 
\]
with $f((g_a-1) \otimes 1)=s_D(a)$ $\forall a \in A(D)$.
\end{corollary}

The next result tells us that Corollary \ref{mapping1} provides a natural quandle map $\QIMG(L)  \to Q_A(L)_\nu$. Later, we will see that this map is always an isomorphism.

\begin{corollary}
\label{mapping}
There is a quandle map $\widehat s_D:\QIMG(L) \to Q_A(L)_\nu$, given by $\widehat s_D(g_a)=s_D(a)$ $\forall a \in A(D)$.
\end{corollary}
\begin{proof}
There is certainly a well-defined function $\IMG(L) \to I \otimes _{\mathbb Z (\IMG(L))} \mathbb Z_{\nu}$, given by $g_a \mapsto (g_a-1) \otimes 1$ $\forall a \in A(D)$. Composing this function with the isomorphism $f$ of Corollary \ref{mapping1}, we obtain a function $\IMG(L) \to M_A(L)_{\nu}$ under which the image of each generator $g_a$ is $s_D(a)$. The restriction of this function to $\QIMG(L)$ is the map $\widehat s_D$ of the statement. We must verify that $\widehat s_D$ maps $\QIMG(L)$ into $Q_A(L)_\nu$, and that $\widehat s_D$ is a quandle map.

Recall that if $i \in I$ and $z \in \mathbb Z (\IMG(L))$, then in $I \otimes _{\mathbb Z (\IMG(L))} \mathbb Z_{\nu}$, we have $(iz)\otimes 1 = i \otimes \nu(z)$. It follows that if $x,y\in \QIMG(L)$, then
\[
\widehat s _D(xyx^{-1})=f((xyx^{-1}-1) \otimes 1)
\]
\[
=f((xyx^{-1}-yx^{-1}) \otimes 1)+f((yx^{-1}-x^{-1}) \otimes 1)+f((x^{-1}-1) \otimes 1)
\]
\[
=f((x-1) \otimes \nu(yx^{-1}))+f((y-1) \otimes \nu(x^{-1}))+f((1-x) \otimes \nu(x^{-1})).
\]
By definition, the elements of $\QIMG(L)$ are conjugates of generators $g_a$, $a \in A(D)$. It follows that every element of $\QIMG(L)$ is represented by a product of an odd number of these generators, so $\nu(y)=\nu(x^{-1})=-1$. Hence 
\[
\widehat s _D(xyx^{-1})=f((x-1) \otimes 1)+f((y-1) \otimes (-1))+f((1-x) \otimes (-1))
\]
\[
=f((x-1) \otimes 1)-f((y-1) \otimes 1)+f((x-1) \otimes 1)
\]
\[
=2f((x-1) \otimes 1)-f((y-1) \otimes 1)=2 \widehat s_D(x) - \widehat s_D(y).
\]
That is, $\widehat s_D$ is a quandle map from $\QIMG(L)$ (with $\triangleright$ defined by $y \triangleright x = xyx^{-1}$) to $\textup{Core}(M_A(L)_\nu)$ (with $\triangleright$ defined by $r \triangleright s = 2s-r$). 

To verify that $\widehat s_D(\QIMG(L))$ is contained in $Q_A(L)_\nu$, notice first that if $a \in A(D)$ then $\widehat s_D(g_a)=s_D(a)$ is certainly an element of $\phi_\nu^{-1} (\phi_\nu (s_D(A(D))))=Q_A(L)_\nu$.

Suppose $g \in \QIMG(L)$ has $\widehat s_D(g) \in Q_A(L)_\nu$. Then some $a \in A(D)$ has $\phi_\nu(\widehat s_D(g)) = \phi_\nu (s_D(a))$. According to the second paragraph of the proof, for any $b \in A(D)$
\[
\phi_\nu(\widehat s_D(g_bgg_b^{-1})) = 2\phi_\nu(\widehat s_D(g_b)) - \phi_\nu(\widehat s_D(g)) 
= 2\phi_\nu( s_D(b)) - \phi_\nu( s_D(a)) .
\]
As $2\phi_\nu( s_D(a))=2\phi_\nu( s_D(b))$ $\forall a,b \in A(D)$, it follows that $\phi_\nu(\widehat s_D(g_bgg_b^{-1})) = \phi_\nu( s_D(a))$, and hence $\widehat s_D(g_bgg_b^{-1}) \in \phi_\nu^{-1} (\phi_\nu (s_D(A(D))))=Q_A(L)_\nu$.

We conclude that $(\widehat s_D) ^{-1}(Q_A(L)_\nu)$ contains $g_a$ for every $a \in A(D)$, and is closed under conjugation by elements $g_b$, $b \in A(D)$. As the $g_b$ elements generate $\IMG(L)$, it follows that $(\widehat s_D) ^{-1}(Q_A(L)_\nu)$ contains all conjugates of elements $g_a$, $a \in A(D)$. That is, $(\widehat s_D) ^{-1}(Q_A(L)_\nu)$ contains all elements of $\QIMG(L)$.
\end{proof}

\subsection{$M_A(L)_{\nu}$ and $H_1(X_2)$}
\label{kerw}
\begin{lemma}
\label{tepi}
There is a homomorphism  $w_{\nu}:M_A(L)_{\nu} \to \mathbb{Z}$ with $w_{\nu}(s_D(a)) = 1$ $\forall a \in A(D)$. For any particular arc $a^* \in A(D)$, $M_A(L)_{\nu}$ is the internal direct sum of $\ker w_{\nu}$ and the infinite cyclic subgroup generated by $s_D(a^*)$.
\end{lemma}
\begin{proof}
There is a homomorphism $W:\mathbb{Z}^{A(D)} \to \mathbb{Z}$ with $W(a)=1$ $\forall a \in A(D)$. As $W(r_D(c))=0$ $\forall c \in C(D)$, $W$ induces a map $w_{\nu}$ on $\coker r_D$.

As $w_{\nu}(s_D(a^*))=1$, $a^*$ is of infinite order in $M_A(L)_{\nu}$. Every other $a \in A(D)$ has $s_D(a)-s_D(a^*) \in \ker w_{\nu}$, so $M_A(L)_{\nu}$ is the sum of $\ker w_{\nu}$ and the subgroup generated by $s_D(a^*)$. The sum is direct because $w_{\nu}(n s_D(a^*)))=n$, so the only multiple of $s_D(a^*)$ contained in $\ker w_{\nu}$ is $0$.
\end{proof}

Notice that $w_\nu$ is simply the first coordinate of the map $\phi_{\nu}:M_A(L)_{\nu} \to A_ \mu$ mentioned in the introduction.

\begin{lemma}
\label{mstruc}
There is an isomorphism
\[
M_A(L)_{\nu} \cong \mathbb{Z}^r
\oplus \mathbb{Z}_{2^{n_1}}
\oplus \dots \oplus \mathbb{Z}_{2^{n_k}} \oplus B
\]
where $r \in \{1, \dots, \mu \}$, $r+k=\mu$ and $|B|$ is an odd integer.
\end{lemma}
\begin{proof}
Every finitely generated abelian group satisfies such an isomorphism, for some $r,k \geq 0$. Lemma \ref{tepi} tells us that $M_A(L)_{\nu}$ is infinite, so $r>0$. 

To verify that $r+k=\mu$, notice that if $\text{id}:\mathbb Z _2 \to \mathbb Z _2$ is the identity map, the right exactness of tensor products implies that $M_A(L)_{\nu} \otimes _{\mathbb Z} \mathbb Z_2$ is isomorphic to the cokernel of the map $r_D \otimes \text{id}:\mathbb{Z}^{C(D)} \otimes \mathbb Z_2 \to \mathbb{Z}^{A(D)} \otimes \mathbb Z_2$. This map has $(r_D \otimes \text{id}) (c \otimes 1)=(b \otimes 1) - (b' \otimes 1)$ whenever $c \in C(D)$ is a crossing with underpassing arcs $b,b'$. It follows that $M_A(L)_{\nu} \otimes _{\mathbb Z} \mathbb Z_2$ is generated by the elements of $(s_D \otimes \text{id})(A(D) \otimes 1)$, and if $a_1,a_2 \in A(D)$, then $(s_D \otimes \text{id})(a_1 \otimes 1)=(s_D \otimes \text{id})(a_2 \otimes 1)$ if and only if $\kappa_D(a_1)=\kappa_D(a_2)$. Therefore $M_A(L)_{\nu} \otimes _{\mathbb Z} \mathbb Z_2 \cong \mathbb Z_2 ^{\mu}$.
\end{proof}

\begin{corollary}
\label{kstruc}
There is an isomorphism
\[
\ker w_{\nu} \cong \mathbb{Z}^{r-1}
\oplus \mathbb{Z}_{2^{n_1}}
\oplus \dots \oplus \mathbb{Z}_{2^{n_k}} \oplus B \text{,}
\]
with the same $r,k$ and $B$ as in Lemma \ref{mstruc}.
\end{corollary}

\begin{proof}
The corollary follows directly from Lemmas \ref{tepi} and \ref{mstruc}.
\end{proof}

The group $\ker w_{\nu}$ is isomorphic to one of the oldest invariants in knot theory: the first homology group of the cyclic double cover $X_2$ of $\mathbb S ^3$, branched over $L$. We cannot provide a simple reference for this isomorphism, because the standard descriptions of $H_1(X_2)$ (e.g., in \cite {F} or \cite [Chap.\ 9]{L}) involve a Goeritz or Seifert matrix, rather than an Alexander matrix. One way to explain the connection is this: if $A$ is a Seifert matrix for $L$ then $A+A^T$ is a presentation matrix for $H_1(X_2)$ as a $\mathbb Z$-module (see for instance \cite[Theorem 9.1]{L}), and $tA-A^T$ is a presentation matrix for the first homology group of the total linking number cover as a $\mathbb Z [t,t^{-1}]$-module. The latter module can also be described as the quotient of the $\mathbb Z [t,t^{-1}]$-module presented by a reduced Alexander matrix (i.e. a matrix obtained from an Alexander matrix by setting all $t_i = t$) obtained by modding out a direct summand isomorphic to $\mathbb Z [t,t^{-1}]$. (See for instance \cite[p.\ 117]{L} \footnote{It is a regrettable fact that terminology is not standard in the literature. Lickorish \cite{L} used the term ``Alexander module'' for the $\mathbb Z [t,t^{-1}]$-module after the direct summand is modded out. We follow Hillman \cite{H} instead, and use the term ``reduced Alexander module'' for the $\mathbb Z [t,t^{-1}]$-module before the direct summand is modded out.}.) It does not matter which particular Alexander matrix is used, because all Alexander matrices of a link $L$ are equivalent as module presentation matrices. It follows that $H_1(X_2)$ is obtained from the abelian group presented by the matrix $R_D$ representing the map $r_D$ of Definition \ref{almod} (i.e., the abelian group $M_A(L)_{\nu}$) by modding out a direct summand isomorphic to $\mathbb Z$. According to Lemma \ref{tepi}, $\ker w_{\nu}$ can also be obtained from $M_A(L)_{\nu}$  by modding out a direct summand isomorphic to $\mathbb Z$, so $H_1(X_2) \cong \ker w_{\nu}$. A more direct description of the situation involves a recent result of Silver, Williams and the present author \cite{STW}: if $L$ is a link then it has a particular Alexander matrix which, when all the variables $t_i$ are set equal to $t$, becomes $tA-A^T$ with a column of zeroes adjoined. Then setting $t$ to $-1$ yields $-A-A^T=-(A+A^T)$ (presenting $H_1(X_2)$) with a column of zeroes adjoined (presenting a direct summand isomorphic to $\mathbb Z$). 

The next proposition is well known (see \cite[Corollary 9.2]{L}, for instance). We provide a proof for the sake of completeness.
\begin{proposition}
\label{kert}
If the determinant of $L$ is not $0$, then $|\ker w_{\nu}| =|\det L|$. If the determinant of $L$ is $0$, then $\ker w_{\nu}$ is infinite.
\end{proposition}
\begin{proof}
Let $R_D$ be the matrix representing the homomorphism $r_D$. Let $m=|C(D)|$ and $n=|A(D)|$, so $R_D$ is an $m \times n$ matrix. The determinant of $L$ satisfies the formula $|\det L|=|\Delta(-1)|$, where $\Delta$ is the reduced (one-variable) Alexander polynomial of $L$. That is, $\Delta$ is the greatest common divisor of the determinants of the $(n-1) \times (n-1)$ submatrices of a matrix obtained from an Alexander matrix by setting all $t_i$ equal to $t$. From the connection between Alexander matrices and $R_D$ mentioned before Lemma \ref{tepi}, we deduce that $|\det L|$ is the greatest common divisor of the determinants of the $(n-1) \times (n-1)$ submatrices of $R_D$. As the columns of $R_D$ sum to $0$, for any $a^* \in A(D)$ this greatest common divisor is the same as the greatest common divisor of the determinants of those $(n-1) \times (n-1)$ submatrices of $R_D$ that avoid the $a^*$ column.

Choose an arc $a^* \in A(D)$, and let $R'_D$ be the $(m+1) \times n$ matrix obtained from $R_D$ by adjoining a row whose only nonzero entry is a $1$ in the $a^*$ column. Lemma \ref{tepi} implies that $R'_D$ is a presentation matrix for the abelian group $\ker w_{\nu}$.

The fundamental structure theorem for finitely generated abelian groups tells us that $\ker w_{\nu}$ is determined up to isomorphism by the elementary ideals of $R'_D$. In particular, $\ker w_{\nu}$ is finite if and only if the greatest common divisor of the determinants of $n \times n$ submatrices of $R'_D$ is not $0$, and if this is the case then this greatest common divisor equals the order of $\ker w_{\nu}$. An $n \times n$ submatrix $S$ of $R'_D$ is either an $n \times n$ submatrix of $R_D$ (in which case $\det S=0$, because the columns of $R_D$ sum to $0$) or a matrix obtained from an $(n-1) \times n$ submatrix of $R_D$ by adjoining the new row of $R'_D$ (in which case $\pm \det S$ equals the determinant of an $(n-1) \times (n-1)$ submatrix of $R_D$ that avoids the $a^*$ column). Considering the first paragraph of this proof, we conclude that $\ker w_{\nu}$ is finite if and only if $\det L \neq 0$, and if this is the case then $|\ker w_{\nu}| =|\det L|$. \end{proof}

\begin{corollary}
\label{torsion}
Suppose $\det L \neq 0$. Then $\ker w_{\nu}$ is the torsion subgroup of $M_A(L)_{\nu}$.
\end{corollary}
\begin{proof}
On the one hand, $| \ker w_{\nu}|= | \det L|$, so $\ker w_{\nu}$ is finite. Of course it follows that $\ker w_{\nu}$ is contained in the torsion subgroup of $M_A(L)_{\nu}$. On the other hand, $w_{\nu}$ is a homomorphism to $\mathbb Z$, so an element of $M_A(L)_{\nu}$ that is not included in $\ker w_{\nu}$ cannot be an element of the torsion subgroup. 
\end{proof}

\section{Longitudes}
\label{sec:longitudes}

In this section, it will be convenient for us to work with link diagrams that satisfy the following.
\begin{definition}
\label{evend}
A diagram $D$ of $L=K_1 \cup \dots \cup K_{\mu}$ is \emph{even} if each component $K_i$ has an even number of associated arcs in $D$.
\end{definition}
\begin{proposition}
\label{evens}
Every link has even diagrams.
\end{proposition}

\begin{proof}
Let $D$ be a diagram of $L$, in which $K_i$ has an odd number of arcs. Suppose $K_i$ is the underpassing component of some crossing of $D$, i.e.\ $\kappa_D(b)=\kappa_D(b')=i$ in Fig.\ \ref{crossfig}. Then we can use a Reidemeister move of the first type to introduce a trivial crossing, which splits an arc of $K_i$ in two. (See Fig.\ \ref{trivfig}.) If $K_i$ is not the underpassing component of any crossing of $D$, then $K_i$ has only one arc in $D$; we can split this arc in two with a pair of trivial crossings.  
\end{proof}

\begin{figure} [bht]
\centering
\begin{tikzpicture}
\draw [thick] (-.5,0) to [out=170, in=180] (0, .7);
\draw [thick] (.5,0) to [out=10, in=0] (0, .7);
\draw [thick] (-1.5,-.5) -- (.5,0);
\draw [thick] (-.3,-0.05) -- (-.5,0);
\draw [thick] (1.5,-.5) -- (.3,-0.2);
\end{tikzpicture}
\caption{A trivial crossing.}
\label{trivfig}
\end{figure}
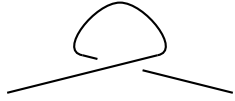

Now, suppose $D$ is an even diagram of $L$. For each component $K_i$, let the arcs of $K_i$ in $D$ be indexed as $b_{i0}, \dots, b_{i(2n_i-1)}, b_{i(2n_i)}=b_{i0}$, in the order they occur as one walks along $K_i$. (Either direction may be followed.) We consider the index $j$ of $b_{ij}$ modulo $2n_i$. Let $c_{ij}$ be the crossing of $D$ at which we pass from $b_{ij}$ to $b_{i(j+1)}$, and let $a_{ij}$ be the overpassing arc at $c_{ij}$. 

\begin{definition}
\label{long}
If $D$ is an even diagram of $L$ then the \emph{longitudes} in  $M_A(L)_\nu$ are $\lambda_1, \dots, \lambda _ \mu$, where
\[
\lambda_ i = \sum\limits _{j=1}^{2n_i} (-1)^j s_D(a_{ij}) .
\]
\end{definition}

Here are some properties of the longitudes. 

\begin{theorem}
\label{longprops}
\begin{enumerate}
    \item For every $i \in \{1, \dots, \mu \}$, $\lambda_i \in \ker w_ \nu$ and $2 \lambda_i = 0$.
    \item If $\mu=1$, then $\lambda_1=0$.
    \item The subgroup of $\ker w_ \nu$ generated by $\{\lambda_1, \dots, \lambda_\mu\}$ is generated by $\mu-1$ longitudes.
    \item If $\mu>1$, then $\det L = 0$ if and only if there is a proper subset $\{i_1, \dots, i_t\} \subset \{1, \dots, \mu\}$ such that 
    $\sum _{s=1}^t \lambda_{i_s}=0$.
    \item If $\mu=2$, then $\lambda_1 = \lambda _2$.
\end{enumerate}
\end{theorem}
\begin{proof} Let $D$ be an even diagram of $L$.

For item 1, note that if $1 \leq i \leq \mu$, then $w_ \nu (\lambda_i)=\sum (-1)^j=0$. Also, according to Definition \ref{almod}, every $j \in \{1, \dots, 2n_i \}$ has $2s_D(a_{ij})-s_D(b_{ij})-s_D(b_{i(j+1)})=0$. It follows that
\[
2\lambda_ i =  \sum\limits _{j=1}^{2n_i} (-1)^j \cdot 2 s_D(a_{ij})  = \sum\limits _{j=1}^{2n_i} (-1)^j (s_D(b_{ij})+s_D(b_{i(j+1)}))
\]
\[
=\sum\limits _{j=1}^{2n_i} ((-1)^j+(-1)^{j-1}) s_D(b_{ij}) =0.
\]

For item 2, recall that according to Corollary \ref{kstruc}, if $\mu=1$ then $\ker w_\nu$ is a finite abelian group of odd order. As $2 \lambda_1=0$, it follows that $\lambda_1=0$. 

For item 3, recall that Corollary \ref{kstruc} tells us
\begin{equation}
\label{group}
\ker w_{\nu} \cong \mathbb{Z}^{r-1}
\oplus \mathbb{Z}_{2^{n_1}}
\oplus \dots \oplus \mathbb{Z}_{2^{n_k}} \oplus B \text{,}
\end{equation}
where $|B|$ is an odd integer. There are $k$ elements of order 2 evident in (\ref{group}); they correspond to $(r+k)$-tuples with only one nonzero coordinate, equal to $2^{n_j-1}$ in $\mathbb Z _ {2^{n_j}}$. Every other element of order 2 in $\ker w_ \nu$ is obtained by adding together some of these $k$ elements. Therefore the set of elements of order $\leq 2$ in $\ker w_ \nu$ is a vector space of dimension $k$ over the field of two elements, $\textup{GF}(2)$. Item 3 follows, because the nonzero longitudes have order $2$ and $\mu -1 = r-1+k \geq k$.

One direction of item 4 is now easy to prove. If $\det L = 0$, then Proposition \ref{kert} tells us that $r-1 \geq 1$ in (\ref{group}), and hence $k \leq \mu-2$. It follows that the set of elements of order $\leq 2$ in $\ker w_ \nu$ is a vector space of dimension $\leq \mu-2$ over $\textup{GF}(2)$, so $\lambda_1, \dots, \lambda_{\mu-1}$ are linearly dependent over $\textup{GF}(2)$. As $1$ is the only nonzero scalar in $\textup{GF}(2)$, it follows that some of $\lambda_1, \dots, \lambda_{\mu-1}$ must add up to $0$.

Proving the other direction of item 4 is more difficult. Suppose $\mu>1$, $\{i_1, \dots, i_t\}$ is a proper subset of $\{1, \dots, \mu\}$, and $\sum_{s=1}^t \lambda_{i_s}=0$ in $M_A(L)_\nu$. Let $n=|A(D)|$; as $D$ is even, $n=|C(D)|$ too. Choose any $i_0 \in \{1, \dots, \mu\} - \{i_1, \dots, i_t\}$. Then according to Proposition \ref{redundant}, the row of $R_D$ corresponding to the crossing $c_{i_0 0}$ is redundant. Let $S_D$ be the $(n-1) \times n$ matrix obtained from $R_D$ by removing this redundant row. As $R_D$ and $S_D$ are both $n$-column presentation matrices for the abelian group $M_A(L)_\nu$, the theory of finitely generated abelian groups tells us that the determinants of $(n-1) \times (n-1)$ submatrices of $S_D$ generate the same ideal in $\mathbb Z$ as the determinants of $(n-1) \times (n-1)$ submatrices of $R_D$. As discussed in Sec.\ \ref{kerw}, the greatest common divisor of these determinants is $| \det L|$.

For each $i \in \{1, \dots, \mu\}$, let $\rho_i$ be the $1 \times n$ matrix corresponding to $\lambda_i$. That is, the entries of $\rho_i$ include $(-1)^j$ in the column corresponding to $a_{ij}$, for $1 \leq j \leq 2n_i$, and $0$ elsewhere. (If the same arc occurs as $a_{ij}$ for several values of $j$, the corresponding $(-1)^j$ values are added together.) As $\sum_{s=1}^t \lambda_{i_s}=0$ in $M_A(L)_\nu$, there is a linear combination of rows of $S_D$, with integer coefficients, whose sum is $\rho = \sum_{s=1}^t \rho_{i_s}$. Let us denote this linear combination $\Sigma$. Then $2\Sigma$ is a linear combination of rows of $S_D$ whose sum is $2\rho$, in which the coefficient of every row is an even integer.

If $1 \leq i \leq \mu$ and $1 \leq j \leq 2n_i$, then according to Definition \ref{almod}, the row of $R_D$ corresponding to the crossing $c_{ij}$ has nonzero entries equal to $2$ in the $a_{ij}$ column, $-1$ in the $b_{ij}$ column and $-1$ in the $b_{i(j+1)}$ column. (If two of these arcs are the same, the corresponding entries are added together.)  It follows that for a given $s \in \{1, \dots, t \}$, if we multiply the $c_{i_sj}$ row by $(-1)^j$ for each $j \in \{ 1, \dots, 2n_{i_s} \}$, then the sum of the resulting linear combination of row vectors is $2 \rho_{i_s}$. Adding these linear combinations together, for $s = 1, \dots, t$, we obtain a linear combination $\Sigma'$ of rows of $S_D$ whose sum is $2\rho$, in which every nonzero row coefficient is $\pm 1$. It is evident that $2\Sigma$ and $\Sigma'$ are different linear combinations, because the coefficient of every row in $2\Sigma$ is even. It follows that $2\Sigma - \Sigma'$ is a nontrivial linear combination of rows of $S_D$, whose sum is $2\rho-2\rho=0$. Therefore the $n-1$ rows of $S_D$ are linearly dependent, so every $(n-1) \times (n-1)$ submatrix of $S_D$ has determinant $0$. The determinant of $L$ is the greatest common divisor of the determinants of these $(n-1) \times (n-1)$ submatrices, so $\det L = 0$.

For item 5, suppose $\mu=2$. The isomorphism (\ref{group}) still holds, with $r-1+k=1$. It follows that $\ker w_\nu$ does not have more than one element of order $2$, so if $\lambda_1 \neq 0 \neq \lambda_2$ then $\lambda_1 = \lambda_2$. If $\lambda_1=0$ or $\lambda_2=0$ then $\det L = 0$, by item 4, and Proposition \ref{kert} tells us that $\ker w_\nu$ is infinite. Then $r-1=1$, so $k=0$. That is, $\ker w_v$ has no element of order $2$. Then $\lambda_1$ and $\lambda_2$ must both equal $0$. \end{proof}

\section{Two examples}
\label{sec:twoexamples}

In this section, we provide two examples to illustrate the ideas of Secs.\ \ref{sec:linkq} -- \ref{sec:longitudes}. These examples also serve to verify that $2 \centernot\implies  1$ in Theorem \ref{main2}.

\subsection{The connected sum of two Hopf links} \label{sec:twohopf}

Our first example is the link $L$ pictured in Fig.\ \ref{2hfig}. The involutory (non-medial) quandle of $L$ was analyzed by Winker \cite{W}.

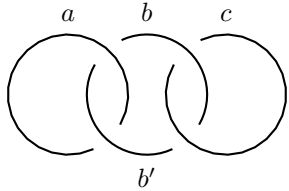
\begin{figure} [bht]
\centering
\begin{tikzpicture} 
\draw [thick, domain=-30:295] plot ({-5+(0.8)*cos(\x)}, {(0.8)*sin(\x)});
\draw [thick, domain=150:295] plot ({-3.95+(0.8)*cos(\x)}, {(0.8)*sin(\x)});
\draw [thick, domain=-30:115]  plot ({-3.95+(0.8)*cos(\x)}, {(0.8)*sin(\x)});
\draw [thick, domain=150:475] plot ({-2.9+(0.8)*cos(\x)}, {(0.8)*sin(\x)});
\node at (-5,1.05) {$a$};
\node at (-3.95,1.1) {$b$};
\node at (-2.9,1.05) {$c$};
\node at (-3.95,-1.1) {$b'$};
\end{tikzpicture}
\caption{$L$ is the connected sum of two Hopf links.}
\label{2hfig}
\end{figure}

The group $\IMG(L)$ is generated by $g_a,g_b,g_{b'}$ and $g_c$, with relations $x^2=1$ and $xyz=zyx$ for all elements $x,y,z$ that are conjugates of $g_a,g_b,g_{b'}$ and $g_c$. There are also four crossing relations: $g_a=g_{b'}g_ag_{b'}$, $g_b=g_ag_{b'}g_a$, $g_b=g_cg_{b'}g_c$ and $g_c=g_bg_cg_b$. The relation $g_a=g_{b'}g_ag_{b'}$ implies $g_{b'}g_a=g_ag_{b'}$, so $g_a$ and $g_{b'}$ commute; it follows that $g_b=g_ag_{b'}g_a=g_{b'}g^2_a=g_{b'}$. Also, $g_c=g_bg_cg_b$ implies $g_bg_c=g_cg_b$, so $g_b$ and $g_c$ commute. Furthermore, since $g_b=g_{b'}$ the other relations imply
\[
g_a g_c = g_bg_ag_bg_bg_cg_b = g_bg_a g_b^2 g_cg_b =  g_b(g_ag_cg_b) = g_b(g_bg_cg_a) = g_b^2g_cg_a = g_cg_a \text{,}
\]
so $g_a$ and $g_c$ commute.

We conclude that $\IMG(L)$ is a commutative group with eight elements, the various products $g_a^i g_b^j g_c^k$ with $i,j,k \in \{0,1\}$. It follows that  $\QIMG(L)=\{g_a,g_b,g_c\}$ is a trivial quandle, i.e.\ $x \triangleright y = x$ $\forall x,y \in \QIMG(L)$.

It takes only a little more work to describe $\IMQ(L)$. Notice first that $\IMG^2(L)$ is the subgroup of $\IMG(L)$ that includes the products $g_a^i g_b^j g_c^k$ with $i,j,k \in \{0,1\}$ and $i+j+k \in \{0,2\}$; hence $|\IMG^2(L)|=4$. According to part 4 of Corollary \ref{imh}, $\beta:\IMG^2(L) \to \Dis(\IMQ(L))$ is a surjective homomorphism. It follows that $\Dis(\IMQ(L)) = \{1,d_1,d_2,d_3\}$ where $d_1=\beta_{q_a} \beta_{q_b}=\beta_{q_b} \beta_{q_a}$, $d_2 = \beta_{q_a} \beta_{q_c}=\beta_{q_c} \beta_{q_a}$ and $d_3 = \beta_{q_b} \beta_{q_c}=\beta_{q_c} \beta_{q_b}$. Corollary \ref{imh} does not guarantee that these displacements are distinct, but it does guarantee that $|\Dis(\IMQ(L))|$ is a divisor of $4$.

As $g_b=g_{b'} \in \IMG(L)$, $\beta_{q_b}=\beta(g_b)=\beta(g_{b'}) =\beta_{q_{b'}} \in \Aut(\IMQ(L))$. Hence
\[
d_1(q_a)=\beta_{q_b} \beta_{q_a}(q_a)=\beta_{q_b}(q_a)=\beta_{q_{b'}}(q_a)=q_a \triangleright q_{b'}=q_a\text{,}
\]
so either (a) $d_1=1$ and hence $|\Dis(\IMQ(L))|\in \{1,2\}$, or (b) the stabilizer $S_{q_a}$ is a nontrivial subgroup of $\Dis(\IMQ(L))$. Either way, Proposition \ref{invorb} tells us that the orbit of $q_a$ in $\IMQ(L)$ has no more than two elements.

Similarly, $d_2(q_b)= \beta_{q_a} \beta_{q_c}(q_b)=\beta_{q_a} (q_{b'})=q_b$ and $d_3(q_c)=\beta_{q_b} \beta_{q_c}(q_c)=\beta_{q_b}(q_c)=q_c$ imply that the orbits of $q_b$ and $q_c$ in $\IMQ(L)$ have no more than two elements apiece. 

Notice that $\beta_{q_b}=\beta_{q_{b'}}$ fixes at least one element in each orbit: $\beta_{q_{b'}}(q_a)=q_a$, $\beta_{q_{b}}(q_b)=q_b$, and $\beta_{q_{b}}(q_c)=q_c$. As each orbit has no more than two elements, it follows that $\beta_{q_b}=\beta_{q_{b'}}$ is the identity map of $\IMQ(L)$.

To verify that each orbit of $\IMQ(L)$ does have two distinct elements, it suffices to exhibit an involutory medial quandle with the required properties. It is not hard to do so: the quandle has six distinct elements, $q_a, \widetilde q_a,q_b, q_{b'},q_c$ and $\widetilde  q_c$, and the quandle operation is given by the following permutations, in cycle notation: $\beta_{q_a}=\beta_{\widetilde q_a}=(q_b q_{b'})(q_c \widetilde  q_c)$, $\beta_{q_b}=\beta_{q_{b'}}=1$ and $\beta_{q_c}=\beta_{\widetilde q_c}= (q_a \widetilde q_a)(q_b q_{b'})$.

We take a moment to provide the (very easy) verification that this is a quandle. Let $Q$ be a set, given with a partition $P$ into subsets of cardinality 2. We say a transposition $(q_1q_2)$ of elements of $Q$ \emph{respects} $P$ if $\{q_1,q_2\} \in P$.

\begin{proposition}
\label{easyq}
Suppose that for each $q\in Q$, we are given a permutation $\beta_q:Q \to Q$. Suppose further that the following properties hold.
\begin{enumerate}
    \item For every $q \in Q$, $\beta_q(q)=q$.
    \item For every $q \in Q$, $\beta_q$ is the product of some transpositions that respect $P$.
    \item If $\{q_1,q_2\} \in P$, then $\beta_{q_1} = \beta_{q_2}$.
\end{enumerate}
Then $Q$ is an involutory medial quandle under the operation $p \triangleright q = \beta_q(p)$.
\end{proposition}
\begin{proof}
The idempotence and involutory properties follow from requirements 1 and 2, respectively. 

To verify the right distributive property, suppose $x,y,z \in Q$, and $\{x,x'\} \in P$. Then $(x \triangleright y) \triangleright z = x'$ if, and only if, the transposition $(xx')$ is included in precisely one of the permutations $\beta_y,\beta_z$. Otherwise, $(x \triangleright y) \triangleright z = x$. It follows that $(x \triangleright y) \triangleright z = (x \triangleright z) \triangleright y$. Requirement 2 implies that $y$ and $y \triangleright z$ are included in a single element of $P$, so requirement 3 implies that $\beta_y = \beta_{y \triangleright z}$. Hence $(x \triangleright y) \triangleright z = (x \triangleright z) \triangleright y = (x \triangleright z) \triangleright (y \triangleright z)$.

The medial property is verified in a similar way, using the fact that if $\{w,w'\} \in P$ then $(w \triangleright x) \triangleright (y \triangleright z)=w'$ if, and only if, $(ww')$ is included in precisely one of $\beta_x,\beta_y$.
\end{proof}

Notice that Corollary \ref{imgcor} and Proposition \ref{imqsemi} hold for $\QIMG(L)$, but they do not hold for $\IMQ(L)$. For instance, $x=g_ag_b \in \IMG^2(L)$ has $x^2= g_a^2 g_b^2 =1$, as $\IMG(L)$ is commutative. Of course $\beta(x)= \beta_{g_a} \beta_{g_b}=1 \in \Dis(\QIMG(L))$, as the only displacement of the trivial quandle $\QIMG(L)$ is the identity map. But $\beta(x) = \beta_{q_a} \beta_{q_b} = \beta_{q_a} \neq 1 \in \Dis(\IMQ(L))$. Moreover, $\IMQ(L)$ is not semiregular, as $\beta(x) = \beta_{q_a}$ fixes $q_a$.

The abelian group $M_A(L)_\nu$ is generated by $s_D(a),s_D(b),s_D(b')$ and $s_D(c)$. The crossing relations given by $r_D$ imply that $s_D(b)=s_D(b')$ and $2s_D(a)=2s_D(b)=2s_D(c)$. Therefore
\begin{equation}
\label{iso2}
M_A(L)_\nu \cong \mathbb Z \oplus \mathbb Z_2 \oplus \mathbb Z _2 \text{,}
\end{equation}
with the summands generated by $s_D(a)$, $s_D(b)-s_D(a)$ and $s_D(c)-s_D(a)$, respectively. The kernel of $w_\nu$ is the torsion subgroup of $M_A(L)_\nu$, and the map $\phi_\nu$ mentioned in the introduction is the isomorphism (\ref{iso2}). 

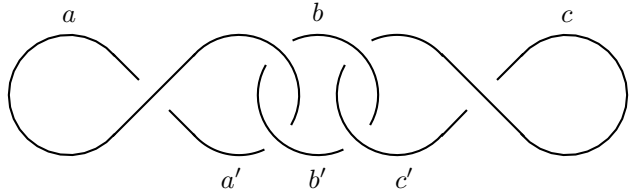
\begin{figure}  [bth]
\centering
\begin{tikzpicture} 
%left
\draw [thick, domain=-30:140] plot ({-5+(0.8)*cos(\x)}, {(0.8)*sin(\x)});
\draw [thick, domain=220:295] plot ({-5+(0.8)*cos(\x)}, {(0.8)*sin(\x)});
\draw [thick] (-5.58,0.55) -- (-6.68,-0.55);
\draw [thick] (-5.58,-0.55) -- (-5.93,-0.2);
\draw [thick] (-6.33,0.2) -- (-6.68,0.55);
\draw [thick, domain=40:320] plot ({-7.26+(0.8)*cos(\x)}, {(0.8)*sin(\x)});
%middle
\draw [thick, domain=150:295] plot ({-3.95+(0.8)*cos(\x)}, {(0.8)*sin(\x)});
\draw [thick, domain=-30:115]  plot ({-3.95+(0.8)*cos(\x)}, {(0.8)*sin(\x)});
%right
\draw [thick, domain=150:320] plot ({-2.9+(0.8)*cos(\x)}, {(0.8)*sin(\x)});
\draw [thick, domain=40:115] plot ({-2.9+(0.8)*cos(\x)}, {(0.8)*sin(\x)});
\draw [thick] (-2.31,0.55) -- (-1.21,-0.55);
\draw [thick] (-2.31,-0.55) -- (-1.97,-0.2);
\draw [thick] (-1.57,0.2) -- (-1.22,0.55);
\draw [thick, domain=-140:140] plot ({-0.64+(0.8)*cos(\x)}, {(0.8)*sin(\x)});
\node at (-7.26,1.05) {$a$};
\node at (-5.1,-1.1) {$a'$};
\node at (-3.95,1.1) {$b$};
\node at (-0.64,1.05) {$c$};
\node at (-2.8,-1.1) {$c'$};
\node at (-3.95,-1.1) {$b'$};
\end{tikzpicture}
\caption{An even diagram of $L$.}
\label{2hfigcurls}
\end{figure}

Consulting Fig.\ \ref{2hfigcurls}, we see that if we index the components $K_1,K_2,K_3$ in $a,b,c$ order, then the longitudes in $M_A(L)_\nu$ are $\lambda_1=s_D(b)-s_D(a)$, $\lambda_2=s_D(c)-s_D(a)$ and $\lambda_3=s_D(c)-s_D(b)$. These are the three nonzero elements of $\ker w_\nu$.  

\subsection {The link $6_3^3$}

Our second example is pictured in Fig. \ref{threechain}. As we will see, $6_3^3$ and the link $L$ of Sec.\ \ref{sec:twohopf} satisfy condition 2 of Theorem \ref{main2}, but they do not satisfy condition 1.

\begin{figure} [bht]
\centering
\begin{tikzpicture}
\draw [thick] (2,3) -- (0.5,3);
\draw [thick] (-2,3) -- (-0.5,3);
\draw [thick] (-2,3) -- (-2,0);
\draw [thick] (0.5,2) -- (-0.5,3);
\draw [thick] (0.1,2.6) -- (0.5,3);
\draw [thick] (-0.1,2.4) -- (-0.5,2);
\draw [thick] (0.5,1) -- (-0.5,2);
\draw [thick] (0.1,1.6) -- (0.5,2);
\draw [thick] (-0.1,1.4) -- (-0.5,1);
\draw [thick] (0.5,1) -- (0.5,0);
\draw [thick] (-0.5,1) -- (-0.5,.7);
\draw [thick] (-0.5,0) -- (-2,0);
\draw [thick] (-0.5,0) -- (-0.5,0.3);
\draw [thick] (-1.5,0.5) -- (0.3,0.5);
\draw [thick] (-1.5,0.5) -- (-1.5,0.2);
\draw [thick] (-1.5,-0.5) -- (-1.5,-0.2);
\draw [thick] (-1.5,-0.5) -- (1.5,-0.5);
\draw [thick] (1.5,0.5) -- (1.5,-0.5);
\draw [thick] (1.5,0.5) -- (0.7,0.5);
\draw [thick] (.5,0) -- (1.3,0);
\draw [thick] (2,0) -- (1.7,0);
\draw [thick] (2,0) -- (2,3);
\node at (-0.62,1.22) {$a'$};
\node at (-1.7,2.7) {$a$};
\node at (1.7,2.75) {$b'$};
\node at (-0.7,2) {$b$};
\node at (1.2,0.75) {$c$};
\node at (-1.2,0.75) {$c'$};
\end{tikzpicture}
\caption{The link $6_3^3$.}
\label{threechain}
\end{figure}
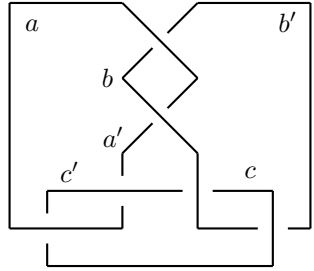

The group $\IMG(6_3^3)$ is generated by $g_a,g_{a'},g_b$, $g_{b'},g_c$ and $g_{c'}$. Relations include $x^2=1$ and $xyz=zyx$ for all elements $x,y,z$ that are conjugates of the listed generators, and the crossing relations. We use three of the crossing relations to eliminate $g_{a'},g_{b'}$ and $g_{c'}$:  $g_{a'}=g_bg_ag_b$, $g_{b'}=g_ag_bg_a$ and $g_{c'}=g_ag_cg_a$. The remaining crossing relations then imply $g_bg_ag_b = g_{c'}g_ag_{c'}=g_ag_cg_ag_cg_a$, $g_ag_bg_a=g_cg_bg_c$ and $g_ag_cg_a=g_bg_cg_b$. Then
\[
g_cg_b = g_bg_bg_cg_b = g_bg_ag_cg_a = g_ag_ag_bg_ag_cg_a = g_ag_cg_bg_cg_cg_a
\]
\[
= g_a(g_cg_bg_a) = g_a(g_ag_bg_c) = g^2_ag_bg_c = g_bg_c \text{,}
\]
so $g_b$ and $g_c$ commute. It follows that $g_ag_cg_a = g_bg_cg_b = g_cg^2_b = g_c$, so $g_cg_a = g_ag_ag_cg_a = g_ag_c$. 
Also,
\[
g_ag_b = g_ag_bg_ag_a = g_cg_bg_cg_a = g_c^2g_bg_a = g_bg_a.
\]
As $g_a,g_b$ and $g_c$ generate the group, we conclude that $\IMG(6_3^3)$ is commutative. It follows that $g_{a'} = g_a$, $g_{b'} = g_b$, $g_{c'} = g_c$ and $\QIMG(6^3_3)$ is the trivial quandle on the set $\{g_a,g_b,g_c\}$. Thus the link $L$ of Sec.\ \ref{sec:twohopf} has $\QIMG(L) \cong \QIMG(6^3_3)$.

The group $\IMG^2(6_3^3)$ is the subgroup of $\IMG(6_3^3)$ that includes the products $g_a^i g_b^j g_c^k$ with $i,j,k \in \{0,1\}$ and $i+j+k \in \{0,2\}$. Part 4 of Corollary \ref{imh} tells us that $\beta:\IMG^2(6_3^3) \to \Dis(\IMQ(6_3^3))$ is surjective, so $\Dis(\IMQ(6_3^3)) = \{1,\beta_{q_b} \beta_{q_a}, \beta_{q_c} \beta_{q_a}, \beta_{q_c} \beta_{q_b}\}$. 

One crossing relation of Fig.\ \ref{threechain} gives us $\beta_{q_c}\beta_{q_b}(q_a)=\beta_{q_c}(q_{a'})$. As $g_c=g_{c'}$ in $\IMG(6^3_3)$, $\beta_{q_{c'}} = \beta(g_{c'})=\beta(g_c)=\beta_{q_c}$, so $\beta_{q_c}(q_{a'}) = \beta_{q_{c'}}(q_{a'})=q_a$. It follows that the displacement $\beta_{q_c}\beta_{q_b}$ has $q_a$ as a fixed point. Therefore, either $\beta_{q_c}\beta_{q_b}=1$ or  $\beta_{q_c}\beta_{q_b}$ is a nontrivial element of the stabilizer $S_{q_a}$. In either case, Proposition \ref{invorb} tells us that the orbit of $q_a$ in $\IMQ(L)$ has no more than two elements.

In a similar way, the equalities $\beta_{q_a} \beta_{q_c}(q_b)=\beta_{q_a} (q_{b'})=q_b$ and $\beta_{q_b} \beta_{q_a}(q_c)=\beta_{q_b}(q_{c'})=q_c$ imply that the orbits of $q_b$ and $q_c$ in $\IMQ(L)$ have no more than two elements apiece. We conclude that $|\IMQ(6^3_3)| \leq 6$, with no more than two elements in any orbit.

To verify that each orbit of $\IMQ(6^3_3)$ does have two elements, consider the following permutations of the set $Q=\{q_a, q_{a'}, q_b,  q_{b'}, q_c, q_{c'}\}$: $\beta_{q_a}=\beta_{q_{a'}}=(q_b q_{b'})(q_c q_{c'})$, $\beta_{q_b}=\beta_{q_{b'}}=(q_a q_{a'})(q_c q_{c'})$ and $\beta_{q_c}=\beta_{q_{c'}}= (q_a q_{a'})(q_b q_{b'})$. Proposition \ref{easyq} tells us that these permutations define an involutory medial quandle on $Q$. As the crossing relations of Fig.\ \ref{threechain} are satisfied, and $|\IMQ(6^3_3)| \leq 6$, this quandle must be isomorphic to $\IMQ(6^3_3)$. No element $q$ has $\beta_q=1$, so this quandle is not isomorphic to the $\IMQ$ quandle of the link $L$ of Sec.\ \ref{sec:twohopf}. 

Analyzing the crossing relations of Fig.\ \ref{threechain}, we conclude that $M_A(6^3_3)_\nu \cong \mathbb Z \oplus \mathbb Z_2 \oplus \mathbb Z _2$, with the summands generated by $s_D(a)$, $s_D(b)-s_D(a)$ and $s_D(c)-s_D(a)$, respectively. This is precisely the same as the description of $M_A(L)_\nu$ in Sec.\ \ref{sec:twohopf}. In the terminology of Definition \ref{nueq}, $6^3_3$ and $L$ are $\phi_\nu$-equivalent. 

\section{The quandle $Q_A(L)_\nu$}
\label{2red}

The subquandle $Q_A(L)_\nu$ of $\textup{Core}(M_A(L)_\nu)$ was defined in the introduction, using the map $\phi_\nu:M_A(L)_\nu \to A_\mu$. Understanding the kernel of $\phi_\nu$ will help us understand $Q_A(L)_\nu$.

\begin{lemma}
\label{kerphi}
The kernel of $\phi_{\nu}$ is generated by $\{2(s_D(a)-s_D(a')) \mid a,a' \in A(D) \}$. That is, $\ker \phi_\nu = 2 \cdot \ker w_\nu$.
\end{lemma}
\begin{proof}
If $x \in \ker w_\nu$, then $\phi_{\nu}(x)$ is an element of 
\[
A_ \mu = \mathbb{Z} \oplus \underbrace{ \mathbb{Z}_2 \oplus \cdots \oplus \mathbb{Z}_2}_{\mu-1}
\] 
whose first coordinate is $0$, so $2 \cdot \phi_{\nu}(x) = 0$.  Thus $2 \cdot \ker w_\nu \subseteq  \ker \phi_{\nu}$.

Now, suppose $a_1,\dots,a_n \in A(D)$, $m_1,\dots,m_n \in \mathbb{Z}$ and
\[
x = \sum _{j=1}^n m_j s_D(a_j) \in \ker \phi_{\nu}.
\]
As $\phi_\nu(x)=0$, these two properties must hold:
\begin{align*}
& \text{(1)  } \sum_{j=1}^n m_j=0 .\\
& \text{(2)  For each } i \in \{1,\dots,\mu\} \text{, } 
\sum _{\kappa_D(a_j)=i}m_j \text{ is even.}
\end{align*}
If $m_1,\dots,m_n$ are all even, then property (1) implies that the sum can be written as a sum of terms of the form $2(s_D(a)-s_D(a'))$, so the lemma is satisfied.

The argument proceeds using induction on the number of odd coefficients $m_j$. Suppose $m_1$ is odd. Property (2) implies that there is an index $j>1$ such that $\kappa_D(a_j)=\kappa_D(a_1)$ and $m_j$ is odd; we may as well assume that $j=2$. Rewrite the sum so that $m_1=m_2=1$, and there are new summands (if necessary) to contribute $(m_1-1)s_D(a_1)$ and $(m_2-1)s_D(a_2)$.

If $a_1=a_2$ then combine terms, replacing $m_1s_D(a_1)+m_2s_D(a_2)=s_D(a_1)+s_D(a_2)$ with $2s_D(a_1)$; this reduces the number of odd coefficients. If $a_1 \neq a_2$, let $a'_1$ be an arc of $K_{\kappa_D(a_1)}$ that is separated from $a_1$ by a crossing. If $a_o$ is the overpassing arc that separates $a_1$ from $a'_1$, then according to the definition of $s_D$, $s_D(a_1)=2s_D(a_o)-s_D(a'_1)=2s_D(a_o)-2s_D(a'_1)+s_D(a'_1)$. Therefore we can replace $s_D(a_1)$ with $2s_D(a_o)-s_D(a'_1)=2s_D(a_o)-2s_D(a'_1)+s_D(a'_1)$ in the sum representing $x$, without increasing the number of odd coefficients. If $a'_1=a_2$, we can combine terms as in the first sentence of this paragraph, and reduce the number of odd coefficients. If $a'_1 \neq a_2$, let $a''_1$ be the arc of $K_{\kappa_D(a_1)}$ that is separated from $a'_1$ by a crossing, and is not equal to $a_1$. By the same process as before, we can replace the sum equal to $x$ with a sum that has precisely the same terms with odd coefficients, except that the summand $s_D(a'_1)$ has been replaced with $s_D(a''_1)$. Repeating this process, we walk along $K_{\kappa_D(a_1)}=K_{\kappa_D(a_2)}$ until we reach $a_2$, and then we combine terms to reduce the number of odd coefficients in the sum. \end{proof}

\begin{proposition}
\label{imq'}
If the determinant of $L$ is $0$, then $Q_A(L)_\nu$ is infinite. If the determinant of $L$ is not $0$, then
\[
|Q_A(L)_\nu|=\frac {\mu |\det L|}{2^{\mu-1}}.
\]
\end{proposition}
\begin{proof}
Recall that $A_{\mu} = \mathbb Z \oplus \mathbb Z_2 \oplus \cdots \oplus \mathbb Z_2$, 
with $\mu -1 $ factors of $\mathbb Z_2$. Let $T(A_{\mu})$ be the torsion subgroup of $A_{\mu}$, i.e., the set of elements whose first coordinate is $0$. As $w_{\nu}:M_A(L)_{\nu} \to \mathbb Z$ is the first coordinate of $\phi_{\nu}:M_A(L)_{\nu} \to A_{\mu}$,
\[
|\ker w_{\nu}| = | \phi_{\nu}^{-1}(T(A_{\mu})) | = |\ker \phi_{\nu}| \cdot |T(A_{\mu})|= |\ker \phi_{\nu}| \cdot 2^{\mu -1}.
\]
By Definition \ref{coreprime}, $Q_A(L)_\nu$ is the union of $\mu$ cosets of $\ker \phi_{\nu}$ in $M_A(L)_{\nu}$, so $
|Q_A(L)_\nu| = \mu \cdot |\ker \phi_{\nu}| = \mu  \cdot (|\ker w_{\nu}| / 2^{\mu-1})$. The proposition now follows from Proposition \ref{kert}.
\end{proof}

\begin{proposition}
\label{imq'dis}
There is an isomorphism $\delta:\ker \phi_\nu \to \Dis(Q_A(L)_\nu)$, defined as follows: If $k \in \ker \phi_\nu$, then $\delta(k)$ is the function $\delta(k):Q_A(L)_\nu \to Q_A(L)_\nu$ given by $\delta(k)(x)=k+x$ $\forall x \in Q_A(L)_\nu$.
\end{proposition}
\begin{proof}
If $k \in \ker \phi_\nu$, then it is obvious that $k+x \in Q_A(L)_\nu$ $\forall x \in Q_A(L)_\nu$. Therefore, there is certainly a function $\delta(k):Q_A(L)_\nu \to Q_A(L)_\nu$ defined as in the statement. 

Notice that $\{ \delta(k) \mid k \in \ker \phi_\nu \}$ is closed under composition: if $k,\ell \in \ker \phi_ \nu$, then  $k + \ell \in \ker \phi_ \nu$, and
\begin{equation}
\label{hom}
(\delta(k) \circ \delta(\ell))(x) = \delta(k)(\ell+x) = k+ \ell + x = \delta(k+\ell)(x) \quad \forall x  \in Q_A(L)_\nu \text{.}
\end{equation}
The identity map of $Q_A(L)_\nu$ is $\delta(0)$, and if $k \in \ker \phi_\nu$ then $\delta(k)^{-1}=\delta(-k)$. We see that $\{ \delta(k) \mid k \in \ker \phi_\nu \}$ is a group under composition. The equation (\ref{hom}) shows that $\delta(k+\ell)=\delta(k) \circ \delta(\ell)$ $\forall k,\ell \in \ker \phi_\nu$, so we have a homomorphism $\delta:\ker \phi_\nu \to \{ \delta(k) \mid k \in \ker \phi_\nu \}$. It is obvious that $\delta$ is surjective. It is also injective: if $\delta(k)$ is the identity map, then $0=\delta(k)(0)=k+0$, so $k=0$.

To complete the proof, it suffices to show that $\{ \delta(k) \mid k \in \ker \phi_\nu \}=\Dis(Q_A(L)_\nu)$. 

First, suppose $y,z \in Q_A(L)_\nu$. Then $w_\nu(y) = w_ \nu(z)=1$, so $y-z \in \ker w_ \nu$. The elementary displacement $\beta_y \beta_z$ is given by 
\[
\beta_y \beta_z(x) = 2y - \beta_z(x) = 2y-(2z-x)=2(y-z)+x \quad \forall x \in Q_A(L)_\nu \text{.}
\]
Lemma \ref{kerphi} tells us $2(y-z) \in \ker \phi_ \nu$, so $\beta_y \beta_z =\delta(2(y-z))$. As $\{ \delta(k) \mid k \in \ker \phi_\nu \}$ is closed under composition and contains all the elementary displacements of $Q_A(L)_\nu$, $\{ \delta(k) \mid k \in \ker \phi_\nu \} \supseteq \Dis(Q_A(L)_\nu)$.

To verify the opposite inclusion $\{ \delta(k) \mid k \in \ker \phi_\nu \} \subseteq \Dis(Q_A(L)_\nu)$, suppose $k \in \ker \phi_ \nu$. According to Lemma \ref{kerphi}, there are $a_1,\dots,a_n,a'_1,\dots,a'_n \in A(D)$ such that
\[
k = 2s_D(a_1-a'_1)+2s_D(a_2-a'_2) + \cdots + 2s_D(a_n-a'_n). 
\]
It follows that $\delta(k)$ is a composition of $n$ elementary displacements:
\[
\delta(k)(x)= 2s_D(a_1)-2s_D(a'_1)+- \cdots + 2s_D(a_n)-2s_D(a'_n) + x
\]
\[
=\beta_{s_D(a_1)}\beta_{s_D(a'_1)} \cdots \beta_{s_D(a_n)}\beta_{s_D(a'_n)}(x)
\]
\[
=(\beta_{s_D(a_1)}\beta_{s_D(a'_1)}) \cdots (\beta_{s_D(a_n)}\beta_{s_D(a'_n)})(x) \quad \forall x \in Q_A(L)_\nu.
\]
Therefore $\{ \delta(k) \mid k \in \ker \phi_\nu \} = \Dis(Q_A(L)_\nu)$, as required. \end{proof}

For later reference, we extract two scholia from the proof of Proposition \ref{imq'dis}.

1. The equality $\beta_y \beta_z = \delta(2(y-z))$ implies that if $a,a' \in A(D)$, then $\beta_{s_D(a)} \beta_{s_D(a')} = \delta(2(s_D(a)-s_D(a')))$.

2. Every element of $\Dis(Q_A(L)_\nu)$ is $\beta_{s_D(a_1)}\beta_{s_D(a'_1)} \cdots \beta_{s_D(a_n)}\beta_{s_D(a'_n)}$ for some $a_1,\dots,a_n,a'_1,\dots,a'_n \in A(D)$.

\begin{proposition}
\label{imq'orb}
$Q_A(L)_\nu$ has $\mu$ orbits, one for each component of $L$. The orbit corresponding to $K_i$ is $\phi_ \nu ^{-1}(\phi_ \nu(s_D(a)))$ for every arc $a \in A(D)$ with $\kappa_D(a)=i$.
\end{proposition}
\begin{proof}
We begin with a claim: every orbit of $Q_A(L)_\nu$ contains $s_D(a)$ for some $a \in A(D)$. If $x \in Q_A(L)_\nu$ then according to Definition \ref{coreprime}, there is an $a \in A(D)$ with $x-s_D(a) \in \ker \phi_\nu$. Then 
\[
x=x-s_D(a)+s_D(a)= \delta(x-s_D(a))(s_D(a)) \text{,}
\]
where $\delta:\ker \phi_\nu \to \Dis(Q_A(L)_\nu)$ is the isomorphism of Proposition \ref{imq'dis}. It follows that $x$ and $s_D(a)$ are elements of the same orbit of $Q_A(L)_\nu$. This justifies the claim.

Now, suppose there is a crossing of $D$ with overpassing arc $a$ and underpassing arcs $b,b'$. Then $s_D(2a-b-b')=0$, so according to the definition of a core quandle, $s_D(b')=2s_D(a)-s_D(b)= s_D(b) \triangleright s_D(a)$ in $Q_A(L)_\nu$. Thus $s_D(b)$ and $s_D(b')$ are elements of the same orbit of $Q_A(L)_\nu$. This applies at every crossing of $D$, so for each component $K_i$ of $L$, a single orbit of $Q_A(L)_\nu$ contains $s_D(b)$ for every $b \in A(D)$ with $\kappa_D(b)=i$. 

On the other hand, if $x,y \in Q_A(L)_\nu$ then $\phi_{\nu}(x \triangleright y)=\phi_{\nu}(2y-x)=\phi_{\nu}(2y-2x)+\phi_{\nu}(x)=0+\phi_{\nu}(x)$, so $\phi_{\nu}$ is constant on each orbit in $Q_A(L)_\nu$. It follows that $\kappa_D$ is also constant on each orbit. \end{proof}

\section{Proof of Theorem \ref{main1}}
\label{proofmain}

We complete the proof of Theorem \ref{main1} in this section. Item 2 is useful in proving item 1, so we prove item 2 first.

\subsection{Item 2 of Theorem \ref{main1}}

\begin{proposition}
\label{ednew}
Let $D$ be a diagram of a link $L$, and let $a^*$ be a fixed arc of $D$. Then there is an epimorphism $e_D:\ker w_{\nu} \to \IMG^2(L)$ with $e_D(s_D(a)-s_D(a^{*})) = h_a=g_a g_{a^*}$ $\forall a \in A(D)$.
\end{proposition}
\begin{proof}
Let $R_D$ be the matrix representing $r_D:\mathbb{Z}^{C(D)} \to \mathbb{Z}^{A(D)}$. Let $R'_D$ be the matrix obtained from $R_D$ by adjoining a row whose only nonzero entry is a $1$ in the $a^*$ column, as in the proof of Proposition \ref{kert}. As noted there, $R'_D$ is a presentation matrix for $\ker w_{\nu}$. To be explicit: there is an epimorphism $s'_D: \mathbb{Z}^{A(D)} \to \ker w_{\nu}$ with $s'_D(a)=s_D(a)-s_D(a^*)$ $\forall a \in A(D)$, and the kernel of $s'_D$ is generated by the elements of $\mathbb{Z}^{A(D)}$ represented by the rows of $R'_D$.

According to Corollary \ref{imh}, $\IMG^2(L)$ is generated (as an abelian group written multiplicatively) by the elements $h_a=g_ag_{a^*}$, and whenever $c$ is a crossing of $D$ with overpassing arc $a$ and underpassing arcs $b,b'$, the formula $h_{b'}=h_a^2 h_b$ holds in $\IMG^2(L)$. This formula matches the element of $\mathbb{Z}^{A(D)}$ represented by the $c$ row of $R'(D)$, namely $r_D(c)=2a - b- b'$ (in additive notation). The one row of $R'_D$ that does not correspond to a crossing of $D$ is the row whose only nonzero entry is a $1$ in the $a^*$ column. As $h_{a^*}=1$ , the relation represented by this row is also valid in $\IMG^2(L)$.

It follows that there is a well-defined homomorphism of abelian groups $e_D:\ker w_{\nu} \to \IMG^2(L)$, with $e_D(s'_D(a)) = h_a$ $\forall a \in A(D)$. The group $\IMG^2(L)$ is generated by the $h_a$ elements, so $e_D$ is surjective.
\end{proof}

\begin{proposition}
\label{fed}
If $x \in \ker w_{\nu}$ and $2x=0$, then $e_D(x)$ is an element of the kernel of the map $\beta:\IMG^2(L) \to \Dis(\QIMG(L))$ mentioned in Proposition \ref{qimgbetahom} and Corollary \ref{qimgh}.
\end{proposition}
\begin{proof}
The operation of the group $\IMG^2(L)$ is written in multiplicative notation, so the hypothesis $2x=0$ implies that $e_D(x)^2=e_D(0)=1$. Corollary \ref{imgcor} then implies that $\beta(e_D(x))=1$.
\end{proof}

\begin{corollary}
\label{inverse}
Let $a^*$ be a fixed arc of $D$. Then there is an epimorphism $f_D:\ker \phi_{\nu} \to \Dis(\QIMG(L))$ with 
\[
f_D(2(s_D(a)-s_D(a^{*}))) = \beta(h_a)=\beta_{g_a} \beta_{ g_{a^*}}  \quad \forall a \in A(D).
\]
\end{corollary}
\begin{proof}
According to Corollary \ref{qimgh} and Proposition \ref{ednew}, $e_D:\ker w_\nu \to \IMG^2(L)$ and $\beta:\IMG^2(L) \to \Dis(\QIMG(L))$ are both epimorphisms, so the composition $\beta e_D:\ker w_\nu \to \Dis(\QIMG(L))$ is an epimorphism too.

There is also an epimorphism $\ker w_{\nu} \to 2 \cdot \ker w_{\nu}$, given by $x \mapsto 2x$ $\forall x \in \ker w_{\nu}$. The kernel of this epimorphism is $(\ker w_{\nu})(2)=\{x \in \ker w_\nu \mid 2x=0\}$. Proposition \ref{fed} tells us that $(\ker w_{\nu})(2) \subseteq \ker (\beta e_D)$, so there is an epimorphism $f _D:2 \cdot \ker w_{\nu} \to \allowbreak \Dis(IMQ(L))$ induced by $\beta e_D$. That is, $f_D(2x)=\beta e_D(x)$ $\forall x \in \ker w_{\nu}$. Then for every $a \in A(D)$,
\[
f_D(2(s_D(a)-s_D(a^{*}))) =\beta e_D(s _D(a) - s_D(a^*) ) = \beta(g_ag_{a^*}) =\beta_{g_a} \beta_{g_{a^*}}.
\]
Lemma \ref{kerphi} tells us that $2 \cdot \ker w_\nu = \ker \phi _\nu$, so the proposition follows.
\end{proof}
\begin{corollary}
\label{veryendcor}
If $L$ is any classical link, then the quandle map $\widehat s_D:\QIMG(L) \to Q_A(L)_\nu$ of Corollary \ref{mapping} is an isomorphism.
\end{corollary}
\begin{proof}
Recall the definition: $\widehat s_D(g_a)=s_D(a)$ $\forall a \in A(D)$. The image of $\widehat s_D$ is a subquandle of $Q_A(L)_\nu$, and it contains $s_D(a)$ for every $a \in A(D)$. According to the second scholium of the proof of Proposition \ref{imq'dis}, the $s_D(a)$ elements generate $Q_A(L)_\nu$, so the image of $\widehat s_D$ is the entire quandle $Q_A(L)_\nu$. We claim that $\widehat s _D$ satisfies the two requirements for a surjective quandle map to be an isomorphism given in Corollary \ref{surjiso}.

According to Propositions \ref{qimgorb} and \ref{imq'orb}, $\widehat s _D$ maps each orbit of $\QIMG(L)$ onto the orbit of $Q_A(L)_\nu$ corresponding to the same component of $L$. It follows that $\widehat s _D$ satisfies requirement (b) of Corollary \ref{surjiso}.

To show that $\widehat s _D$ satisfies requirement (a) of Corollary \ref{surjiso}, we must show that the induced epimorphism $\Dis(\widehat s _D):\Dis(\QIMG(L)) \to \Dis(Q_A(L)_\nu)$ is an isomorphism. Recall the definition: if $a,a^* \in A(D)$, then
\[
\Dis(\widehat s _D)(\beta_{g_a} \beta_{g_{a^*}})=\beta_{ \widehat s _D(g_a) } \beta_{\widehat s _D(g_{a^*})}=\beta_{s _D(a) } \beta_{s_D(a^*) }.
\]

Lemma \ref{kerphi} tells us $2 \cdot \ker w_{\nu} = \ker \phi_ \nu$. Proposition \ref{imq'dis} provides an isomorphism $\delta:\ker \phi_ \nu \to \Dis(Q_A(L)_\nu)$. As noted in the first scholium of the proof of Proposition \ref{imq'dis}, this isomorphism $\delta$ has
\[
\delta(2(s_D(a)-s_D(a^*))) = \beta_{s_D(a)} \beta_{s_D(a^*)} \quad \forall a,a^* \in A(D).
\]

We claim that the identity map of $\Dis(\QIMG(L))$ is equal to the composition $f _D \delta^{-1} \Dis(\widehat s _D)$. To verify the claim, note that if $a^* \in A(D)$ is a fixed element then for every $a \in A(D)$, 
\[
f _D \delta^{-1} \Dis(\widehat s _D)(\beta_{g_a} \beta_{g_{a^*}}) =f_D \delta^{-1} (\beta_{s _D(a) } \beta_{s_D(a^*)})
\]
\[
=f _D(2(s _D(a)-s_D(a^*)))
=\beta_{g_a} \beta_{g_{a^*}}.
\]
The elementary displacements $\beta_{g_a} \beta_{g_{a^*}}$ generate $\Dis(\QIMG(L))$, so the claim holds.

The claim implies that $\Dis(\widehat s _D)$ is injective, so requirement (a) of Corollary \ref{surjiso} is satisfied, and Corollary \ref{surjiso} tells us that $\widehat s _D$ is an isomorphism.
\end{proof}

Corollary \ref{veryendcor} gives us item 2 of Theorem \ref{main1}. 

\subsection{Item 1 of Theorem \ref{main1}}

Definitions \ref{img} and \ref{qimg} imply that the quandle $\QIMG(L)$ is generated by the $g_a$ elements, and $g_{b'}=g_b \triangleright g_a$ whenever $a,b,b'$ appear at a crossing of $D$ as pictured in Fig.\ \ref{crossfig}. It is easy to see that $\QIMG(L)$ is an involutory medial quandle, so Definition \ref{imq} immediately implies that there is a surjective quandle map $\IMQ(L) \to \QIMG(L)$, with $q_a \mapsto g_a$ $\forall a \in A(D)$.

Suppose $\mu=1$. Then $L$ is a knot, so we denote it $K$. The surjection $\IMQ(K) \to \QIMG(K)$ and the isomorphism of item 2 imply that $|\IMQ(K)| \geq |\QIMG(K)| = |Q_A(K)_\nu|$. By Proposition \ref{imq'}, it follows that $|\IMQ(K)| \geq |\det K|$. On the other hand, Corollary \ref{imh} and Proposition \ref{ednew} provide an epimorphism $\beta e_D:\ker w_\nu \to \Dis(\IMQ(K))$, so according to Proposition \ref{kert}, $|\det K|=|\ker w_\nu| \geq |\Dis(\IMQ(K))|$. (N.b.\ As $K$ is a knot, $\det K$ is an odd integer; in particular, $\det K \neq 0$.) The quandle $\IMQ(K)$ has only one orbit, so $|\Dis(\IMQ(K))| \geq |\IMQ(K)|$. Combining inequalities, we conclude that $\IMQ(K)$ and $\QIMG(K)$ are both finite quandles of cardinality $|\det K|$. It follows that the surjective quandle map $\IMQ(K) \to \QIMG(K)$ must be an isomorphism.

It remains to verify the assertion of item 1 regarding two-component links of nonzero determinant. We begin with two more general results.

\begin{proposition}
\label{longfix}
Let $D$ be an even diagram of $L=K_1 \cup \dots \cup K_ \mu$, and let $i \in \{1, \dots, \mu \}$. The displacement $\beta e_D(\lambda_i) \in \Dis(\IMQ(L))$ fixes every element of the $K_i$ orbit of $\IMQ(L)$.
\end{proposition}
\begin{proof}
Recall first that according to Proposition \ref{evens}, assuming that $D$ is even does not involve a significant loss of generality. As in Sec.\ \ref{sec:longitudes}, let $b_{i0}, \dots, b_{i(2n_i)}=b_{i0}$ be the arcs of $K_i$ in $D$, indexed in order as we walk along $K_i$. Also, let $c_{ij}$ be the crossing of $D$ at which we pass from $b_{ij}$ to $b_{i(j+1)}$ as we walk along $K_i$, and let $a_{ij}$ be the overpassing arc at $c_{ij}$. 

Now, let $a^*$ be a fixed arc of $D$. Then $g_{a^*}^2=1$ in $\IMG^2(L)$, so
\[
\beta e_D(\lambda_i) = \beta e_D \bigg( \sum\limits _{j=0}^{2n_i-1} (-1)^j s_D(a_{ij}) \bigg) = \beta e_D \bigg( \sum\limits _{j=0}^{2n_i-1} (-1)^j (s_D(a_{ij})-s_D(a^*)) \bigg)
\]
\[
= \beta \bigg( \prod\limits_{j=0}^{2n_i-1} (g_{a_{ij}}g_{a^*})^{(-1)^{j}} \bigg)=\beta((g_{a_{i0}}g_{a^*})(g_{a^*}g_{a_{i1}}) \cdots (g_{a^*}g_{a_{i(2n_i-1)}}))
\]
\[
= \beta(g_{a_{i0}}g_{a^*}^2g_{a_{i1}} \cdots g_{a^*}^2 g_{a_{i(2n_i-1)}})) = \beta \bigg( \prod\limits_{j=0}^{2n_i-1} g_{a_{ij}} \bigg)= \prod\limits_{j=0}^{2n_i-1} \beta_{q_{a_{ij}}} \text{.}
\]
For each value of $j$, $c_{ij}$ is a crossing at which $a_{ij}$ separates $b_{ij}$ from $b_{i(j+1)}$, so $\beta_{q_{a_{ij}}}(q_{b_{i(j+1)}})=q_{b_{i(j+1)}}\triangleright q_{a_{ij}} = q_{b_{ij}}$. Therefore
\[
(\beta e_D(\lambda_i))(q_{b_{i0}}) = \bigg (\prod\limits_{j=0}^{2n_i-1} \beta_{q_{a_{ij}}} \bigg ) (q_{b_{i(2n_i)}})= \bigg( \prod\limits_{j=0}^{2n_i-2} \beta_{q_{a_{ij}}} \bigg) \big( \beta_{q_{a_{i(2n_i-1)}}} (q_{b_{i(2n_i)}}) \big)
\]
\[
= \bigg( \prod\limits_{j=0}^{2n_i-2} \beta_{q_{a_{ij}}} \bigg) \big (q_{b_{i(2n_i-1)}} \big) = \cdots =  \beta_{q_{a_{i0}}} (q_{b_{i1}})=q_{b_{i0}}.
\]

As $\Dis(\IMQ(L))$ is commutative, it follows that every $d \in \Dis(\IMQ(L))$ has $\beta e_D(\lambda_i)(d(q_{b_{i0}}))=d(\beta e_D(\lambda_i)(q_{b_{i0}}))= d(q_{b_{i0}})$. Every element of the orbit of $q_{b_{i0}}$ is $d(q_{b_{i0}})$ for some $d \in \Dis(\IMQ(L))$, so the proposition is proven.
\end{proof}

\begin{corollary}
\label{orbcard}
Let $L$ be a link of $\mu\geq 2$ components, with $\det L \neq 0$. Then each orbit of $\IMQ(L)$ contains no more than $|\det L|/2$ elements.
\end{corollary}
\begin{proof} Let the arcs of an even diagram $D$ be indexed as in Proposition \ref{longfix}, and suppose $i \in \{1, \dots, \mu \}$. According to Proposition \ref{invorb}, the $K_i$ orbit of $\IMQ(L)$ has the same number of elements as the quotient $\Dis(\IMQ(L))/S$, where $S=\{d \in \Dis(\IMQ(L)) \mid d(q_{b_{i0}})= q_{b_{i0}}\}$. 

Proposition \ref{longfix} tells us that $\beta e_D(\lambda_i)$ fixes $q_{b_{i0}}$, so either (i) $\beta e_D(\lambda_i)$ is a nontrivial element of the subgroup $S$, or (ii) $\lambda_i \in \ker(\beta e_D)$. If (i) holds, then $|S| \geq 2$. The map $\beta e_D: \ker w_\nu \to \Dis(\IMQ(L))$ is surjective, so the cardinality of the $K_i$ orbit is $|\Dis(\IMQ(L))|/|S|\leq |\Dis(\IMQ(L))|/2 \leq |\ker w_\nu|/2 = |\det L|/2$. If (ii) holds, recall that according to Proposition \ref{longprops}, the hypothesis $\det L \neq 0$ implies that $\lambda_i \neq 0$; thus $|\ker(\beta e_D)| \geq 2$. The map $\beta e_D: \ker w_\nu \to \Dis(\IMQ(L))$ is surjective, so $|\Dis(\IMQ(L))| = |\ker w_\nu|/|\ker(\beta e_D)| \leq |\ker w_\nu|/2 = |\det L|/2$. \end{proof}

Now, suppose $\mu=2$ and $\det L \neq 0$. There are two orbits in $\IMQ(L)$, so Corollary \ref{orbcard} implies that $|\IMQ(L)| \leq |\det L|$. The isomorphism of item 2 of Theorem \ref{main1} implies that $|\QIMG(L)|=|Q_A(L)_\nu|=|\det L|$, so the surjective quandle map $\IMQ(L) \to \QIMG(L)$ must be an isomorphism.

To complete the proof of item 1 of Theorem \ref{main1}, observe that the three-component link of Sec.\ \ref{sec:twohopf} has $|\IMQ(L)|=6$ and $|\QIMG(L)|=3$.

\subsection{Item 3 of Theorem \ref{main1}}

As discussed in Sec.\ \ref{kerw}, $H_1(X_2) \cong \ker w_\nu$. Therefore, we may verify item 3 of Theorem \ref{main1} for $\ker w_\nu$ rather than $H_1(X_2)$. Recall that $A_\mu = \mathbb Z \oplus \mathbb Z_2^{\mu-1}$, and $w_\nu:M_A(L)_\nu \to \mathbb Z$ is the first coordinate of $\phi_\nu:M_A(L)_\nu \to A_\mu$. 

If $a^*$ is any fixed element of $A(D)$, then $w_\nu(s_D(a^*)) = 1 = w_\nu(x)$ $\forall x \in Q_A(L)_\nu$, so $x-s_D(a^*) \in \ker w_\nu$ $\forall x \in Q_A(L)_\nu$. It follows that there is a function $g:Q_A(L)_\nu \to \ker w_\nu$, defined by $g(x) = x-s_D(a^*)$. It is easy to see that $g$ is injective. Also, $g$ is a quandle map into $\Core(\ker w_\nu)$: 
\[
g(x \triangleright y) = g(2y-x) = 2y-x-s_D(a^*) = 2(y-s_D(a^*))-(x-s_D(a^*)) = g(x) \triangleright g(y).
\]

If $\mu=1$, then $Q_A(L)_\nu = \phi_{\nu}^{-1}(\phi_{\nu}(s_D(A(D)))) = \phi_\nu ^{-1}(\{1\}) =w_\nu ^{-1}(\{1\}) =\{x+s_D(a^*) \mid x \in \ker w_\nu \}$, so $g$ is surjective. If $\mu=2$, then $Q_A(L)_\nu = \phi_{\nu}^{-1}(\phi_{\nu}(s_D(A(D)))) = \phi_\nu ^{-1}(\{(1,0),(1,1)\})=w_\nu ^{-1}(\{1\}) =\{x+s_D(a^*) \mid x \in \ker w_\nu \}$. Again, it follows that $g$ is surjective.

Now, suppose $\mu>2$. Then according to Corollary \ref{kstruc},
\[
\ker w_\nu \cong \mathbb{Z}^{r-1}
\oplus \mathbb{Z}_{2^{n_1}}
\oplus \dots \oplus \mathbb{Z}_{2^{n_k}} \oplus B \text{,}
\]
where $r \in \{1, \dots, \mu \}$, $r+k=\mu$ and $|B|$ is an odd integer. We think of an element $x \in \ker w_\nu$ as a $\mu$-tuple $(x_1, \dots, x_\mu)$, with $x_\mu \in B$. Notice that then $x \triangleright y = (2y_1-x_1, \dots , 2y_\mu - x_\mu)$ and for $1 \leq i \leq \mu-1$, $2y_i-x_i$ has the same parity (mod 2) as $x_i$. (As $|B|$ is odd, ``parity (mod 2)'' is meaningless when $i=\mu$.) Therefore, every element of the orbit of $x$ in $\textup{Core}(\ker w_\nu)$ is a $\mu$-tuple with the same pattern of parities (mod 2) in its first $\mu-1$ coordinates. There are $2^{\mu-1}$ different patterns of parities, so there are $2^{\mu-1}$ different orbits in $\textup{Core}(\ker w_\nu)$. There are only $\mu$ different orbits in $Q_A(L)_\nu$, and the hypothesis $\mu >2$ implies $\mu < 2^{\mu-1}$. It follows that there cannot be a surjective quandle map $Q_A(L)_\nu \to \Core(\ker w_\nu)$.

\section{Proof of Theorem \ref{main3}}
\label{sec:proof3}

If $\mu=1$, Theorem \ref{main1} tells us that the quandles $\IMQ(L)$, $\QIMG(L)$, $Q_A(L)_\nu$ and $\textup{Core}(\ker w_\nu)$ are all isomorphic. According to Proposition \ref{kert}, their common cardinality is $|\ker w_\nu| = |\det L|$.

If $\det L = 0$, Theorem \ref{main1} provides surjective maps $\IMQ(L) \to \QIMG(L) \to Q_A(L)_\nu$, and an injective map $Q_A(L)_\nu \to \Core(\ker w_\nu)$. Proposition \ref{imq'} tells us that $Q_A(L)_\nu$ is infinite, so all of these quandles are infinite.

If $\mu>1$ and $\det L \neq 0$, Corollary \ref{orbcard} tells us that each orbit of $\IMQ(L)$ has no more than $|\det L|/2$ elements. There are $\mu$ orbits, so $|\IMQ(L)| \leq \mu |\det L|/2$. Theorem \ref{main1} provides a surjection $\IMQ(L) \to \QIMG(L)$ and a bijection $\QIMG(L) \to Q_A(L)_\nu$, so $|\IMQ(L)| \geq |\QIMG(L)| = |Q_A(L)_\nu|$. Proposition \ref{imq'} tells us that $|Q_A(L)_\nu| = \mu |\det L| / 2^{\mu -1}$.

\section{Using $Q_A(L)_\nu$ to construct $M_A(L) _\nu$}
\label{proof1}

In this section, we show that $Q_A(L)_\nu$ provides a presentation of the abelian group $M_A(L)_\nu$. The presentation also determines the map $\phi_\nu$, up to permutations of the components of $L=K_1 \cup \dots \cup K_\mu$. In the next two sections, we verify similar results for two other quandles. 

For convenience, we temporarily use $Q$ to denote $Q_A(L)_\nu$. This notation is used only in this section, until the end of the proof of Proposition \ref{kerf}.

Let $\mathbb Z ^Q$ be the free abelian group on the set $Q$, and let $f:\mathbb Z ^Q \to M_A(L)_\nu$ be the homomorphism that sends each $q \in Q$ (considered as a free generator of $\mathbb Z ^Q$) to itself (considered as an element of $M_A(L)_\nu$). Also, let $g:Q \to \mathbb Z ^Q$ be the function that sends each $q \in Q$ (considered as an element of $Q_A(L)_\nu$) to itself (considered as a free generator of $\mathbb Z ^Q$).

\begin{lemma}
\label{klem}
Let $D$ be a diagram of $L$, and let $K$ be the subgroup of $\mathbb Z ^Q$ generated by the subset $\{ 2g(y)-g(x)-g(2y-x) \mid x,y \in Q \}$.
Then for every $x \in \mathbb Z ^Q$, there are arcs $a_1, \dots, a_n \in A(D)$ and integers $m_1, \dots, m_n \in \mathbb Z$ such that
\[
x - \sum _{i=1}^n m_i g(s_D(a_i)) \in K.
\]
\end{lemma}
\begin{proof}
It suffices to verify the lemma for an element $x=g(q)$, where $q \in Q$. According to Definition \ref{coreprime}, there is an $a \in A(D)$ with $\phi_\nu(q)=\phi_ \nu(s_D(a))$, so $q - s_D(a) \in \ker \phi_ \nu$. According to Lemma \ref{kerphi}, it follows that there are arcs $a_1, \dots, a_{2p} \in A(D)$ such that 
\[
q - s_D(a) = 2 \cdot \sum_{j=1}^{2p} (-1)^j s_D(a_j).
\]

Let $y_0=s_D(a)$, and for each $i \in \{1, \dots, 2p\}$, let 
\[
y_i = (-1)^i s_D(a) + 2 \cdot \sum _{j=1}^i (-1)^{j+i} s_D(a_j) = 2s_D(a_i) - y_{i-1}.
\]
Notice that $y_{2p} = q$. Also, $\phi_\nu(y_i) = \phi_\nu(s_D(a))$ for every $i \in \{0, \dots, 2p\}$, so $y_0, \dots, y_{2p} \in Q$. It follows that for every $i \in \{0, \dots, 2p-1\}$, $K$ includes the element
\[
z_i=2g(s_D(a_{i+1}))-g(y_i)-g(2s_D(a_{i+1})-y_i) 
\]
\[
= 2g(s_D(a_{i+1}))-g(y_i)-g(y_{i+1}).
\]
Therefore $K$ also includes the element
\[
\sum _{i=0}^{2p-1} (-1)^i z_i = g(y_{2p})-g(y_0) + 2 \cdot \sum _{i=0}^{2p-1} (-1)^i g(s_D(a_{i+1}))
\]
\[
=g(q) -g(s_D(a))+ 2 \cdot \sum _{i=1}^{2p} (-1)^{i-1} g(s_D(a_{i})).
\]
\end{proof}

\begin{proposition}
\label{kerf}
The kernel of $f$ is the subgroup $K$ mentioned in Lemma \ref{klem}.
\end{proposition}
\begin{proof}
As $f(g(x))=x$ $\forall x \in Q$, $f(2g(y)-g(x)-g(2y-x))=2y-x-(2y-x)=0$ $\forall x,y \in Q$. Hence $K \subseteq \ker f$.

For the reverse inclusion, suppose $D$ is a diagram of $L$, and let $\widehat g: \mathbb Z^{A(D)} \to \mathbb Z ^Q$ be the homomorphism with $\widehat g (a) = g(s_D(a))$ $\forall a \in A(D)$.

Suppose $x \in \ker f$. According to Lemma \ref{klem}, there are arcs $a_1,\dots a_n \in A(D)$ and integers $m_1, \dots m_n \in \mathbb Z$ with
\[
x - \sum _{i=1}^n m_i g(s_D(a_i)) \in K.
\]
Let
\[
x' = \sum_{i=1}^n m_i a_i \in \mathbb Z ^{A(D)} \text{,}
\]
so that $x-\widehat g (x') \in K$.
As $f(x)=0$ and $K \subseteq \ker f$,
\[
0=f(x-(x-\widehat g (x'))) =f(\widehat g (x'))= \sum_{i=1}^n m_i s_D(a_i) = s_D(x').
\]
That is, $x' \in \ker s_D$. By definition, $\ker s_D$ is the subgroup of $\mathbb Z ^{A(D)}$ generated by the various elements $r_D(c) = 2 a - b-b'$, where $c \in C(D)$ is a crossing with overpassing arc $a$ and underpassing arcs $b,b'$. It follows that $\widehat g(x')$ is equal to a linear combination (with integer coefficients) of elements $\widehat g (r_D(c))$. Notice that if $c \in C(D)$ then the fact that $r_D(c)=2 a-b-b' \in \ker s_D$ implies that $s_D(b') = 2s_D(a) - s_D(b)$, so $\widehat g (r_D(c)) = 2g(s_D(a)) - g(s_D(b))-g(s_D(b')) \in K$. Therefore $\widehat g(x')$ is a linear combination of elements of $K$, so $\widehat g(x') \in K$.

As $x-\widehat g(x') \in K$ too, it follows that $x \in K$. \end{proof}

We now revert to our usual notation, with $Q_A(L)_\nu$ rather than $Q$. Proposition \ref{kerf} tells us that $Q_A(L)_\nu$ provides a presentation of $M_A(L)_\nu$, with the elements of $Q_A(L)_\nu$ as generators and the equations $2y-x-(2y-x)=0$, with $x,y \in Q_A(L)_\nu$, as defining relations. The map $\phi_\nu$ is constant on each orbit of $Q_A(L)_\nu$, so $\phi_\nu$ is determined by this presentation, together with the correspondence between the orbits of $Q_A(L)_\nu$ and the components $K_1, \dots, K_ \mu$ of $L$. 

The following notion will be useful.

\begin{definition}
\label{nueq}
Suppose $L$ and $L'$ are classical links, and there is an isomorphism $h:M_A(L)_\nu \to M_A(L')_\nu$ that is compatible with the $\phi_\nu$ maps of $L$ and $L'$, i.e., $\phi_\nu = \phi'_\nu h:M_A(L)_\nu \to A_ \mu$. Then we say that $L$ and $L'$ are \emph{$\phi_\nu$-equivalent}.
\end{definition}

\begin{theorem}
\label{imqmod}
Let $L_1$ and $L_2$ be links. Then $Q_A(L_1)_\nu \cong Q_A(L_2)_\nu$ if and only if the components of $L_1$ and $L_2$ can be indexed to make $L_1$ and $L_2$ $\phi_\nu$-equivalent.
\end{theorem}

\begin{proof}
The difficult part of the proof has already been done: if $Q_A(L_1)_\nu \cong Q_A(L_2)_\nu$, then the quandle isomorphism gives us an equivalence between the presentations of the groups $M_A(L_1)_\nu$ and $M_A(L_2)_\nu$ provided by Proposition \ref{kerf}. If we re-index the components of $L_1$ and $L_2$ so that the quandle isomorphism $Q_A(L_1)_\nu \cong Q_A(L_2)_\nu$ always matches orbits corresponding to components with the same index, then the equivalence between the group presentations will be compatible with the $\phi_\nu$ maps.

The other direction is obvious, as Definition \ref{coreprime} defines $Q_A(L)_\nu$ using $\phi_\nu$. \end{proof}

It is important to allow re-indexing of link components in Theorem \ref{imqmod}, because it is possible for two links to fail to be $\phi_\nu$-equivalent even if the only difference between them is the indexing of their components. An example is given in Sec.\ \ref{tlink}.

\section{Using $\IMQ(L)$ to construct $M_A(L) _\nu$}
\label{proof2}

In this section, we modify the discussion of Sec.\ \ref{proof1} to provide a presentation of $M_A(L)_\nu$ derived from $\IMQ(L)$, rather than $Q_A(L)_\nu$. 

As mentioned in Theorem \ref{main1}, we have a surjective quandle map $\IMQ(L) \to \QIMG(L)$ under which $q_a \mapsto g_a  \thinspace \forall a \in A(D)$, and we have an isomorphism $\QIMG(L) \to Q_A(L)_\nu$ under which $g_a \mapsto s_D(a)$  $\forall a \in A(D)$. The composition is a surjective quandle map $\widetilde s_D:\IMQ(L) \to Q_A(L)_\nu$, with $\widetilde s_D(q_a)=s_D(a)$ $\forall a \in A(D)$. We also use $\widetilde s_D$ to denote the linear extension of $\widetilde s_D$ to a homomorphism $\mathbb Z ^{\IMQ(L)} \to M_A(L)_\nu$ of abelian groups.

\begin{lemma}
\label{imqlem}
Let $D$ be a diagram of $L$, and let $K$ be the subgroup of $\mathbb Z ^{\IMQ(L)}$ generated by the subset $\{ 2q-p-p \triangleright q \mid p,q \in \IMQ(L) \}$.
Then for every $x \in \mathbb Z ^{\IMQ(L)}$, there are arcs $a_1, \dots, a_n \in A(D)$ and integers $m_1, \dots, m_n \in \mathbb Z$ such that
\[
x - \sum _{i=1}^n m_i q_{a_i} \in K.
\]
\end{lemma}
\begin{proof}
Every element of $\mathbb Z ^{\IMQ(L)}$ is a linear combination of elements of $\IMQ(L)$, with coefficients in $\mathbb Z$. Therefore, it suffices to verify the lemma when $x=q \in \mathbb Z ^{\IMQ(L)}$. If the lemma holds for $p$ and $q$, then the lemma also holds for $p \triangleright q$, because $p \triangleright q- (2q-p) \in K$. As $\IMQ(L)$ is generated by the elements $q_a$ with $a \in A(D)$, then, it suffices to verify the lemma when $x=q_a$.

When $x=q_a$ the lemma is obvious, as $x-q_a=0$.
\end{proof}

\begin{proposition}
\label{imqkerf}
The kernel of $\widetilde s_D$ is the subgroup $K$ mentioned in Lemma \ref{imqlem}.
\end{proposition}
\begin{proof}
The argument is quite similar to the proof of Proposition \ref{kerf}.

If $p,q \in Q$ then as $\widetilde s_D:\IMQ(L) \to Q_A(L)_\nu$ is a quandle map, $\widetilde s_D(2q-p-p \triangleright q)=2 \widetilde s _D(q)-\widetilde s _D(p)-\widetilde s _D(p \triangleright q)= \widetilde s _D(p) \triangleright \widetilde s _D(q) - \widetilde s _D(p \triangleright q)=0$. Hence $K \subseteq \ker \widetilde s _D$.

For the reverse inclusion, suppose $D$ is a diagram of $L$, and let $\widehat q: \mathbb Z^{A(D)} \to \mathbb Z ^Q$ be the homomorphism with $\widehat q (a) = q_a$ $\forall a \in A(D)$.

Suppose $x \in \ker \widetilde s _D$. According to Lemma \ref{imqlem}, there are arcs $a_1,\dots a_n \in A(D)$ and integers $m_1, \dots m_n \in \mathbb Z$ such that
\[
x' = \sum _{i=1}^n m_i a_i \in \mathbb Z ^{A(D)}
\]
has $x-\widehat q (x') \in K$. As $\widetilde s _D(x)=0$ and $K \subseteq \ker \widetilde s _D$, $s_D(x') = \widetilde s _D(\widehat q (x')) = \widetilde s _D(x - (x- \widehat q (x')))=0-0=0$, so $x' \in \ker s_D = r_D(\mathbb Z ^{C(D)})$. Then $\widehat q(x')$ is equal to a linear combination (with integer coefficients) of elements $\widehat q (r_D(c))$, $c \in C(D)$. Notice that if $c \in C(D)$ is a crossing as pictured in Fig.\ \ref{crossfig}, then 
\[
\widehat q (r_D(c)) = \widehat q(2 s_D(a) - s_D(b) -s_D(b')) =  2q_a - q_b-q_{b \triangleright a}\in K.
\]
Therefore $\widehat q(x') \in K$. As $x-\widehat q(x') \in K$, it follows that $x \in K$. \end{proof}

Proposition \ref{imqkerf} tells us that the abelian group $M_A(L)_\nu$ has a presentation determined by the quandle $\IMQ(L)$, with generators the elements $\widetilde s _D(q)$, $q \in \IMQ(L)$, and relations $2\widetilde s _D(q)-\widetilde s _D(p)-\widetilde s _D(p \triangleright q)=0$ $\forall p,q \in \IMQ(L)$. Each orbit of $\IMQ(L)$ corresponds to a component $K_i$ of $L$, so once we know which component $K_i$ corresponds to each orbit, this presentation of $M_A(L)_\nu$ will also determine the map $\phi_\nu:M_A(L)_\nu \to A_ \mu$. We deduce an analogue of one direction of Theorem \ref{imqmod}:

\begin{theorem}
\label{imqmod2}
Let $L_1$ and $L_2$ be links. If $\IMQ(L_1) \cong \IMQ(L_2)$, then the components of $L_1$ and $L_2$ can be indexed to make $L_1$ and $L_2$ $\phi_\nu$-equivalent.
\end{theorem}

Notice that unlike Theorem \ref{imqmod}, Theorem \ref{imqmod2} is not an ``if and only if.'' The difference is that Definition \ref{coreprime} uses $\phi_\nu$ to define $Q_A(L)_\nu$, but Definition \ref{imq} does not mention $\phi_\nu$. Examples that confirm that $\phi_\nu$ does not determine $\IMQ(L)$ are mentioned in Sec.\ \ref{sec:twoexamples}.

\section{Characteristic subquandles of core quandles}
\label{coresec}

Suppose $A$ is an abelian group
\[
A = \mathbb Z ^r \oplus \mathbb{Z}_{2^{n_1}}
\oplus \dots \oplus \mathbb{Z}_{2^{n_k}} \oplus B \text{,}
\]
where $k,r \geq 0$, $n_1,\dots,n_k \geq 1$, $\mathbb Z ^r$ is a free abelian group of rank $r$, and $|B|$ is odd. (If $r=0, k=0$ or $|B|=1$ then the corresponding direct summands need not appear.) Think of elements of $A$ as $(r+k+1)$-tuples, with the first $r$ coordinates coming from $\mathbb Z$ and the last coordinate coming from $B$. 

\begin{definition}
\label{charcore}
The \emph{characteristic subquandle} of $\Core(A)$ is
\[
\Core'(A) = \{ (x_1,\dots,x_{r+k+1}) \in A \mid \text{at most one of } x_1,\dots,x_{r+k} \text{ is odd}\} \text{,}
\]
considered as a quandle using the operation $\triangleright$ of $\Core(A)$.
\end{definition}

The number $r+k$ is the \emph{2-rank} of $A$. Notice that $\Core'(A)$ is the union of $r+k+1$ cosets of $2 \cdot A$ in $A$. There are $2^{r+k}$ cosets in all, so $\Core'(A)=\Core(A)$ if $r+k \leq 1$, and $\Core'(A)$ is a proper subset of $\Core(A)$ if $r+k > 1$. 

Every finitely generated abelian group $A$ is isomorphic to a direct sum like the one in Definition \ref{charcore}. The summands in the direct sum are uniquely determined, but the isomorphism is not. (For instance, if $A \cong \mathbb{Z}_2 \oplus \mathbb{Z}_2$ then there are three distinct direct sum decompositions of $A$, one for each basis of $A$ as a vector space over the two-element field $\textup{GF}(2)$.) In general, then, $\Core(A)$ has several different characteristic subquandles, all isomorphic to each other. We use the notation $\Core'(A)$ with the understanding that the characteristic subquandle is defined only up to automorphisms of $A$.

The purpose of this section is to prove that $\Core'(A)$ is a classifying invariant of $A$. We prove this by constructing a presentation of $A$ from $\Core'(A)$. The construction is similar to those of the preceding sections, though the argument is a bit more complicated. Because of the similarity, we use similar notation.

To wit: For convenience, we let $Q=\Core'(A)$ for much of the rest of this section (and only this section). Let $\mathbb{Z}^Q$ be the free abelian group on the set $Q$, and let $f:\mathbb{Z}^Q \to A$ be the homomorphism that sends each generator $x \in Q$ to itself, considered as an element of $A$. Let $g:Q \to \mathbb{Z}^Q$ be the function that sends each $x \in Q$ to itself, considered as a generator of $\mathbb{Z}^Q$.

Now, suppose $A$ is a finitely generated abelian group. Then $A$ is described up to isomorphism by a direct sum 
\[
A = \mathbb{Z}^r
\oplus \mathbb{Z}_{2^{n_1}}
\oplus \dots \oplus \mathbb{Z}_{2^{n_k}} \oplus \mathbb{Z}_{m_1}
\oplus \dots \oplus \mathbb{Z}_{m_ \ell} \text{,}
\]
where $r,k,\ell \geq 0$, if $k \geq 1$ then $n_1,\dots,n_k \geq 1$, and if $\ell \geq 1$ then $m_1, \dots, m_{\ell}$ are odd. 

Define a function $j:Q \to \{0,\dots,r+k\}$, as follows. If $x = (x_1,\dots,x_{r+k+ \ell }) \in Q$, $1 \leq j \leq r+k$ and $x_j$ is odd, then $j(x)=j$. If $x = (x_1,\dots,x_{r+k+ \ell}) \in Q$ and $x_1,\dots,x_{r+k}$ are all even, then $j(x)=0$. Also, for $1 \leq j \leq r+k + \ell$ let $1_j$ be the element $(0,\dots,0,1,0,\dots,0) \in Q$, with the $1$ in the $j$th coordinate. These elements generate $A$, so $f$ is surjective.

\begin{lemma}
\label{simpforml}
Let $K$ be the subgroup of $\mathbb{Z}^Q$ generated by 
\[
\{g(2y)-g(x)-g(x \triangleright y) \mid x,y \in Q \}.
\]
Then $g(mx)-mg(x) \in K \thinspace \forall x \in Q$ $\forall m \in \mathbb{Z}$.
\end{lemma}
\begin{proof}
The lemma holds trivially when $m=1$, and it holds when $m=2$ because $x \triangleright x=x$ and $g(2x)-g(x)-g(x \triangleright x) \in K$. The lemma holds when $m=0$ because $g(2 \cdot 0)-g(0)-g(0 \triangleright 0)=g(0)-g(0)-g(0)=-g(0) \in K$, so $g(0 \cdot x)-0 \cdot g(x)=g(0) \in K$.
The lemma holds when $m=-1$ because $g(0)-g(x)-g(x \triangleright 0) = g(0)-g(x)-g(-x) \in K$, and $g(0) \in K$, so $g(-x)+g(x) \in K$.

Now, suppose $m \geq 2$, $x \in Q$ and $g(py)-pg(y) \in K$ $\forall p \in \{-1, \dots, m \}$ $\forall y \in Q$. Then $K$ contains $g(2x)-2g(x)$ and $g((m-1)x)-(m-1)g(x)$. Also, $K$ contains $g((m-1)x)+g((1-m)x)$, because $y=(m-1)x \in Q$ and $g(y)+g(-y)=g((m-1)x)+g((1-m)x)$. As $K$ contains
\[
g(2x)-g((1-m)x)-g(((1-m)x) \triangleright x) = g(2x)-g((1-m)x)-g((m+1)x) \text{,}
\]
it follows that $K$ contains
\[
2g(x)+(m-1)g(x)-g((m+1)x)=(m+1)g(x)-g((m+1)x).
\]

If $m \leq -2$ and $x \in Q$ then $-x \in Q$ too, so $K$ contains $g(|m|(-x))-|m|g(-x)=g(mx)+mg(-x)$. As $K$ contains $g(-x)+g(x)$, it follows that $K$ contains $g(mx)+m \cdot(-g(x))=g(mx)-mg(x)$.
\end{proof}

\begin{lemma}
\label{simpform}
For every $z \in \mathbb{Z}^Q$, there are integers $z_1, \dots,z_{r+k+ \ell}$ such that
\[
z - \sum_{j=1}^{r+k+ \ell} z_j \cdot g(1_j) \in K.
\]
\end{lemma}
\begin{proof} 
Consider an element $x=(x_1,\dots,x_{r+k+ \ell}) \in Q$, with $x_1, \dots, x_r \leq 0$. Notice that if $1 \leq j \leq r+k+ \ell$ then $(-x) \triangleright 1_j=2 \cdot 1_j+x$, so $K$ contains $g(2 \cdot 1_j)-g(-x) -g(2 \cdot 1_j +x)$. According to Lemma \ref{simpforml}, $K$ also contains $g(2 \cdot 1_j)-2 \cdot g( 1_j)$ and $g(-x)+g(x)$, so 
\begin {equation}
\label{reduce}
g(x) - g(x + 2 \cdot 1_j) + 2 \cdot g(1_j) \in K .
\end{equation}

Suppose $j \neq j(x) \in \{1, \dots, r+k+\ell \}$. If $j> r+k$ then $m_{j - r - k}$ is odd, so multiplication by $2$ defines an automorphism of $\mathbb{Z}_{m_{j-r-k}}$; hence there is a non-negative integer $y_j$ with $2y_j=-x_{j}$ in $\mathbb{Z}_{m_{j-r-k}}$. If $r < j \leq r+k$ then $x_j$ is even, so there is a non-negative integer $y_j$ with $2y_j=-x_j$ in $\mathbb{Z}_{2^{n_{j-r}}}$. If $j \leq r$ there is a non-negative integer $y_j$ with $2y_j=-x_j$ in $\mathbb{Z}$. Applying formula (\ref{reduce}) $y_j$ times with respect to each such $j$, we deduce that $K$ contains the element
\[
x' = g(x) - g \bigg(x - \sum_{\substack{ j=1 \\ j \neq j(x)}} ^{r+k + \ell} x_{j} \cdot 1_{j} \bigg) - \sum_{\substack{ j=1 \\ j \neq j(x)}} ^{r+k + \ell} x_{j} \cdot g(1_j).
\]

If $j(x)=0$, then it follows that
\[
g(x) - g(0) - \sum_{j=1} ^{r+k+ \ell} x_{j} \cdot g(1_j) \in K .
\]
As $g(0) \in K$, we deduce that 
\[
g(x) - \sum_{j=1} ^{r+k+ \ell} x_{j} \cdot g(1_j) \in K .
\]

If $j(x)>r$ then $x_{j(x)}$ is odd, and there is a positive integer $y_{j(x)}$ with $2y_{j(x)}=1-x_{j(x)}$ in $\mathbb{Z}_{2^{n_{j(x)-r}}}$. We apply the formula (\ref{reduce}) to $x'$ $y_{j(x)}$ times, using $j=j(x)$. We deduce that $K$ contains
\[
g(x) - g\bigg(x - \sum_{\substack{ j=1 \\ j \neq j(x)}} ^{r+k + \ell} x_{j} \cdot 1_{j} + (1-x_{j(x)}) \cdot 1_{j(x)} \bigg)+ (1-x_{j(x)}) \cdot g(1_{j(x)}) -
\sum_{\substack{ j=1 \\ j \neq j(x)}} ^{r+k + \ell} x_{j} \cdot g(1_j)
\]
\[
=g(x) - g(1_{j(x)})+ g(1_{j(x)}) -  \sum_{j=1} ^{r+k+ \ell} x_{j} \cdot g(1_j)
=g(x) -  \sum_{j=1} ^{r+k+ \ell} x_{j} \cdot g(1_j).
\]

If $1 \leq j(x) \leq r$ then $x_{j(x)}$ is odd and negative, so there is a positive integer $y_{j(x)}$ with $2y_{j(x)}=1-x_{j(x)}$ in $\mathbb{Z}$. We apply (\ref{reduce}) to $x'$ $y_{j(x)}$ times, using $j=j(x)$. We deduce that $K$ contains
\[
g(x) - g\bigg(x - \sum_{\substack{ j=1 \\ j \neq j(x)}} ^{r+k + \ell} x_{j} \cdot 1_{j} + (1-x_{j(x)}) \cdot 1_{j(x)} \bigg)+ (1-x_{j(x)}) \cdot g(1_{j(x)}) -
\sum_{\substack{ j=1 \\ j \neq j(x)}} ^{r+k + \ell} x_{j} \cdot g(1_j)
\]
\[
=g(x) - g(1_{j(x)})+ g(1_{j(x)}) -  \sum_{j=1} ^{r+k+ \ell} x_{j} \cdot g(1_j)
=g(x) -  \sum_{j=1} ^{r+k+ \ell} x_{j} \cdot g(1_j).
\]

We see that if $x=(x_1,\dots,x_{r+k+ \ell}) \in Q$ and $x_1, \dots, x_r \leq 0$, then the lemma holds for $z=g(x)$. 

Now, suppose $x=(x_1,\dots,x_{r+k+ \ell}) \in Q$, $j_0\leq r$ and $x_{j_0} > 0$. Suppose further that whenever $y=(y_1,\dots,y_{r+k+ \ell}) \in Q$ and the list $y_1,\dots,y_r $ includes strictly fewer positive numbers than the list $x_1,\dots,x_r$,
\[
g(y)-  \sum_{j=1} ^{r+k+ \ell} y_{j} \cdot g(1_j) \in K.
\]
Let $y=(y_1,\dots,y_{r+k+ \ell}) \in Q$ be the element with $y_j=x_j$ for $j \neq j_0$ and $y_{j_0}=-x_{j_0}$. According to the formula (\ref{reduce}), for $i \geq 0$ 
\[
g(y+ 2i \cdot 1_{j_0}) - g(y + 2(i+1) \cdot 1_{j_0}) + 2 \cdot g(1_{j_0}) \in K.
\]
It follows that $K$ contains
\[
\sum_{i=0}^{x_{j_0}-1} \bigg( g(y+ 2i \cdot 1_{j_0}) - g(y + 2(i+1) \cdot 1_{j_0}) + 2 \cdot g(1_{j_0}) \bigg)=g(y)-g(x)+2x_{j_0} \cdot 1_{j_0} \text{,}
\]
so $K$ contains the difference
\[
g(y)-  \sum_{j=1} ^{r+k+ \ell} y_{j} \cdot g(1_j) - \bigg(g(y)-g(x)+2x_{j_0} \cdot 1_{j_0} \bigg) = g(x) - \sum_{j=1} ^{r+k+ \ell} x_{j} \cdot g(1_j). 
\]
Using induction on the number of positive coordinates $x_j$ with $j \leq r$, we conclude that the lemma holds for $z=g(x)$ whenever $x \in Q$.

If $z$ is an arbitrary element of $\mathbb{Z}^Q$, then $z$ is equal to a linear combination over $\mathbb Z$ of various elements $g(x)$ with $x \in Q$. Applying the lemma to each $g(x)$ and collecting terms, we conclude that $K$ contains the difference between $z$ and a linear combination over $\mathbb Z$ of $g(1_1), \dots, g(1_{r+k+ \ell})$. \end{proof}
\begin{proposition}
\label{coresecprop}
The kernel of the epimorphism $f:\mathbb{Z}^Q \to A$ is $K$.
\end{proposition}
\begin{proof}
If $x,y \in Q$ then certainly $f(g(2y)-g(x)-g(x \triangleright y))=2y-x-x \triangleright y = 0$ in $A$. Therefore $K \subseteq \ker f$.

Now, suppose $z \in \ker f$. According to Lemma \ref{simpform},
\begin{equation}
\label{last}
z - \sum_{j=1}^{r+k+ \ell} z_j \cdot g(1_j) \in K
\end{equation}
for some integers $z_1, \dots,z_{r+k+ \ell}$. As $K \subseteq \ker f$, it follows that 
\[
f(z) = f \bigg( \sum_{j=1}^{r+k+ \ell} z_j \cdot g(1_j) \bigg) = \sum_{j=1}^{r+k+ \ell} z_j \cdot fg(1_j) = \sum_{j=1}^{r+k+ \ell} z_j  \cdot 1_j \in A.
\]
As $f(z)=0$, it follows that $z_j \cdot 1_j=0$ for each $j$. Lemma \ref{simpforml} implies that for each $j$, $z_jg(1_j)-g(0) = z_jg(1_j)-g(z_j \cdot 1_j) \in K$. As $g(0) \in K$,  it follows that for each $j$, $z_j g(1_j) \in K$. With (\ref{last}), this implies that $z\in K$. \end{proof}

We are now ready to prove the following:

\begin{proposition}
\label{class}
If $A$ and $A'$ are finitely generated abelian groups, then $A \cong A'$ if and only if $\Core'(A) \cong \Core'(A')$.
\end{proposition}

\begin{proof}
Let $A$ and $A'$ be finitely generated abelian groups. We use the notation established above for both $A$ and $A'$, with apostrophes where appropriate; for instance, $Q' = \Core'(A')$. If $A \cong A'$, then Definition \ref{charcore} makes it clear that $Q \cong Q'$. 

For the converse, suppose $h:Q \to Q'$ is a quandle isomorphism. Let $\sigma:A' \to A'$ be the function given by $\sigma(x)=x-h(0)$. Then $\sigma$ is an automorphism of $\Core(A')$, so the composition $h'=\sigma h$ maps $Q$ isomorphically onto a subquandle $Q''=\sigma(Q')$ of $\Core(A')$. Notice that according to Definition \ref{charcore}, $2x \in Q'$ $\forall x \in A'$. It follows that $Q''=\{y-h(0) \mid y \in Q'\}$ contains $2h(0)-h(0)=h(0)$. As $Q'$ generates $A'$, and the subgroup generated by $Q''$ includes $y=h(0)+y-h(0)$ for each $y \in Q'$, $Q''$ must also generate $A'$.

As $f':\mathbb Z ^{Q'} \to A'$ is surjective, for each $x \in Q$ we may choose an element  $\eta(x) \in \mathbb Z ^{Q'}$ with $f' \eta(x)=h'(x)$. Extending linearly, we obtain a homomorphism $\eta:\mathbb Z ^{Q} \to \mathbb Z ^{Q'}$ that has $f' \eta g=h': Q \to Q''$. The subset $Q''$ generates $A'$, and $h':Q \to Q''$ is surjective, so it follows that $f' \eta:\mathbb Z ^{Q} \to A'$ is an epimorphism.

As $h'(0)=h(0)-h(0)=0$ and $h':Q \to Q''$ is a quandle isomorphism, every $x \in Q$ has
\[
h'(2x)=h'(2x-0)=h'(0 \triangleright x)=h'(0) \triangleright h'(x)
=2h'(x)-h'(0)=2h'(x) .
\]
It follows that if $x,y \in Q$ then
\[
f' \eta (g(2y)-g(x)-g(x \triangleright y)) =h' (2y)-h' (x)- h' (x \triangleright y)
\]
\[
=  2 h'(y) - h' (x)- h'(x) \triangleright h'(y) =0.
\]
We deduce that $K \subseteq \ker (f' \eta)$. According to Proposition \ref{coresecprop}, $K = \ker f$; $f:\mathbb Z ^{Q} \to A$ is an epimorphism, so $f' \eta$ induces an epimorphism $A \to A'$. 

Interchanging the roles of $A$ and $A'$, we obtain an epimorphism $A' \to A$. As $A$ and $A'$ are finitely generated modules over the Noetherian ring $\mathbb Z$, such paired epimorphisms exist only if $A \cong A'$. \end{proof}

At this point we stop using $Q$ to denote $\Core'(A)$. 

Before proceeding, we discuss the relationship between orbits in core quandles and orbits in characteristic subquandles.
\begin{proposition}
\label{dischar}
 Consider a finitely generated abelian group,
\[
A = \mathbb{Z}^r
\oplus \mathbb{Z}_{2^{n_1}}
\oplus \dots \oplus \mathbb{Z}_{2^{n_k}} \oplus \mathbb{Z}_{m_1}
\oplus \dots \oplus \mathbb{Z}_{m_ \ell} \text{.}
\]
\begin{enumerate}
    \item The groups $\Dis(\Core(A))$ and $\Dis(\Core'(A))$ are isomorphic.
    \item Every orbit in $\Core'(A)$ is also an orbit in $\Core(A)$.
    \item There are $2^{r+k}$ orbits in $\Core(A)$.
    \item There are $r+k+1$ orbits in $\Core'(A)$.
\end{enumerate}
\end{proposition}
\begin{proof}
As $\Core(A)$ is semiregular, its subquandle $\Core'(A)$ is semiregular too. It follows that there is a well-defined monomorphism $\text{ext}:\Dis(\Core'(A)) \to \Dis(\Core(A))$, with $\text{ext}(\beta_{x_1} \cdots \beta_{x_{2n}}) =\beta_{x_1} \cdots \beta_{x_{2n}}$ $\forall x_1, \dots, x_{2n} \in \Core'(A)$.

To verify the first assertion, we prove that ext is surjective. Suppose $d \in \Dis(\Core(A))$ is an elementary displacement. Then there are $x_1, x_2 \in A$ such that $d(x) = \beta_{x_1} \beta_{x_2}(x)=2x_1 -2x_2 + x$ $ \forall x \in A$. Choose integers $m_1, \dots, m_{r+k+\ell}$ such that 
\[
x_1-x_2=\sum_{j=1}^{r+k+\ell} (-1)^{j+1} m_j \cdot 1_j \text{,}
\]
and notice that $m_j \cdot 1_j \in \Core'(A)$ for each index $j$. If $r+k+\ell$ is even, then 
\[
d'=\beta_{(m_1 \cdot 1_1)} \cdots \beta_{(m_{r+k+\ell} \cdot 1_{r+k+\ell})} \in \Dis(\Core'(A)) \text{,}
\]
and $\text{ext}(d')(x)=d(x)$ $\forall x \in A$. If $r+k+\ell$ is odd then as $0 \in \Core'(A)$,
\[
d'=\beta_{(m_1 \cdot 1_1)} \cdots \beta_{(m_{r+k+\ell} \cdot 1_{r+k+\ell})} \beta_0 \in \Dis(\Core'(A)) \text{,}
\]
and $\text{ext}(d')(x)=d(x)$ $\forall x \in A$. Either way, $d=\text{ext}(d')$. The elementary displacements generate $\Dis(\Core(A))$, so the first assertion holds. The second assertion follows from the surjectivity of ext and Proposition \ref{orb}.

For the third assertion, notice that for each $j \in \{1,\dots,r+k+\ell\}$, $d_j=\beta_{1_j}\beta_0 \in \Dis(\Core(A))$ and  $d'_j=\beta_{(-1_j)}\beta_0 \in \Dis(\Core(A))$. It follows that if $x \in A$ then $d_j(x)=2 \cdot 1_j +x$ and $d'_j(x)=-2 \cdot 1_j+x$ are both elements of the orbit of $x$ in $\Core(A)$. Applying these displacements $d_j$ and $d'_j$ repeatedly, we see that every orbit in $\Core(A)$ includes an element $(y_1,\dots,y_{r+k},0,\dots,0)$ such that $y_1,\dots,y_{r+k} \in \{0,1\}$. It is easy to see that no two such elements appear in the same orbit in $\Core(A)$; this implies the third assertion. The fourth assertion follows from the fact that $\Core'(A)$ contains precisely $r+k+1$ of the elements $(y_1,\dots,y_{r+k},0,\dots,0)$ with $y_1,\dots,y_{r+k} \in \{0,1\}$.
\end{proof}

Combining Proposition \ref{dischar} with earlier results, we come to the conclusion that the quandles $\textup{\Core}'(\ker w_ \nu)$ and $Q_A(L)_\nu$ are very closely related to each other.

\begin{proposition}
\label{closeq} 
If $L$ is a link, the following statements hold.
\begin{enumerate}
    \item $\textup{Core}'(\ker w_ \nu)$ has $\mu$ orbits, each of which is also an orbit of $\textup{Core}(\ker w_ \nu)$.
     \item The displacement groups of $\textup{Core}'(\ker w_ \nu)$ and $Q_A(L)_\nu$ are isomorphic.
     \item $Q_A(L)_\nu$ is isomorphic to a subquandle $Q' \subseteq \Core(\ker w_ \nu)$, which satisfies item 1.
     \item $|\Core'(\ker w_ \nu)| = |Q_A(L)_\nu|$.
     \item If $\mu \leq 3$, then $\textup{Core}'(\ker w_ \nu) \cong Q_A(L)_\nu$.
\end{enumerate}
\end{proposition}
\begin{proof}
For item 1, refer to Proposition \ref{dischar}.

For item 2, Proposition \ref{dischar} tells us $\Dis(\Core'(\ker w_ \nu)) \cong \Dis(\Core(\ker w_ \nu))$, and Proposition \ref{Coreprop} tells us $\Dis(\Core(\ker w_ \nu)) \cong \ker w_ \nu/ \ker w_ \nu(2)$. There is a natural surjection $\ker w_ \nu \to 2 \cdot \ker w_ \nu$, defined by $x \mapsto 2x$. It is obvious that the kernel of this surjection is $\ker w_ \nu(2)$, so $\ker w_ \nu/ \ker w_ \nu(2) \cong 2 \cdot\ker w_ \nu$. Lemma \ref{kerphi} tells us $2 \cdot\ker w_ \nu = \ker \phi_\nu$, and Proposition \ref{imq'dis} tells us $\ker \phi_\nu \cong \Dis(Q_A(L)_\nu)$.

For item 3, let $D$ be a diagram of $L$, pick any arc $a^* \in A(D)$, and let $Q'=\{ y - s_D(a^*) \mid y \in Q_A(L)_\nu \}$. The function $g(y) = y - s_D(a^*)$ defines a bijection between the subquandles $Q_A(L)_\nu$ and $Q'$ of $\textup{Core}(M_A(L))_\nu$, and this bijection is a quandle isomorphism because for any $y_1,y_2 \in Q_A(L)_\nu$,
\[
g(y_1 \triangleright y_2)=g(2y_2-y_1)=2y_2-y_1- s_D(a^*) 
\]
\[
= 2(y_2 - s_D(a^*)) - (y_1 - s_D(a^*)) = g(y_1) \triangleright g(y_2).
\]
The fact that $Q'$ has $\mu$ orbits follows from Corollary \ref{imq'orb}. 

It remains to verify that for every $q \in Q'$, the orbits of $q$ in $\textup{Core}(\ker w_ \nu)$ and $Q'$ are the same. According to Proposition \ref{orb}, if $q \in Q'$ then the orbit of $q$ in $Q'$ is $\{d(q) \mid d \in \Dis(Q') \}$. As $g:Q_A(L)_\nu \to Q'$ is an isomorphism, Proposition \ref{imq'dis} implies that $\Dis(Q')$ is the set of compositions $g \circ \delta(k) \circ g^{-1}$ such that $k \in \ker \phi_ \nu$. It follows that if $q = g(y) \in Q'$ then the orbit of $q$ in $Q'$ is the set of all elements
\[
(g \circ \delta(k) \circ g^{-1})(q) = g (\delta(k)(y)) = g(k+y) = k+y - s_D(a^*) =k+q
\]
such that $k \in \ker \phi_ \nu$.

On the other hand, if $q \in Q'$ then the orbit of $q$ in $\textup{Core}(\ker w_ \nu)$ is $\{d(q) \mid d \in \Dis(\Core(\ker w_ \nu)) \}$. Proposition \ref{Coreprop} tells us that every displacement of $\textup{Core}(\ker w_ \nu)$ is $f(y)$ for some $y \in \ker w_ \nu$, where $f(y)(x) = \beta_y \beta_0(x) = 2y+x$ $\forall x \in \ker w_ \nu$. It follows that the orbit of $q$ in $\textup{Core}(\ker w_ \nu)$ is the set of all elements $2y + q$ such that $y \in \ker w_ \nu$. Lemma \ref{kerphi} tells us that $2 \cdot \ker w_ \nu = \ker \phi _ \nu$, so the orbit of $q$ in $\textup{Core}(\ker w_ \nu)$ is the same as the orbit of $q$ in $Q'$.

Item 4 follows from items 1, 2 and 3, as the cardinality of a semiregular involutory medial quandle is the product (number of orbits) $\times$ (size of displacement group).

For item 5, consider that Proposition \ref{dischar} and Corollary \ref{kstruc} tell us that $\textup{Core}(\ker w_ \nu)$ has $2^{\mu-1}$ orbits. If $\mu=1$ or $\mu=2$, then $2^{\mu-1}=\mu$ and items 1 and 3 imply that $\textup{Core}'(\ker w_ \nu)=Q'=\Core(\ker w_ \nu)$.

If $\mu=3$ then according to Corollary \ref{kstruc}, there is an isomorphism
\[
\ker w_{\nu} \cong \mathbb{Z}^{r-1}
\oplus \mathbb{Z}_{2^{n_1}}
\oplus \dots \oplus \mathbb{Z}_{2^{n_k}} \oplus B \text{,}
\]
with $r+k=3$ and $B$ a finite group of odd order. We can use such an isomorphism to think of elements of $\ker w_ \nu$ as $3$-tuples $(x_1,x_2,x_3)$, with $x_3 \in B$. As discussed in Proposition \ref{dischar}, $\textup{Core}(\ker w_ \nu)$ has four orbits, the cosets of the elements $(0,0,0),(1,0,0),(0,1,0)$ and $(1,1,0)$ with respect to the subgroup $2 \cdot \ker w_ \nu$ of $\ker w_ \nu$.

For any fixed element $q^*$ of $\textup{Core}(\ker w_ \nu)$, the map $x \mapsto x-q^*$ is an automorphism of $\textup{Core}(\ker w_ \nu)$. It is easy to see that for any two sets of three orbits in $\textup{Core}(\ker w_ \nu)$, one of these automorphisms maps the union of the first three orbits onto the union of the second three orbits. As $\textup{Core}'(\ker w_ \nu)$ and $Q'$ are both unions of three of the four orbits in $\textup{Core}(\ker w_ \nu)$, it follows that $\textup{Core}'(\ker w_ \nu) \cong Q'$. \end{proof}

With Proposition \ref{closeq}, we complete the proofs of the positive assertions in Theorem \ref{main2}. The general implication $3 \implies 4$ follows from Theorem \ref{imqmod}, and the fact that the converse holds when $\mu \leq 3$ follows from item 5 of Proposition \ref{closeq}. The equivalence $2 \iff 3$ follows from item 2 of Theorem \ref{main1}. Theorem \ref{imqmod2} gives us the general implication $1 \implies 3$, and of course $1 \implies 2$ follows from this, as $2$ and $3$ are equivalent. Item 1 of Theorem \ref{main1} gives us the converse of $1 \implies 2$ when $\mu=1$, or $\mu=2$ and $\det L \neq 0$.

In the next sections we show that when $\mu>3$, it is not always true that $\textup{Core}'(\ker w_ \nu)$ and $Q_A(L)_\nu$ are isomorphic. First, though, we need to explain a way to distinguish these quandles from each other.

\section{The quandles $\textup{Core}'(\ker w_\nu)$ and $Q_A(L)_\nu$}
\label{proof3}
Proposition \ref{closeq} tells us that the quandles $\textup{Core}'(\ker w_\nu)$ and $Q_A(L)_\nu$ are closely related to each other, and the results of Secs. \ref{proof1} and \ref{coresec} tell us that both of these quandles provide presentations of the group $M_A(L)_\nu$. With these similarities in mind, it seems reasonable to guess that $\textup{Core}'(\ker w_\nu)$ and $Q_A(L)_\nu$ are always isomorphic. (In early versions of this work, we mistakenly asserted that this is the case.) But it turns out that despite their many similarities, the two quandles are distinct from each other. The purpose of this section is to explain this point.

Lemma \ref{mstruc} tells us that
\begin{equation}
\label{endsum}
M_A(L)_{\nu} \cong \mathbb{Z}^r
\oplus \mathbb{Z}_{2^{n_1}}
\oplus \dots \oplus \mathbb{Z}_{2^{n_{(\mu-r)}}} \oplus B \text ,
\end{equation}
where $r \in \{1,\dots,\mu \}$, $n_1, \dots, n_{(\mu -r)} \geq 1$ and $B$ is an abelian group of odd order. (If $\mu=r$ then the $\mathbb Z _{2^{n_i}}$ summands are absent.) We can use (\ref{endsum}) to represent elements of $M_A(L)_{\nu}$ as $(\mu+1)$-tuples $(x_1, \dots, x_{\mu+1})$, with $x_1, \dots, x_{r} \in \mathbb Z$ and $x_{\mu+1} \in B$. It is apparent that there is an epimorphism $M_A(L)_{\nu} \to A_ \mu $ defined by $(x_1, \dots, x_{\mu+1}) \mapsto (x_1, \overline x _2, \dots , \overline x_ {\mu})$, where the overline denotes reduction modulo $2$. We say that this epimorphism is \emph{obtained directly} from (\ref{endsum}).

\begin{proposition}
\label{kertchar}
The characteristic subquandle $\textup{Core}'(\ker w_{\nu})$ is isomorphic to $Q_A(L)_\nu$ if, and only if, there is some way to index the components of $L=K_1 \cup \dots \cup K_ \mu$ so that $\phi_ \nu$ is obtained directly from a direct sum (\ref{endsum}).
\end{proposition}
\begin{proof}
Suppose $\textup{Core}'(\ker w_{\nu}) \cong Q_A(L)_\nu$. As noted after Lemma \ref{mstruc}, 
\begin{equation}
\label{endsum2}
\ker w_{\nu} \cong \mathbb Z ^{r-1}  \oplus \mathbb Z_{2^{n_1}} \oplus \cdots \oplus \mathbb Z_{2^{n_{(\mu - r)}}} \oplus B
\end{equation}
where $r \in \{1,\dots,\mu \}$, $n_1, \dots, n_{(\mu -r)} \geq 1$ and $B$ is an abelian group of odd order. We can use (\ref{endsum2}) to think of elements of $\ker w _ \nu$ as $\mu$-tuples $(y_1, \dots, y_{\mu})$, with $y_1, \dots, y_{r-1} \in \mathbb Z$ and $y_ \mu \in B$. For $1 \leq j \leq \mu -1$, let $1_j=(0, \dots,0,1,0, \dots, 0)$, with $1$ in the $j$th coordinate. Then one orbit of $\textup{Core}'(\ker w_{\nu})$ contains $0$, and each of the other orbits of $\textup{Core}'(\ker w_{\nu})$ contains precisely one element $1_j$. 

As $Q_A(L)_\nu$ has one orbit for each component of $L$, an isomorphism between $Q_A(L)_\nu$ and $\textup{Core}'(\ker w_{\nu})$ will provide a correspondence between the components $K_1, \dots ,K_{\mu}$ and the orbits of $\textup{Core}'(\ker w_{\nu})$. Reindex $K_1, \dots ,K_{\mu}$ so that $K_1$ corresponds to the orbit of $\textup{Core}'(\ker w_{\nu})$ that contains $0$, and for $2 \leq j \leq \mu$, $K_j$ corresponds to the orbit of $\textup{Core}'(\ker w_{\nu})$ that contains $1_{(j-1)}$. 

If $a^* \in A(D)$ has $\kappa_D(a^*)=1$, then Lemma \ref{mstruc} tells us that $\mathbb Z \oplus \ker w_{\nu} \cong M_A(L) _ \nu$, with $(n,x) \in \mathbb Z \oplus \ker w_{\nu}$ corresponding to $n s_D(a^*) + x \in M_A(L) _ \nu$. Therefore (\ref{endsum2}) provides a direct sum representation of $M_A(L) _ \nu$ in the obvious way, by attaching a $\mathbb Z$ summand at the front, and the map $\phi _\nu$ is obtained directly from this direct sum representation.

For the converse, suppose $\phi _ \nu$ is obtained directly from (\ref{endsum}). If we use (\ref{endsum}) to represent elements of $M_A(L)_ \nu$ as $(\mu+1)$-tuples, then an element $x =(x_1, \dots, \allowbreak x_{\mu+1})  \in M_A(L)_{\nu}$ is included in $\ker w_{\nu}$ if and only if $x_1=0$. Hence (\ref{endsum}) yields a direct sum representation of $\ker w _ \nu$, by suppressing the first $\mathbb Z$ summand. Therefore $\textup{Core}'(\ker w_{\nu})$ is isomorphic to the following subquandle of $\textup{Core}(M_A(L)_ \nu)$:
\[
S=\{(x_1, \dots, x_{\mu+1}) \mid x_1 = 0 \text{ and no more than one of }x_2, \dots, x_{\mu} \text{  is odd} \}.
\]

Let $D$ be a diagram of $L$, and $a^*$ a fixed arc of $D$ with $\kappa_D(a^*)=1$. According to Definition \ref{coreprime}, $Q_A(L)_\nu=\phi_{\nu}^{-1}(\phi_{\nu}(s_D(A(D))))$. It is obvious that $Q_A(L)_\nu$ is isomorphic, as a subquandle of $\textup{Core}(M_A(L)_{\nu})$, to $\{z-s_D(a^*) \mid z \in Q_A(L)_\nu\}$. As $\phi _ \nu$ is obtained directly from (\ref{endsum}), the latter subquandle is precisely the same as the subquandle $S$ mentioned at the end of the previous paragraph. \end{proof}

In early versions of this paper, we made the mistake of assuming that the indexing requirement of Proposition \ref{kertchar} can always be satisfied. But this assumption is not justified: for the link $L$ discussed in Subsection \ref{llink} below, $\phi_\nu$ cannot be obtained directly from any direct sum decomposition of $M_A(L)_\nu$. 

\section{Five examples}
\label{sec:fiveexamples}

In this section, we present two pairs of examples to illustrate the failure of the converse of the implication $3 \implies 4$ of Theorem \ref{main2}. First, we mention an example to verify a point mentioned in Sec.\ \ref{proof1}: re-indexing the components of a link can produce a new link that is not $\phi_\nu$-equivalent to the original.

\subsection{The link $T_{(2,2)} \# T_{(2,4)}$}
\label{tlink}

The connected sum $T_{(2,2)} \# T_{(2,4)}$ of a Hopf link and a $(2,4)$-torus link has $\ker w_\nu \cong \mathbb Z_2 \oplus \mathbb Z _4$. (We leave the easy calculation to the reader.) It follows from Proposition \ref{closeq} that $Q_A(T_{(2,2)}\# T_{(2,4)})_\nu$ is isomorphic to
\[
\Core'(\mathbb Z_2 \oplus \mathbb Z _4) = \{(0,0),(0,2),(1,0),(1,2),(0,1),(0,3)\} \subset \mathbb Z_2 \oplus \mathbb Z _4 \text{.}
\]

The quandle $\Core'(\mathbb Z_2 \oplus \mathbb Z _4)$ has three orbits: $\{(0,0),(0,2)\}$, $\{(1,0),(1,2)\}$ and $\{(0,1), (0,3)\}$. As predicted by Proposition \ref{invorb}, the orbits are isomorphic to each other as separate quandles. But the orbits are not equivalent to each other within $\Core'(\mathbb Z_2 \oplus \mathbb Z _4)$: each of the translations $\beta_{(0,0)}$, $\beta_{(0,2)}$, $\beta_{(1,0)}$, and $\beta_{(1,2)}$ has four fixed points, but $\beta_{(0,1)}$ and $\beta_{(0,3)}$ have only two fixed points. (This happens because $2 \cdot (0,0)=2 \cdot (0,2)=2 \cdot (1,0)=2 \cdot (1,2)=(0,0)$, while $2 \cdot (0,1)=2 \cdot (0,3)=(0,2).$) It follows that the orbit $\{(0,1),(0,3)\}$ is preserved by every automorphism of $\Core'(\mathbb Z_2 \oplus \mathbb Z _4)$, so the component of $T_{(2,2)} \# T_{(2,4)}$ corresponding to the orbit $\{(0,1),(0,3)\}$ is singled out by the structure of $Q_A(T_{(2,2)} \# T_{(2,4)})_\nu$.

Therefore, if we index the components of $T_{(2,2)} \# T_{(2,4)}$ in such a way that the singled-out component is $K_1$, we obtain a link that is not $\phi_\nu$-equivalent to the result of indexing the components of $T_{(2,2)} \# T_{(2,4)}$ in such a way that the singled-out component is $K_2$.

\subsection {A four-component link}
\label{llink}

Let $L$ be the link represented by the diagram $E$ pictured in Fig.\ \ref{lfig}. 

We obtain a description of $M_A(L)_ \nu$ by using the crossings not marked with $*$ to eliminate the generators other than $s_E(a)$, $s_E(b)$, $s_E(d)$ and $s_E(j)$.
\begingroup
\allowdisplaybreaks
\begin{align*}
  &  s_E(c)=2s_E(a)-s_E(b) \\
  &  s_E(g)=2s_E(d)-s_E(a) \\
  &  s_E(e)=2s_E(c)-s_E(d) = 4s_E(a)-2s_E(b)-s_E(d) \\
  &  s_E(f)=2s_E(e)-s_E(a)=7s_E(a)-4s_E(b)-2s_E(d) \\
  &  s_E(h)=2s_E(g)-s_E(c) = -4s_E(a)+s_E(b)+4s_E(d) \\
  &  s_E(i)=2s_E(f)-s_E(j) = 14s_E(a)-8s_E(b)-4s_E(d)-s_E(j) \\
  &  s_E(k)=2s_E(j)-s_E(f) = -7s_E(a)+4s_E(b)+2s_E(d)+2s_E(j) \\
  &  s_E(m)=2s_E(j)-s_E(g) = s_E(a)-2s_E(d)+2s_E(j) \\
  &  s_E(n)=2s_E(m)-s_E(h) = 6s_E(a)-s_E(b)-8s_E(d)+4s_E(j) \\
  &  s_E(p)=2s_E(e)-s_E(k) = 15s_E(a)-8s_E(b)-4s_E(d)-2s_E(j)
\end{align*}
\endgroup
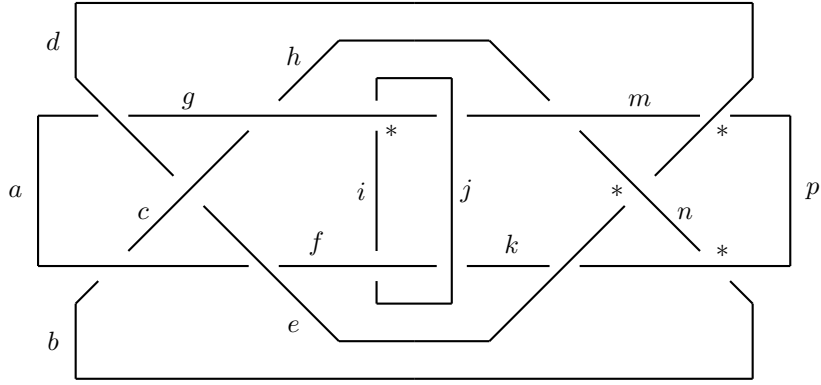
\begin{figure} [bht]
\centering
\begin{tikzpicture}
%left
\draw [thick] (-5,-1) -- (-5,1);
\draw [thick] (-5,-1) -- (-2.2,-1);
\draw [thick] (-4.2,1) -- (-5,1);
\draw [thick] (-4.5,1.5) -- (-3.2,0.2);
\draw [thick] (-2.8,-0.2) -- (-1,-2);
\draw [thick] (0,-2) -- (-1,-2);
\draw [thick] (-3.8,-0.8) -- (-2.2,0.8);
\draw [thick] (-3.8,1) -- (0.3,1);
\draw [thick] (-1.8,-1) -- (0.3,-1);
\draw [thick] (-4.5,-1.5) -- (-4.2,-1.2);
\draw [thick] (-4.5,-1.5) -- (-4.5,-2.5);
\draw [thick] (0,-2.5) -- (-4.5,-2.5);
\draw [thick] (-4.5,1.5) -- (-4.5,2.5);
\draw [thick] (0,2.5) -- (-4.5,2.5);
\draw [thick] (-1.8,1.2) -- (-1,2);
\draw [thick] (0,2) -- (-1,2);
%middle
\draw [thick] (0.5,1.5) -- (0.5,-1.5);
\draw [thick] (0.5,1.5) -- (-0.5,1.5);
\draw [thick] (0.5,-1.5) -- (-0.5,-1.5);
\draw [thick] (-0.5,1.5) -- (-0.5,1.2);
\draw [thick] (-0.5,-1.5) -- (-0.5,-1.2);
\draw [thick] (-0.5,-0.8) -- (-0.5,0.8);
%right
\draw [thick] (5,-1) -- (5,1);
\draw [thick] (5,-1) -- (2.2,-1);
\draw [thick] (4.2,1) -- (5,1);
\draw [thick] (4.5,1.5) -- (3.2,0.2);
\draw [thick] (2.8,-0.2) -- (1,-2);
\draw [thick] (0,-2) -- (1,-2);
\draw [thick] (3.8,-0.8) -- (2.2,0.8);
\draw [thick] (3.8,1) -- (0.7,1);
\draw [thick] (1.8,-1) -- (0.7,-1);
\draw [thick] (4.5,-1.5) -- (4.2,-1.2);
\draw [thick] (4.5,-1.5) -- (4.5,-2.5);
\draw [thick] (0,-2.5) -- (4.5,-2.5);
\draw [thick] (4.5,1.5) -- (4.5,2.5);
\draw [thick] (0,2.5) -- (4.5,2.5);
\draw [thick] (1.8,1.2) -- (1,2);
\draw [thick] (0,2) -- (1,2);
\node at (-5.3,0) {$a$};
\node at (-4.8,-2) {$b$};
\node at (-3.6,-0.3) {$c$};
\node at (-4.8,2) {$d$};
\node at (-1.6,-1.8) {$e$};
\node at (-1.3,-0.7) {$f$};
\node at (-3,1.2) {$g$};
\node at (-1.6,1.8) {$h$};
\node at (-0.7,0) {$i$};
\node at (0.7,0) {$j$};
\node at (1.3,-0.7) {$k$};
\node at (3,1.2) {$m$};
\node at (3.6,-0.3) {$n$};
\node at (5.3,0) {$p$};
\node at (-0.3,0.8) {$*$};
\node at (4.1,-0.8) {$*$};
\node at (4.1,0.8) {$*$};
\node at (2.7,0) {$*$};
\end{tikzpicture}
\caption{The link $L$.}
\label{lfig}
\end{figure}

The four crossings marked with $*$ then yield the following relations.
\begin{align*}
  &  0=2s_E(g)-s_E(i)-s_E(j) = -16s_E(a)+8s_E(b)+8s_E(d) \\
  &  0=2s_E(n)-s_E(d)-s_E(e) = 8s_E(a)-16s_E(d)+8s_E(j) \\
  &  0=2s_E(d)-s_E(m)-s_E(p) = -16s_E(a)+8s_E(b)+8s_E(d) \\
  &  0=2s_E(p)-s_E(n)-s_E(b) = 24s_E(a)-16s_E(b)-8s_E(j) 
\end{align*}

The last two relations are redundant: the third relation is the same as the first, and the fourth relation is $-2$ times the first relation, minus the second relation. We conclude that 
\[
M_A(L)_ \nu \cong \mathbb Z \oplus \mathbb Z \oplus \mathbb Z_8 \oplus \mathbb Z_8 \text ,
\]
with the four direct summands generated by $s_E(a)$, $x=s_E(b)-s_E(a)$, $y=s_E(b)+s_E(d)-2s_E(a)$ and $z=s_E(a)+s_E(j)-2s_E(d)$. 

\begin{proposition}
There is no way to index the components of $L$ so that the resulting map $\phi_{\nu}:M_A(L)_{\nu} \to A_4 = \mathbb{Z} \oplus \mathbb{Z}_2 \oplus \mathbb{Z}_2 \oplus \mathbb{Z}_2$ sends two elements of finite order to two elements of the set $\{ (1,0,0,0),(0,1,0,0),(0,0,1,0),(0,0,0,1) \}$.
\end{proposition}
\begin{proof} The proposition is verified directly, by checking each of the $24$ ways to index the components of $L$. We give details for four of the 24.

Suppose we index the components of $L$ so that $K_1$, $K_2$, $K_3$ and $K_4$ correspond to the arcs $a$, $b$, $d$ and $j$, respectively. Then according to the definition given in Sec.\ \ref{2red}, the map $\phi_{\nu}:M_A(L)_{\nu} \to A_4 = \mathbb{Z} \oplus \mathbb{Z}_2 \oplus \mathbb{Z}_2 \oplus \mathbb{Z}_2$ has $\phi_{\nu}(s_E(a))=(1,0,0,0)$, $\phi_{\nu}(x)=(0,1,0,0)$, $\phi_{\nu}(y)=(0,1,1,0)$ and $\phi_{\nu}(z)=(0,0,0,1)$. Every element of $M_A(L)_ \nu$ of finite order is $uy+vz$ for some integers $u,v$. The image of such an element under $\phi_ \nu$ is $(0,\overline u, \overline u, \overline v)$, where the overline indicates reduction modulo 2. This image cannot equal $(1,0,0,0)$, $(0,1,0,0)$ or $(0,0,1,0)$.

Similarly, if $K_1$, $K_2$, $K_3$ and $K_4$ correspond respectively to $j$, $b$, $a$ and $d$, then $\phi_{\nu}(s_E(a))=(1,0,1,0)$, $\phi_{\nu}(x)=(0,1,1,0)$, $\phi_{\nu}(y)=(0,1,0,1)$ and $\phi_{\nu}(z)=(0,0,1,0)$. Therefore a finite-order element $uy+vz$ has $\phi_{\nu}(uy+vz)=(0,\overline u, \overline v, \overline u)$. This cannot equal $(1,0,0,0)$, $(0,1,0,0)$ or $(0,0,0,1)$.

If $K_1$, $K_2$, $K_3$ and $K_4$ correspond respectively to $b$, $d$, $j$, and $a$, then $\phi_{\nu}(uy+vz)=(0,\overline u, \overline v, \overline v)$, which cannot equal $(1,0,0,0)$, $(0,0,1,0)$ or $(0,0,0,1)$. If $K_1$, $K_2$, $K_3$ and $K_4$ correspond respectively to $d$, $j$, $a$ and $b$, then $\phi_{\nu}(uy+vz)=(0,\overline v, \overline v, \overline u)$, which cannot equal $(1,0,0,0)$, $(0,1,0,0)$ or $(0,0,1,0)$. \end{proof}

According to Proposition \ref{kertchar}, it follows that $Q_A(L)_\nu \not \cong \Core'(\mathbb Z \oplus \mathbb Z_8 \oplus \mathbb Z _8)$.

\subsection{The link $(T_{(2,8)} \# T_{(2,8)}) \# T_{(2,0)}$}
\label{tlink2}
We use the link diagram $D$ pictured in Fig.\ \ref{tfig} to describe the group $M_A(T)_ \nu$, where $T$ is the connected sum of torus links $(T_{(2,8)} \# T_{(2,8)}) \# T_{(2,0)}$.
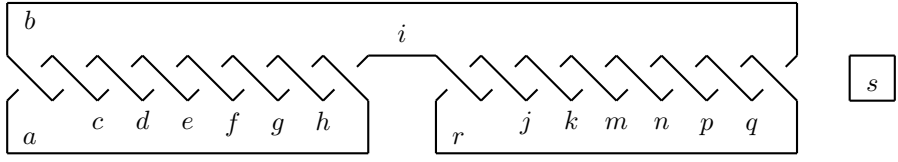
\begin{figure} [bht]
\centering
\begin{tikzpicture}
%left
\draw [thick] (-6,0.3) -- (-5.4,-0.3);
\draw [thick] (-5.4,0.3) -- (-4.8,-0.3);
\draw [thick] (-4.8,0.3) -- (-4.2,-0.3);
\draw [thick] (-4.2,0.3) -- (-3.6,-0.3);
\draw [thick] (-3.6,0.3) -- (-3,-0.3);
\draw [thick] (-3,0.3) -- (-2.4,-0.3);
\draw [thick] (-2.4,0.3) -- (-1.8,-0.3);
\draw [thick] (-6,-0.3) -- (-5.85,-0.15);
\draw [thick] (-5.4,0.3) -- (-5.55,0.15);
\draw [thick] (-5.4,-0.3) -- (-5.25,-0.15);
\draw [thick] (-4.8,0.3) -- (-4.95,0.15);
\draw [thick] (-4.8,-0.3) -- (-4.65,-0.15);
\draw [thick] (-4.2,0.3) -- (-4.35,0.15);
\draw [thick] (-4.2,-0.3) -- (-4.05,-0.15);
\draw [thick] (-3.6,0.3) -- (-3.75,0.15);
\draw [thick] (-3.6,-0.3) -- (-3.45,-0.15);
\draw [thick] (-3,0.3) -- (-3.15,0.15);
\draw [thick] (-3,-0.3) -- (-2.85,-0.15);
\draw [thick] (-2.4,0.3) -- (-2.55,0.15);
\draw [thick] (-2.4,-0.3) -- (-2.25,-0.15);
\draw [thick] (-1.8,0.3) -- (-1.95,0.15);
\draw [thick] (-1.8,-0.3) -- (-1.65,-0.15);
\draw [thick] (-1.2,0.3) -- (-1.35,0.15);
\draw [thick] (-1.8,0.3) -- (-1.2,-0.3);
%centerandacross
\draw [thick] (-1.5+1.2,0.3) -- (-1.2,0.3);
\draw [thick] (-1.5+6,0.3) -- (-1.5+6,1);
\draw [thick] (-6,0.3) -- (-6,1);
\draw [thick] (-1.5+6,1) -- (-6,1);
\draw [thick] (-6,-0.3) -- (-6,-1);
\draw [thick] (-1.2,-1) -- (-6,-1);
\draw [thick] (-1.2,-1) -- (-1.2,-0.3);
\draw [thick] (-1.5+6,-0.3) -- (-1.5+6,-1);
\draw [thick] (-1.5+1.2,-1) -- (-1.5+6,-1);
\draw [thick] (-1.5+1.2,-1) -- (-1.5+1.2,-0.3);
\draw [thick] (5.2,-0.3) -- (5.8,-0.3);
\draw [thick] (5.2,-0.3) -- (5.2,0.3);
\draw [thick] (5.2,0.3) -- (5.8,0.3);
\draw [thick] (5.8,-0.3) -- (5.8,0.3);
%right
\draw [thick] (-1.5+6,-0.3) -- (-1.5+5.4,0.3);
\draw [thick] (-1.5+5.4,-0.3) -- (-1.5+4.8,0.3);
\draw [thick] (-1.5+4.8,-0.3) -- (-1.5+4.2,0.3);
\draw [thick] (-1.5+4.2,-0.3) -- (-1.5+3.6,0.3);
\draw [thick] (-1.5+3.6,-0.3) -- (-1.5+3,0.3);
\draw [thick] (-1.5+3,-0.3) -- (-1.5+2.4,0.3);
\draw [thick] (-1.5+2.4,-0.3) -- (-1.5+1.8,0.3);
\draw [thick] (-1.5+1.8,-0.3) -- (-1.5+1.2,0.3);
\draw [thick] (-1.5+6,0.3) -- (-1.5+5.85,0.15);
\draw [thick] (-1.5+5.4,-0.3) -- (-1.5+5.55,-0.15);
\draw [thick] (-1.5+5.4,0.3) -- (-1.5+5.25,0.15);
\draw [thick] (-1.5+4.8,-0.3) -- (-1.5+4.95,-0.15);
\draw [thick] (-1.5+4.8,0.3) -- (-1.5+4.65,0.15);
\draw [thick] (-1.5+4.2,-0.3) -- (-1.5+4.35,-0.15);
\draw [thick] (-1.5+4.2,0.3) -- (-1.5+4.05,0.15);
\draw [thick] (-1.5+3.6,-0.3) -- (-1.5+3.75,-0.15);
\draw [thick] (-1.5+3.6,0.3) -- (-1.5+3.45,0.15);
\draw [thick] (-1.5+3,-0.3) -- (-1.5+3.15,-0.15);
\draw [thick] (-1.5+3,0.3) -- (-1.5+2.85,0.15);
\draw [thick] (-1.5+2.4,-0.3) -- (-1.5+2.55,-0.15);
\draw [thick] (-1.5+2.4,0.3) -- (-1.5+2.25,0.15);
\draw [thick] (-1.5+1.8,-0.3) -- (-1.5+1.95,-0.15);
\draw [thick] (-1.5+1.8,0.3) -- (-1.5+1.65,0.15);
\draw [thick] (-1.5+1.2,-0.3) -- (-1.5+1.35,-0.15);
\node at (-5.7,-0.8) {$a$};
\node at (-5.7,0.77) {$b$};
\node at (-4.8,-0.6) {$c$};
\node at (-4.2,-0.55) {$d$};
\node at (-3.6,-0.6) {$e$};
\node at (-3,-0.62) {$f$};
\node at (-2.4,-0.65) {$g$};
\node at (-1.8,-0.57) {$h$};
\node at (-.75,0.6) {$i$};
\node at (-1.5+5.4,-0.65) {$q$};
\node at (-1.5+4.8,-0.65) {$p$};
\node at (-1.5+4.2,-0.6) {$n$};
\node at (-1.5+3.6,-0.6) {$m$};
\node at (-1.5+3,-0.55) {$k$};
\node at (-1.5+2.4,-0.6) {$j$};
\node at (-1.5+1.5,-0.8) {$r$};
\node at (5.5,-0.1) {$s$};
\end{tikzpicture}
\caption{The link $T=(T_{(2,8)} \# T_{(2,8)}) \# T_{(2,0)}$.}
\label{tfig}
\end{figure}

The generators $s_D(c),s_D(d),s_D(e),s_D(f),s_D(g)$ and $s_D(h)$ can be eliminated using relations from the six leftmost crossings: $s_D(c)=2s_D(b)-s_D(a)$, $s_D(d)=2s_D(c)-s_D(b)=3s_D(b)-2s_D(a)$,  $s_D(e)=2s_D(d)-s_D(c)=4s_D(b)-3s_D(a)$,  $s_D(f)=2s_D(e)-s_D(d)=5s_D(b)-4s_D(a)$,  $s_D(g)=2s_D(f)-s_D(e)=6s_D(b)-5s_D(a)$ and  $s_D(h)=2s_D(g)-s_D(f)=7s_D(b)-6s_D(a)$. Similarly, the generators $s_D(q),s_D(p),s_D(n),s_D(m),s_D(k)$ and $s_D(j)$ can be eliminated using relations from the six rightmost crossings: $s_D(q)=2s_D(r)-s_D(b)$, $s_D(p)=2s_D(q)-s_D(r)=3s_D(r)-2s_D(b)$,  $s_D(n)=2s_D(p)-s_D(q)=4s_D(r)-3s_D(b)$,  $s_D(m)=2s_D(n)-s_D(p)=5s_D(r)-4s_D(b)$,  $s_D(k)=2s_D(m)-s_D(n)=6s_D(r)-5s_D(b)$ and  $s_D(j)=2s_D(k)-s_D(m)=7s_D(r)-6s_D(b)$.

We are left with the generators $s_D(a),s_D(b), s_D(i), s_D(r)$ and $s_D(s)$. The four crossings in the middle provide the relations $0=2s_D(h)-s_D(g)-s_D(a)=8s_D(b)-8s_D(a)$,
$s_D(i)=2s_D(a)-s_D(h)=8s_D(a)-7s_D(b)$, $0=2s_D(i)-s_D(j)-s_D(r) = 16s_D(a)-8s_D(b)-8s_D(r)$ and $0=2s_D(j)-s_D(i)-s_D(k)=8s_D(r)-8s_D(a)$.

We conclude that $M_A(T)_{\nu}$ is generated by $s_D(a),s_D(b), s_D(r)$ and $s_D(s)$, subject to two relations: $8(s_D(b)-s_D(a))=0$ and $8(s_D(r)-s_D(a))=0$. Therefore
\[
M_A(T)_{\nu} \cong \mathbb Z \oplus \mathbb Z _8 \oplus \mathbb Z _8 \oplus \mathbb Z \text ,
\]
with the four direct summands generated (in order) by $s_D(a),s_D(b)-s_D(a),$ $s_D(r)-s_D(a)$ and $s_D(s)-s_D(a)$. If the components of $T$ corresponding to $a,b,r$ and $s$ are indexed as $K_1$, $K_2$, $K_3$ and $K_4$ respectively, then $\phi_ \nu (s_D(a))=(1,0,0,0)$, $\phi_ \nu (s_D(b)-s_D(a))=(0,1,0,0)$, $\phi_ \nu (s_D(r)-s_D(a))=(0,0,1,0)$ and $\phi_ \nu (s_D(s)-s_D(a))=(0,0,0,1)$. 

According to Proposition \ref{kertchar}, it follows that $Q_A(T)_\nu \cong \Core'(\mathbb Z _8 \oplus \mathbb Z _8 \oplus \mathbb Z)$. Comparing this result with the discussion of the link $L$ in Sec.\ \ref{llink}, we 
see that both links have $\ker w_\nu \cong \mathbb Z _8 \oplus \mathbb Z _8 \oplus \mathbb Z$, but $Q_A(L)_\nu \not \cong Q_A(T)_\nu$.

\subsection{Two split links}

We close with another pair of examples, $L'$ and $L''$, to illustrate that $4 \centernot \implies 3$ in Theorem \ref{main2}. These two links are both split, so they may seem less interesting than the links $L$ and $T$ considered above. We mention them because it is easy to see that $L'$ and $L''$ are mutants. According to Viro \cite{V}, it follows that if $X'_2$ and $X''_2$ are the cyclic double covers of $\mathbb S ^3$ branched over $L'$ and $L''$, then $X'_2 \cong X''_2$. We deduce that in general, for a link $L$ the branched double cover does not determine the quandle $Q_A(L)_\nu$; of course, it follows that the branched double cover does not determine $\IMQ(L)$ either.

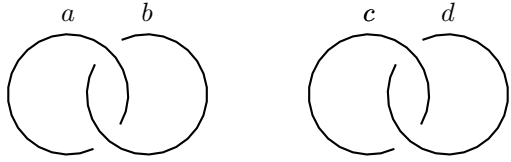
\begin{figure} [bht]
\centering
\begin{tikzpicture} 
\draw [thick, domain=-30:295] plot ({-5+(0.8)*cos(\x)}, {(0.8)*sin(\x)});
\draw [thick, domain=150:475] plot ({-3.95+(0.8)*cos(\x)}, {(0.8)*sin(\x)});
\draw [thick, domain=-30:295] plot ({4-5+(0.8)*cos(\x)}, {(0.8)*sin(\x)});
\draw [thick, domain=150:475] plot ({4-3.95+(0.8)*cos(\x)}, {(0.8)*sin(\x)});
\node at (-5,1.05) {$a$};
\node at (-3.95,1.1) {$b$};
\node at (-1,1.05) {$c$};\node at (-1,1.05) {$c$};
\node at (0.05,1.1) {$d$};
\end{tikzpicture}
\caption{$L'$ consists of two copies of the Hopf link.}
\label{hhfig}
\end{figure}

If $L'$ is the link with the diagram $D'$ illustrated in Fig.\ \ref{hhfig}, then $M_A(L')_\nu$ is generated by the four elements $s_{D'}(a), s_{D'}(b), s_{D'}(c)$ and $s_{D'}(d)$. The crossing relations are $2s_{D'}(a)=2s_{D'}(b)$ and $2s_{D'}(c)=2s_{D'}(d)$. It follows that 
\begin{equation}
\label{amods4}
M_A(L')_\nu \cong \mathbb Z \oplus \mathbb Z_2 \oplus
\mathbb Z \oplus \mathbb Z_2 \textup,
\end{equation}
with the four summands generated by $s_{D'}(a), s_{D'}(b)-s_{D'}(a), s_{D'}(c)-s_{D'}(a)$ and $s_{D'}(d)-s_{D'}(a)$, respectively.

Now, let $L''$ be the link with the diagram $D''$ illustrated in Fig.\ \ref{hhufig}. Then $M_A(L'')_\nu$ is generated by $s_{D''}(w),s_{D''}(x),s_{D''}(x'),s_{D''}(y)$ and $s_{D''}(z)$. The two crossings on the left tell us that $2s_{D''}(w)=2s_{D''}(x')=s_{D''}(x)+s_{D''}(x')$, so $s_{D''}(x')=s_{D''}(x)$ and $2s_{D''}(w)=2s_{D''}(x)$. Taking $s_{D''}(x')=s_{D''}(x)$ into account, the two crossings on the right tell us that $2s_{D''}(x)=2s_{D''}(y)$. Therefore
\begin{equation}
\label{amods5}
M_A(L'')_\nu \cong \mathbb Z \oplus \mathbb Z_2 \oplus \mathbb Z_2 \oplus \mathbb Z  \textup,
\end{equation}
with the four summands generated by $s_{D''}(w),s_{D''}(x)-s_{D''}(w),s_{D''}(y)-s_{D''}(w)$ and $s_{D''}(z)-s_{D''}(w)$, respectively. 

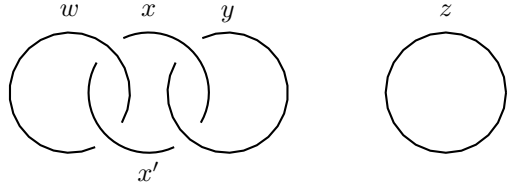
\begin{figure} [bht]
\centering
\begin{tikzpicture} 
\draw [thick, domain=-30:295] plot ({-5+(0.8)*cos(\x)}, {(0.8)*sin(\x)});
\draw [thick, domain=150:295] plot ({-3.95+(0.8)*cos(\x)}, {(0.8)*sin(\x)});
\draw [thick, domain=-30:115]  plot ({-3.95+(0.8)*cos(\x)}, {(0.8)*sin(\x)});
\draw [thick, domain=150:475] plot ({-2.9+(0.8)*cos(\x)}, {(0.8)*sin(\x)});
\draw [thick, domain=0:360] plot ({(0.8)*cos(\x)}, {(0.8)*sin(\x)});
\node at (-5,1.1) {$w$};
\node at (-3.95,1.1) {$x$};
\node at (-2.9,1.05) {$y$};
\node at (-3.95,-1.05) {$x'$};
\node at (0,1.1) {$z$};
\end{tikzpicture}
\caption{The link $L''$.}
\label{hhufig}
\end{figure}

It is easy to see that $M_A(L')_\nu \cong M_A(L'')_\nu$. It is only a little bit harder to see that $L'$ and $L''$ are not $\phi_\nu$-equivalent with respect to any order of their components.

\begin{proposition}
\label{nophinu}
No matter how their components are indexed, $L'$ and $L''$ are not $\phi_\nu$-equivalent.
\end{proposition}
\begin{proof}
According to (\ref{amods4}), $M_A(L')_\nu$ has three nonzero elements of finite order: $s_{D'}(b)-s_{D'}(a),s_{D'}(d)-s_{D'}(c)$ and their difference, $s_{D'}(b)-s_{D'}(a)+s_{D'}(d)-s_{D'}(c)$. The map $\phi_{\nu}:M_A(L')_\nu \to A_4 = \mathbb Z \oplus \mathbb Z_2 \oplus \mathbb Z_2 \oplus \mathbb Z_2$ sends $s_{D'}(a),s_{D'}(b)$, $s_{D'}(c)$ and $s_{D'}(d)$ to $(1,0,0,0),(1,1,0,0),(1,0,1,0)$ and $(1,0,0,1)$, in some order. No matter what order is used, the image of $s_{D'}(b)-s_{D'}(a)+s_{D'}(d)-s_{D'}(c)$ will be $(0,1,1,1)$.

According to (\ref{amods5}), $M_A(L'')_\nu$ also has three nonzero elements of finite order: $s_{D''}(x)-s_{D''}(w),s_{D''}(y)-s_{D''}(w)$ and their difference, $s_{D''}(x)-s_{D''}(y)$. The map $\phi_{\nu}:M_A(L'')_\nu \to A_4$ sends $s_{D''}(w),s_{D''}(x),s_{D''}(y)$ and $s_{D''}(z)$, in some order, to $(1,0,0,0),(1,1,0,0),(1,0,1,0)$ and $(1,0,0,1)$. No matter what order is used, $(0,1,1,1)$ will not be the image of an element of finite order under $\phi_\nu$.

Of course, every isomorphism $f:M_A(L')_\nu \to M_A(L'')_\nu$ has the property that whenever $m$ is of finite order, so is $f(m)$. It follows that no such isomorphism is compatible with the $\phi_\nu$ maps of $L'$ and $L''$. \end{proof}


\begin{thebibliography}{99}

\bibitem{C1} R. H. Crowell, Corresponding link and module sequences, \emph{Nagoya Math. J.} \textbf{19} (1961) 27-40.

\bibitem{C3} R. H. Crowell, The derived module of a homomorphism, \emph{Adv. in Math.} \textbf{6} (1971) 210-238.

\bibitem {EN} M. Elhamdadi and S. Nelson, \emph{Quandles}, Student Mathematical Library, Vol. 74 (Amer. Math. Soc., Providence, R.I., 2015).

\bibitem {F} R. H. Fox, A quick trip through knot theory, in \emph{Topology of 3-Manifolds and Related Topics} (Proc. The Univ. of Georgia Institute, 1961) (Prentice-Hall, Englewood Cliffs, N.J., 1962), pp. 120-167.

\bibitem{F1} R. H. Fox, Free differential calculus. I, \emph{Ann. Math.} \textbf{57} (1953) 547 - 560. 

\bibitem {H} J. A. Hillman, \emph{Algebraic Invariants of Links}, 2nd edn. Series on Knots and Everything, Vol. 52 (World Scientific, Singapore, 2012).

\bibitem {JPSZ1} P. Jedli\v{c}ka, A. Pilitowska, D. Stanovsk\'{y} and A. Zamojska-Dzienio, The structure of medial quandles, \emph{J. Algebra} \textbf{443} (2015) 300-334.

\bibitem {JPSZ2} P. Jedli\v{c}ka, A. Pilitowska, D. Stanovsk\'{y} and A. Zamojska-Dzienio, Subquandles of affine quandles, \emph{J. Algebra} \textbf{510} (2018) 259-288.

\bibitem {J} D. Joyce, A classifying invariant of knots, the knot quandle, \emph{J. Pure Appl. Algebra} \textbf{23} (1982) 37-65.

\bibitem {L}W. B. R. Lickorish, \emph{An Introduction to Knot Theory}, Graduate Texts
in Mathematics, Vol. 175 (Springer, New York, 1997).

\bibitem {M} S. V. Matveev, Distributive groupoids in knot theory, \emph{Mat. Sb. (N.S.)} \textbf{119} (1982) 78-88.

\bibitem{Mi} K. Miller, personal correspondence, May 2020.

\bibitem{STW} D. S. Silver, L. Traldi and S. G. Williams, Goeritz and Seifert matrices from Dehn presentations,  \emph{Osaka J. Math.} \textbf{57} (2020) 663-677.

\bibitem{mvaq3} L. Traldi, Multivariate Alexander quandles, III. Sublinks, \emph{J. Knot Theory Ramifications} \textbf{28} (2019), article 1950090.

\bibitem{mvaq1} L. Traldi, Multivariate Alexander quandles, I. The module sequence of a link, \emph{J. Knot Theory Ramifications} \textbf{29} (2020), article 2050009.

\bibitem{mvaq2} L. Traldi, Multivariate Alexander quandles, II. The involutory medial quandle of a link, \emph{J. Knot Theory Ramifications} \textbf{29} (2020), article 2050024. Erratum: \emph{ibid.}, article 2092001.

\bibitem{V} O. Ja. Viro, Nonprojecting isotopies and knots with homeomorphic coverings, \emph{Zap. Naučn. Sem. Leningrad. Otdel. Mat. Inst. Steklov. (LOMI)} \textbf{66} (1976), 133–147, 207–208. 

\bibitem{W} S. K. Winker, Quandles, knot invariants and the $n$-fold branched cover, Ph. D. Thesis, Univ. of Illinois at Chicago, 1984.

\end{thebibliography}
\end{document}